\documentclass[11pt,letterpaper]{amsart}
\usepackage{amsthm,amsfonts,amsmath,amssymb,latexsym}
\usepackage{epsfig,graphics,color}
\usepackage[dvipsnames]{xcolor}
\usepackage[all]{xy}
\usepackage[breaklinks=true]{hyperref}
\usepackage{mathrsfs}
\usepackage{stmaryrd}
\usepackage{verbatim}
\usepackage{bm}
\usepackage{mathabx}
\usepackage{enumitem}
\usepackage{pgf,amsmath,tikz,type1cm,fix-cm}
\usetikzlibrary{arrows.meta,bending,positioning}
 
\newlength{\XHeight}
\newlength{\XWidth}
\setlist[itemize,1]{leftmargin=\dimexpr 26pt-.1in}

\newtheorem{PARA}{}[section]
\newtheorem{theorem}[PARA]{Theorem}
\newtheorem{corollary}[PARA]{Corollary}
\newtheorem{lemma}[PARA]{Lemma}
\newtheorem{proposition}[PARA]{Proposition}
\newtheorem{definition}[PARA]{Definition}
\newtheorem{definition-proposition}[PARA]{Definition-Proposition}

\theoremstyle{definition}
\newtheorem{remark}[PARA]{Remark}
\theoremstyle{theorem}
\newtheorem{example}[PARA]{Example}

\newcommand{\para}{\begin{PARA}\rm}
\newcommand{\arap}{\end{PARA}\rm}
\newcommand{\dfn}{\begin{definition}\rm}
\newcommand{\nfd}{\end{definition}\rm}
\newcommand{\rmk}{\begin{remark}\rm}
\newcommand{\kmr}{\end{remark}\rm}
\newcommand{\xmpl}{\begin{example}\rm}
\newcommand{\lpmx}{\end{example}\rm}

\newcommand{\cH}{\mathcal{H}}
\newcommand{\cI}{\mathcal{I}}

\newcommand{\cL}{\mathcal{L}}
\newcommand{\cM}{\mathcal{M}}

\newcommand{\cN}{\mathcal{N}}

\newcommand\co{
  \mathchoice
    {{\scriptstyle\mathcal{O}}}
    {{\scriptstyle\mathcal{O}}}
    {{\scriptscriptstyle\mathcal{O}}}
    {\scalebox{.7}{$\scriptscriptstyle\mathcal{O}$}}
  }
\newcommand{\cP}{\mathcal{P}}

\newcommand{\C}{{\mathbb{C}}}

\renewcommand{\H}{{\mathbb{H}}}

\newcommand{\N}{{\mathbb{N}}}
\newcommand{\Q}{{\mathbb{Q}}}
\newcommand{\R}{{\mathbb{R}}}

\newcommand{\Z}{{\mathbb{Z}}}
\newcommand{\coker}{\mathrm{ coker }}  
\newcommand{\Cr}{\mathrm{ Cr}\, }  
\newcommand{\im}{\mathrm{im}\,}        

\newcommand{\ind}{\mathrm{ind}}

\newcommand{\ev}{\mathrm{ev}}

\newcommand{\Hom}{\mathrm{Hom}}

\newcommand{\bk}{\mathbf{k}}
\newcommand{\Top}{\mathrm{Top}}
\DeclareMathOperator{\hatotimes}{{\hat{\otimes}}}
\newcommand{\eps}{{\varepsilon}}
\newcommand{\om}{{\omega}}
\newcommand{\Om}{{\Omega}}

\newcommand{\CZ}{\mathrm{CZ}}
\def\NABLA#1{{\mathop{\nabla\kern-.5ex\lower1ex\hbox{$#1$}}}}
\def\Nabla#1{\nabla\kern-.5ex{}_{#1}}
\def\Tabla#1{\Tilde\nabla\kern-.5ex{}_{#1}}
\renewcommand{\Tilde}{\widetilde}

\newcommand{\p}{{\partial}}

\newenvironment{enum}
{\begin{enumerate}}
{\end{enumerate}}

\newcommand{\la}{\langle}
\newcommand{\ra}{\rangle}
\newcommand{\wh}{\widehat}
\newcommand{\ol}{\overline}

\newcommand{\Crit}{{\rm Crit}}
\newcommand{\wt}{\widetilde}

\newcommand{\into}{\hookrightarrow}

\newcommand*\Hamlambdamuaxis[2]{\ensuremath{
       \settowidth{\XWidth}{$\lambda$}
       \addtolength{\XWidth}{.6em}
       \setlength{\XHeight}{\XWidth}
       \tikz[baseline]{
         \draw[thick] (-\XWidth,2\XHeight) -- (0,2\XHeight) -- (2\XWidth,0) -- (3\XWidth,\XHeight); 
         \draw[arrows={[sep] - Stealth[length=3pt]}] (-\XWidth,-.5\XHeight) -- (-\XWidth,3\XHeight);
         \draw[arrows={[sep] - Stealth[length=3pt]}] (-1.5\XWidth,0) -- (3.5\XWidth,0);
         \node at (0.6\XWidth,0.9\XHeight)[scale=0.5] {#1};
         \node at (2.9\XWidth,.5\XHeight)[scale=0.5] {#2};
         \node at (2\XWidth,2\XHeight)[scale=0.75] {$H$};
       }
     }
}

\newcommand*\HamLlambdaaxis[1]{\ensuremath{
       \settowidth{\XWidth}{$\lambda$}
       \addtolength{\XWidth}{.6em}
       \setlength{\XHeight}{\XWidth}
       \tikz[baseline]{
         \draw[thick] (-\XWidth,2\XHeight) -- (0,2\XHeight) -- (3\XWidth,-\XHeight);
         \draw[arrows={[sep] - Stealth[length=3pt]}] (-\XWidth,-.5\XHeight) -- (-\XWidth,3\XHeight);
         \draw[arrows={[sep] - Stealth[length=3pt]}] (-1.5\XWidth,0) -- (3.5\XWidth,0);
         \node at (0.6\XWidth,0.9\XHeight)[scale=0.5] {#1};
         \node at (2\XWidth,2\XHeight)[scale=0.75] {$L$};
       }
     }
}

\newcommand*\HamZaxis[3]{\ensuremath{
       \settowidth{\XWidth}{$\lambda$}
       \addtolength{\XWidth}{.6em}
       \setlength{\XHeight}{\XWidth}
       \tikz[baseline]{
         \draw[thick] (-\XWidth,2\XHeight) -- (0,2\XHeight) -- (1.5\XWidth,0.5\XHeight) -- (1.75\XWidth,-0.5\XHeight) -- (2\XWidth,0) -- (3\XWidth,-\XHeight);
         \draw[arrows={[sep] - Stealth[length=3pt]}] (-\XWidth,-.5\XHeight) -- (-\XWidth,3\XHeight);
         \draw[arrows={[sep] - Stealth[length=3pt]}] (-1.5\XWidth,0) -- (3.5\XWidth,0);
         \node at (0.6\XWidth,0.9\XHeight)[scale=0.5] {#1};
         \node at (1.95\XWidth,0.35\XHeight)[scale=0.5] {#2};
         \node at (2\XWidth,-0.5\XHeight)[scale=0.5] {#3};
         \node at (2\XWidth,2\XHeight)[scale=0.75] {$Z$};
	}
     }
}

\newcommand*\HamKaxis[3]{\ensuremath{
       \settowidth{\XWidth}{$\lambda$}
       \addtolength{\XWidth}{.6em}
       \setlength{\XHeight}{\XWidth}
       \tikz[baseline]{
         \draw[thick] (-\XWidth,2\XHeight) -- (0,2\XHeight) -- (1.5\XWidth,0.5\XHeight) -- (1.75\XWidth,-0.5\XHeight) 
         -- (3.25\XWidth,1.5\XHeight);
         \draw[arrows={[sep] - Stealth[length=3pt]}] (-\XWidth,-.5\XHeight) -- (-\XWidth,3\XHeight);
         \draw[arrows={[sep] - Stealth[length=3pt]}] (-1.5\XWidth,0) -- (3.5\XWidth,0);
         \node at (0.6\XWidth,0.9\XHeight)[scale=0.5] {#1};
         \node at (1.95\XWidth,0.35\XHeight)[scale=0.5] {#2};
         \node at (2.05\XWidth,-0.5\XHeight)[scale=0.5] {#3};
         \node at (2\XWidth,2\XHeight)[scale=0.75] {$K$};
       }
     }
}

\definecolor{darkgreen}{rgb}{0.12, 0.3, 0.17}
\definecolor{burntorange}{rgb}{0.8, 0.33, 0.0}
\definecolor{chromeyellow}{rgb}{1.0, 0.65, 0.0}
\definecolor{darkorange}{rgb}{1.0, 0.55, 0.0}
\definecolor{flame}{rgb}{0.89, 0.35, 0.13}
\usetikzlibrary{arrows}
\usetikzlibrary{shapes}

\newcommand{\boldmu}{{\boldsymbol{\mu}}}
\newcommand{\boldlambda}{{\boldsymbol{\lambda}}}
\newcommand{\boldeta}{{\boldsymbol{\eta}}}
\newcommand{\boldeps}{{\boldsymbol{\eps}}}
\newcommand{\boldc}{{\boldsymbol{c}}}
\newcommand{\boldp}{{\boldsymbol{p}}}

\newcommand{\boldzeta}{{\boldsymbol{\zeta}}}
\newcommand{\boldi}{{\boldsymbol{i}}}

\begin{document}

\title[Poincar\'e duality for loop spaces]{Poincar\'e duality for loop spaces}
\author{Kai Cieliebak}
\address{Universit\"at Augsburg \newline Universit\"atsstrasse 14, D-86159 Augsburg, Germany}
\email{kai.cieliebak@math.uni-augsburg.de}
\author{Nancy Hingston}
\address{Department of Mathematics and Statistics, College of New Jersey \newline Ewing, New
Jersey 08628, USA\newline
{\it Current address:} 4 Jeffrey Ln., Princeton Jct., NJ 08550
}
\email{hingston@tcnj.edu} 
\author{Alexandru Oancea}
\address{Universit\'e de Strasbourg \newline 
Institut de recherche math\'ematique avanc\'ee (IRMA) \newline
7 Rue Descartes, 67084 Strasbourg Cedex, France}
\email{oancea@unistra.fr}
\date{\today}


\begin{abstract}
We show that Rabinowitz Floer homology and cohomology carry the structure of a graded Frobenius algebra for both closed and open strings. We prove a Poincar\'e duality theorem between homology and cohomology that preserves this structure. This lifts to a duality theorem between graded open-closed TQFTs. We use in a systematic way the formalism of Tate vector spaces.

Specializing to the case of cotangent bundles, we define Rabinowitz loop homology and cohomology and explain from a unified perspective pairs of dual results that have been observed over the years in the context of the search for closed geodesics. These concern critical levels, relations to the based loop space, manifolds all of whose geodesics are closed, Bott index iteration, and level-potency. Moreover, the graded Frobenius algebra structure gives meaning and proof to a relation conjectured by Sullivan between the loop product and coproduct.
\end{abstract}

\maketitle



\parskip=0pt
\hspace{2cm}{\it One Ring to rule them all, One Ring to find them,}

\hspace{2cm}{\it One Ring to bring them all and in the darkness bind them}
\parskip=4pt
\bigskip

\section{Introduction}\label{sec:introduction}

It has been understood since the ground-breaking work of Floer, Hofer, Viterbo and others in the 1990s that the symplectic geometry of a cotangent bundle is intimately related to the topology of the loop space of the underlying manifold.
Building on classical results on loop spaces, this relation has provided an important source of examples in symplectic geometry and has inspired many developments on the symplectic side, see e.g.~\cite{Viterbo99,Fukaya-Lag,Cieliebak-Latschev,AS,AS2,Hingston_Conley_conjecture,Ginzburg_Conley_conjecture,Abouzaid2011a,Abouzaid2012a,Abouzaid-Kragh,Shelukhin19}.  
This paper tells a story about how some open questions concerning loop spaces lead to new results in symplectic geometry, which in turn serve to gain new insights on loop space topology. 

{\bf Puzzles in string topology. }
Loop spaces have played a prominent role in algebraic topology since
its early days. For example, the fact that a based loop space is an
H-space makes its homology a commutative Hopf algebra, a rigid
algebraic structure which can be completely classified under some
finiteness assumptions~\cite{Pontrjagin,Hopf-Gruppen}.
By contrast, the algebraic structure on the homology of a {\em free}
loop space still remains a mystery (cf.~\cite[\S2.4]{Sullivan-present-state}).
In the case of an $n$-dimensional manifold $M$, Chas and Sullivan in their seminal 1999
article~\cite{CS} describe a collection of operations on the
homology of its free loop space $\Lambda=\Lambda M$ which
became collectively known under the name {\em string topology}.  
On the $S^1$-equivariant homology $H_*^{S^1}(\Lambda,\Lambda_0)$ relative to the subspace $\Lambda_0\subset \Lambda$ of constant loops, the string bracket and cobracket induce the structure of an involutive Lie bialgebra~\cite{CS04}. The corresponding operations on non-equivariant homology are:
\begin{itemize}
\item the {\em loop product} $\mu=\bullet$ on $H_*\Lambda$ defined in~\cite{CS}, see also~\cite{Chataur,Laudenbach-CS,Irie,Hingston-Wahl-prod-coprod}. It is graded commutative, associative and unital of degree $-n$;
\item the {\em loop coproduct} $\lambda$ on $H_*(\Lambda,\Lambda_0)$ defined in~\cite{Sullivan-open-closed,Hingston-Wahl-prod-coprod}, which recently came to prominence in relation with simple homotopy invariance~\cite{Naef,Naef-Safronov,Kenigsberg-Porcelli,Wahl}, or dually the {\em cohomology product} $\oast$ on $H^*(\Lambda,\Lambda_0)$ studied in~\cite{Goresky-Hingston}. This is graded commutative and associative of degree $n-1$.\footnote{
The loop product and cohomology product are also known as the {\em Chas--Sullivan product} and {\em Goresky--Hingston product}, respectively. }
\end{itemize}
In subsequent studies of these operations
the following puzzles arose:

(a) Note that $\mu$ is defined on all of $\Lambda$, and $\lambda$ only relative to the constant loops. Can we remove this asymmetry and define both operations on a common domain?

(b) Heuristically, it looks as if ``the cohomology product is the loop product on Morse cochains on loop space''.
Can this statement be given precise mathematical sense?

(c) The cohomology product has the flavour of a ``secondary product'', e.g.~its degree is shifted by $1$ with respect to the loop product. Can $\oast$ indeed be constructed as a secondary product derived from $\bullet$?

(d) Assuming a positive answer to (a), what is the algebraic structure defined by $\mu$ and $\lambda$? In particular, do they satisfy the relation
\begin{equation}\label{eq:Sullivan}
   \lambda\mu = (1\otimes\mu)(\lambda\otimes 1) + (\mu\otimes 1)(1\otimes\lambda)
\end{equation}
stated by Sullivan~\cite{Sullivan-open-closed}? Note that this is {\em not} a TQFT type relation, but it rather resembles Drinfeld compatibility for Lie bialgebras. 

(e) Many results concerning $\bullet$ and $\oast$ arise in dual pairs. For example, the {\em critical levels} $\Cr(X)$ for $X\in H_*\Lambda$ and ${\rm cr}(x)$ for $x\in H^*(\Lambda,\Lambda_0)$ defined in~\cite{Goresky-Hingston} satisfy the dual inequalities
$$
   \Cr(X\bullet Y)\le \Cr(X) + \Cr(Y),\qquad
   {\rm cr}(x\oast y )\ge {\rm cr}(x) + {\rm cr}(y). 
$$
See \S\ref{sec:applications} for more details on this and other pairs of dual results concerning 
relations to the based loop space, 
manifolds all of whose geodesics are closed,
Bott index iteration, 
and level-potency.
Can each such dual pair be derived from one common result, with a unified proof, via some kind of ''Poincar\'e duality''?

It turns out that all the puzzles get resolved at once by introducing a new player into string topology.
This new player arises from symplectic geometry and our main results hold in the more general setting of Rabinowitz Floer homology, as we describe next. 

{\bf Poincar\'e duality for Rabinowitz Floer homology. }
Let $V$ be a {\em Liouville domain} of dimension $2n$, i.e., a compact connected manifold $V$ equipped with a $1$-form $\lambda$ such that $d\lambda$ is symplectic and $\lambda|_{\p V}$ is a positive contact form  (see~\cite{CO} for more background).
We assume that $2c_1(V)=0$ and the square of the canonical line bundle of $TV$ is trivialized, so that Floer chain complexes are canonically $\Z$-graded by the Conley-Zehnder indices of periodic orbits 
(see Appendix~\ref{sec:grading}).\footnote{
Everything remains true without this assumption if we replace $\Z$-gradings by $\Z/2\Z$-gradings.
}
We denote by\footnote{By default we use coefficients in a commutative unital ring $R$.}
$$
   RFH_*(\p V) = SH_*(\p V) 
$$
the {\em Rabinowitz Floer homology}~\cite{Cieliebak-Frauenfelder}, or in the terminology of~\cite{CO} the (V-shaped) symplectic homology, of $\p V$ with respect to the filling $V$. It was proved in~\cite{CO} that the pair-of-pants product $\boldmu$ (of degree $-n$) makes Rabinowitz Floer homology a $\Z$-graded unital associative and commutative algebra. 
Moreover, it is related to ordinary symplectic homology and cohomology by a long exact sequence 
\begin{equation}\label{eq:les-intro}
\xymatrix
@C=18pt
{
    \ar[r]& SH^{-*}(V) \ar[r]^-\eps & SH_*(V) \ar[r]^-\iota & RFH_*(V) \ar[r]^-\pi & SH^{1-*}(V)\ar[r]& 
}
\end{equation}
in which the map $\iota$ intertwines the pair-of-pants products.
Our first result is

\begin{theorem}[Poincar\'e duality for Rabinowitz Floer homology]\label{thm:PD-RFH-intro} 
{\em Rabinowitz Floer cohomology} $RFH^*(\p V)$ carries a canonical secondary pair-of-pants product $\boldlambda^\vee \tau$ 
of degree $n-1$,
\footnote{Given a graded module $A$ the twist $\tau:A\otimes A\to A\otimes A$ acts by $\tau(a\otimes b)=(-1)^{|a||b|}b\otimes a$. We write the product in the form $\boldlambda^\vee\tau$ in order to emphasize the operation $\boldlambda^\vee$, which is the dual of a coproduct.
The appearance of the twist $\tau$ reflects the fact that, when represented as $Y$-graphs with ends on a circle as in~\cite[\S 7]{CO-cones}, the inputs of a product are numbered in the clockwise order, whereas the outputs of a coproduct are numbered in the counterclockwise order. In Theorem~\ref{thm:PD-RFH-intro} the cocommutativity of $\boldlambda$ implies the simple relation $\boldlambda^\vee\tau=-\boldlambda^\vee$, which is no longer true with Lagrangian boundary conditions as in Theorem~\ref{thm:TQFT-RFH-intro}.}
and there is a canonical Poincar\'e duality isomorphism of unital algebras
$$
   PD: (RFH_*(\p V),\boldmu)\stackrel\simeq\longrightarrow (RFH^{1-*}(\p V),\boldlambda^\vee \tau).  
$$
\end{theorem}

This isomorphism should be understood as generalizing Poincar\'e duality for the (canonically oriented) closed manifold $\p V$. The latter is realized in the zero energy sector, cf.~\S\ref{sec:length-filtration}. 

{\bf Graded Frobenius algebra structure. } 
It turns out that the algebraic structure is in fact much richer. To state this, we introduce the degree shifted versions
$$
   RF\H_*(\p V) = RFH_{*+n}(\p V),\qquad RF\H^*(\p V) = RFH^{*+n}(\p V).
$$
We denote the coproducts dual to $\boldmu,\boldlambda^\vee$ 
by $\boldmu^\vee,\boldlambda$.\footnote{Whenever we deal with coproducts and/or Tate vector spaces, we use coefficients in a field $\bk$. See also the discussion on the K\"unneth isomorphism at the beginning of~\S\ref{sec:PD-coproducts}.} 

\begin{theorem}[Graded Frobenius algebra structure on Rabinowitz Floer homology]\label{thm:Frob-RFH-intro}
Degree shifted Rabinowitz Floer homology and cohomology with field coefficients are commutative cocommutative graded Frobenius algebras, and Poincar\'e duality yields an isomorphism of such (bi)algebras
$$
   PD: (RF\H_*(\p V),\boldmu,\boldlambda)
   \stackrel\simeq\longrightarrow 
   (RF{\H}^{1-2n-*}(\p V),\boldlambda^\vee \tau, \tau\boldmu^\vee).
$$
\end{theorem}
   
Here a {\em graded Frobenius algebra} is a graded Tate vector space $A$ endowed with an associative degree zero product $\boldmu$ with unit $\boldeta$, and a coassociative coproduct $\boldlambda$ with counit $\boldeps$, such that the pairing $\boldp=(-1)^{|\boldlambda|}\boldeps\mu$ is symmetric and induces an isomorphism $\vec\boldp:A\to A^\vee$.
{\em Tate vector spaces} are a suitable class of topological vector spaces which accommodate the fact that, in contrast to ordinary Frobenius algebras, Rabinowitz Floer homology is often infinite dimensional.
See~\S\S\ref{sec:Tate}--\ref{sec:coFrob} for the precise definitions and further discussion. The Poincar\'e duality isomorphism is actually given by $-\vec\boldp$, so in this context it appears naturally as part of the algebraic structure. 

The appearance of such a structure on Rabinowitz Floer homology came as a surprise to us. For example, as pointed out above, Sullivan's relation~\eqref{eq:Sullivan} is {\em not} a relation in a graded Frobenius algebra. We will discuss below how it emanates (in modified form) from Theorem~\ref{thm:Frob-RFH-intro} when passing to reduced loop homology. 


{\bf Graded open-closed TQFT structure. }
The preceding results have analogues with Lagrangian boundary conditions. For this, let $(V,\lambda)$ be a $2n$-dimensional Liouville domain and $L\subset V$ an {\em exact Lagrangian submanifold with Legendrian boundary}, i.e., a compact $n$-dimensional submanifold $L\subset V$ which is conical (with respect to the Liouville vector field) near its boundary $\p L\subset\p V$ such that $\lambda|_L$ is exact and $\lambda|_{\p L}=0$ (see again~\cite{CO}).
We assume that $L$ is oriented, Spin and satisfies $2c_1(V,L)=0$, and we endow it with a Spin structure and a grading in the sense of Seidel~\cite{Seidel-graded} with respect to the fixed grading on $V$. 
This ensures that Floer moduli spaces carry a coherent orientation, and chain complexes relative to $L$ are canonically $\Z$-graded by the Conley-Zehnder indices of Hamiltonian chords. We refer to Seidel~\cite[11i-j]{Seidel-book} and Chen-Zinger~\cite{Chen-Zinger} for a discussion of orientations of moduli spaces of Floer discs in the context of Spin and Pin structures, and to Abouzaid~\cite{Abouzaid2012a} for a discussion of how orientations carry over to the situation where Floer homology is twisted by a local system that is transgressed from the base (``background class''), e.g. for cotangent bundles the local system $\sigma$ determined by the second Stiefel-Whitney class, which is used in this paper and discussed later on. See also Appendix~\ref{sec:grading} and footnote 2.

We denote by
$$
   RFH_*(\p L) = SH_*(\p L),\qquad RF\H_*(\p L) = RFH_{*+n}(\p L) 
$$
the {\em Rabinowitz Floer homology}, or (V-shaped) symplectic homology, of $\p L$ with respect to the filling $L$ defined in~\cite{CO}, and its degree shifted version.  

Our next theorem states the existence of a graded open-closed TQFT structure on Rabinowitz Floer homology. A {\em graded open-closed TQFT} consists of two graded Frobenius algebras $(C,\boldmu_C,\boldlambda_C)$ (the closed sector) and $(A,\boldmu_A,\boldlambda_A)$ (the open sector) together with a closed-open map $\boldzeta:C\to A$ and an open-closed map $\boldzeta^*:A\to C$ satisfying suitable relations stated in~\S\ref{ss:graded-open-closed-TQFT}. If the dimensions are finite and all operations have even degrees this coincides with the description of an open-closed TQFT by Lauda and Pfeiffer~\cite{Lauda-Pfeiffer}. Rabinowitz Floer homology differs from this classical setup in two distinct ways: it is typically infinite dimensional, and it is a graded theory with product and coproduct of opposite parities in the closed sector. To deal with the first aspect we introduce the formalism of Tate vector spaces, and to deal with the second aspect we keep track of nontrivial signs in the relations.

\begin{theorem}[Graded open-closed TQFT structure on Rabinowitz Floer homology]\label{thm:TQFT-RFH-intro}
The graded Frobenius algebra structure on $RF\H_*(\p V)$ with field coefficients canonically extends to a graded open-closed TQFT structure on the pair 
$$
(RF\H_*(\p V),RF\H_*(\p L)),
$$ 
with coproducts of degrees $|\boldlambda_C|=1-2n$ and $|\boldlambda_A|=1-n$. The Poincar\'e duality isomorphisms intertwine this structure with the corresponding structure on cohomology $$
(RF{\H}^{1-2n-*}(\p V),RF{\H}^{1-n-*}(\p L)).
$$ 
\end{theorem}


To our knowledge, the notion of a graded open-closed TQFT has not been considered before in the literature. This is somewhat surprising because, as explained in~\cite{CHO-algebra}, this structure appears naturally on the cohomology rings $(H^*P,H^*Q)$ of a closed oriented manifold $P$ with closed oriented submanifold $Q\subset P$. The graded open-closed TQFT in Theorem~\ref{thm:TQFT-RFH-intro} extends the one on the manifold pair $\p L\subset\p V$ under mild topological assumptions  discussed in~\cite{CHO-algebra}. 

TQFT-like structures have appeared earlier in symplectic topology:\break M.~Schwarz has constructed a Frobenius algebra structure on the Floer homology of a closed symplectic manifold~\cite{Schwarz95}, whereas P.~Seidel~\cite{Seidel07} and A.~Ritter~\cite{Ritter} have constructed a noncompact open-closed TQFT structure on the symplectic homology and wrapped Floer homology of a Liouville domain and an exact Lagrangian submanifold with Legendrian boundary. Note that the first case is finite dimensional, whereas the infinite dimensional second case carries only part of the TQFT operations. Theorem~\ref{thm:TQFT-RFH-intro} provides a large class of infinite dimensional examples carrying full (graded) open-closed TQFT structures. We refer to the paper of Moore-Segal~\cite{Moore-Segal} for a comprehensive list of references on open-closed TQFT structures outside of symplectic topology. In particular, we consider it interesting to study our structure in relation with the work of  Costello~\cite{Costello}, Godin~\cite{Godin} and Wahl-Westerland~\cite{Wahl-Westerland}. 

{\bf Application to loop spaces. }
Let now $M$ be an $n$-dimensional closed connected manifold $M$ and $q\in M$ a basepoint.
Its unit disk cotangent bundle $D^*M\subset T^*M$ (with respect to some Riemannian metric) with its canonical Liouville form is a Liouville domain with boundary the unit sphere cotangent bundle $S^*M$,
and the fibre $D_q^*M\subset D^*M$ is an exact Lagrangian submanifold with Legendrian boundary $S_q^*M$.
Their symplectic homologies are related to the homologies of the free loop space $\Lambda=\Lambda M$ and the based loop space $\Om=\Om_qM$ by {\em Viterbo's isomorphisms}
$$
   SH_*(D^*M)\cong H_*\Lambda,\qquad SH_{*+n}(D_q^*M)\cong H_*\Om
$$
(see Theorems~\ref{thm:Viterbo} and~\ref{thm:Viterbo-Lag} for the precise statements and references). 
Here the loop space homologies need to be twisted by a suitable local system, see Appendix~\ref{ss:local-systems}. Specifically, as in~\S\ref{ss:local-systems}, let $\sigma$ be the local system on $\Lambda$ determined by the second Stiefel-Whitney class, $\co$ the orientation local system on $M$, and $\tilde\co=\ev_0^*\co$ its pull-back to $\Lambda$ under the evaluation map. Then the Viterbo isomorphisms write $SH_*(D^*M;\sigma)\simeq H_*(\Lambda;\tilde\co)$ and $SH_{*+n}(D_q^*M;\sigma|_\Omega)\simeq H_*\Omega$. In the rest of this introduction we omit the local systems from the notation and we systematically compute Hamiltonian Floer homology groups in $D^*M$ with local coefficients in $\sigma$, and Lagrangian Floer homology groups for $D_q^*M$ with local coefficients in $\sigma|_\Omega$. Accordingly, our notation $H_*\Lambda$ stands for $H_*(\Lambda;\tilde\co)$. Note that $\tilde\co$ restricts to the trivial local system on $\Omega$, so that $H_*\Omega$ stands for homology with constant coefficients.

We define the {\em Rabinowitz loop homology} and its degree shifted version by
$$
   \wh H_*\Lambda := RFH_*(S^*M) = SH_*(S^*M),\qquad \wh\H_*\Lambda = \wh H_{*+n}\Lambda,
$$
and the {\em based Rabinowitz loop homology} by\footnote{We do \emph{not} use the notation $\wh \H_*\Om$ in order to stress that $\wh H_*\Om$ contains the Pontryagin ring $H_*\Om$, with its product of degree zero. See Theorem~\ref{thm:splitting-intro} below.}
$$
   \wh H_*\Om := RFH_{*+n}(S_q^*M) = SH_{*+n}(S_q^*M). 
$$
Now the preceding theorems specialize to results on (based) Rabinowitz loop homology. For the precise statement, note that the inclusion $i:\Om\into\Lambda$ induces pushforward and shriek maps on homology
$$
   i_*:H_*\Om\to H_*\Lambda,\qquad i_!:H_{*+n}\Lambda\to H_*\Om.
$$

\begin{theorem}[TQFT structure and Poincar\'e duality for loop spaces]\label{thm:TQFT-loop-intro}
The pair $(\wh\H_*\Lambda,\wh H_*\Om)$ with field coefficients carries a canonical graded open-closed TQFT structure whose closed-open map $\boldi_!:\wh\H_*\Lambda\to\wh H_*\Om$ extends the shriek map $i_!$, and whose open-closed map $\boldi_*:\wh H_{*+n}\Om\to\wh\H_*\Lambda$ extends the pushforward map $i_*$. The Poincar\'e duality isomorphisms intertwine this structure with the corresponding structure on Rabinowitz loop cohomology $(\wh\H^{1-2n-*}\Lambda,\wh H^{1-n-*}\Om)$. 
\end{theorem}

A partial version of the open-closed graded Frobenius algebra structure on Rabinowitz loop homology has been defined from an algebraic perspective at chain level by Rivera and Wang~\cite{Rivera-Wang} using Tate-Hochschild cohomology of a differential graded symmetric Frobenius algebra. We expect a full equivalence between the geometric and the algebraic perspective. 
See~\cite{Naef-Rivera-Wahl,Naef-Willwacher} for more details in the algebraic direction, and~\cite{Legout} for relevant work at chain level on the symplectic side. 

{\bf Reduced loop homology and splitting. }
We restrict to coefficients in a field $\bk$ and assume that the manifold $M$ is $\bk$-orientable. 
In view of Viterbo's isomorphism, the long exact sequence~\eqref{eq:les-intro} becomes
\begin{equation}\label{eq:les-loop-intro}
\xymatrix
@C=30pt
{
    \ar[r]& H^{-*}\Lambda \ar[r]^-\eps & H_*\Lambda \ar[r]^-\iota & \wh H_*\Lambda \ar[r]^-\pi & H^{1-*}\Lambda\ar[r]& 
}
\end{equation}
It was shown in~\cite{Cieliebak-Frauenfelder-Oancea} that the map $\eps$ lives only in degree zero, where it is given by multiplication with the Euler characteristic $\chi(M)$ of $M$. Therefore, the  
{\em reduced loop homology and cohomology groups}
$$
   \ol H_*\Lambda := \coker\,\eps,\qquad \ol H^*\Lambda := \ker\eps
$$
differ from $H_*\Lambda$ and $H^*\Lambda$ only by $\chi(M)$ times the point class (and not at all if $\chi(M)=0$).
Since $\iota$ is a ring map, the loop product $\mu=\bullet$ descends to a product $\bar\mu$ on $\ol H_*\Lambda$.
Replacing loop homology and cohomology by their reduced counterpart turns the long exact sequence~\eqref{eq:les-loop-intro} into the short exact sequence
\begin{equation}\label{eq:ses-intro}
\xymatrix
@C=30pt
{
   0 \ar[r] & \ol{H}_*\Lambda \ar[r]^-{\iota} & \wh H_*\Lambda \ar[r]^-{\pi} & \ol{H}^{1-*}\Lambda \ar[r] & 0.
}
\end{equation}
In the based loop case, the corresponding map $\eps$ vanishes for degree reasons and we obtain a short exact sequence
\begin{equation}\label{eq:ses-based-intro}
\xymatrix
@C=30pt
{
   0 \ar[r] & H_*\Om \ar[r]^-\iota & \wh{H}_*\Om \ar[r]^-\pi & H^{1-n-*}\Om \ar[r] & 0. 
}
\end{equation}

To formulate the next theorem, we denote the Pontrjagin product (of degree $0$) on $H_*\Om$ by $\bullet_\Om$, and the based cohomology product (of degree $n-1$) on $H^*(\Om,q)$ (relative to the constant loop $q$ at the basepoint) by $\oast_\Om$. 

\begin{theorem}[{Splitting~\cite{CO-cones,CHO-MorseFloerGH,CHO-reducedSH}}]\label{thm:splitting-intro}
(a) For $\bk$-orientable $M$ of dimension $n\ge 3$ the short exact sequence~\eqref{eq:ses-intro} admits a splitting (which is canonical if $H_1(M)=0$)
\begin{equation}\label{eq:splitting}
   \wh{H}_*\Lambda = \ol{H}_*\Lambda\oplus \ol{H}^{1-*}\Lambda
\end{equation}
such that the product on $\wh{H}_*\Lambda$ restricts to the loop product $\bar\mu$ on the subring $\ol{H}_*\Lambda$, and to an extension $\bar\oast$ of the cohomology product on the subring (not containing the unit) $\ol{H}^{1-*}\Lambda$. 

(b) The short exact sequence~\eqref{eq:ses-based-intro} admits a splitting (which is canonical for $n\geq 2$) 
\begin{equation}\label{eq:splitting-based}
   \wh{H}_*\Om = H_*\Om\oplus H^{1-n-*}\Om
\end{equation}
such that the product on $\wh{H}_*\Om$ restricts to the Pontrjagin product on the subring $H_*\Om$, 
and to an extension $\bar\oast_\Om$ of the based cohomology product on the subring (not containing the unit) $H^{1-n-*}\Om$.

(c) The cohomologies $\wh{H}^*\Lambda$ and $\wh{H}^*\Om$ admit similar splittings such that the Poincar\'e duality isomorphisms from Theorem~\ref{thm:TQFT-loop-intro} simply flip the two factors in the splittings.
\end{theorem}

The statement is seen to hold in the case $n=1$ by direct inspection (see~\S\ref{sec:n=1}). We also expect item~(a) to hold for $n=2$, but this case is not covered by the method from~\cite{CHO-reducedSH}.

On the level of modules, these splittings recover the computations of $\wh{H}_*\Lambda$ in~\cite{Cieliebak-Frauenfelder-Oancea} and of $\wh{H}_*\Om$ in~\cite{Merry}. One can heuristically understand these splittings by thinking of the generators of the Floer complex that computes $\wh{\H}_*\Lambda$, which are of three kinds: positively parametrized closed Reeb orbits on $S^*M$, constants in $S^*M$, and negatively parametrized closed Reeb orbits on $S^*M$. These in turn respectively correspond to nonconstant closed geodesics, to one homological copy of $M$ and one cohomological copy of $M$ bound together as a copy of $S^*M$, and to nonconstant closed geodesics ``traversed backwards'', which we also interpret as cohomological data. A more formal perspective is that of the so-called \emph{Drinfel'd double construction}: in the closed string case e.g. it is proved in~\cite{CHO-reducedSH} that, if $\dim M\ge 3$ and $H^{n-1}(M)=0$, reduced loop homology $B=\ol{\H}_*\Lambda$ carries the structure of a commutative cocommutative  infinitesimal bialgebra. Identifying reduced loop cohomology with $B^\vee$, the bialgebra structure on $B$ induces canonically a Frobenius algebra structure on the direct sum $B\oplus B^\vee$, and this structure is isomorphic to $\wh{\H}_*\Lambda$~\cite{CO-cones,Latschev-Oancea}. It is worth noting that the algebra structure on $B\oplus B^\vee$ has mixed terms that multiply nontrivially the two factors (these mixed terms are already visible in the explicit computations from~\S\ref{sec:spheres}). The Splitting Theorem therefore expresses the fact that Rabinowitz loop homology is the Drinfel'd double of reduced loop homology. A similar statement holds for based loop homology. 

The extended product $\bar\oast$ on $\ol{H}^*\Lambda$ can have nontrivial contributions involving classes of constant loops. This happens, for example, for the loop spaces of odd-dimensional spheres discussed in~\S\ref{sec:spheres}. 

{\bf Puzzles resolved. }
Now we can resolve the puzzles above. 

Puzzle (a) is resolved in two ways. The first one is given by Theorem~\ref{thm:splitting-intro}: the loop product $\mu=\bullet$ descends and the loop coproduct $\lambda$ extends to reduced loop homology $\ol H_*\Lambda$. Here $\ol H_*\Lambda$ is the unique space with this property: we need to mod out at least $\chi(M)$ times the point class for the coproduct to extend, and we cannot mod out more for the product still to descend. A drawback of this solution is the non-canonicity of the extension $\bar\lambda$. A more satisfactory solution is given by Theorem~\ref{thm:Frob-RFH-intro} (applied to $S^*M$), which provides {\em canonical} extensions of $\mu$ and $\lambda$ to $\wh H_*\Lambda$. 

Puzzle (b) is resolved by Theorem~\ref{thm:splitting-intro}: the loop product has a canonical extension from $\ol H_*\Lambda$ (Morse homology on $\Lambda$) to a product on $\wh{H}_*\Lambda=\ol H_*\Lambda\oplus\ol H^{1-*}\Lambda$ whose restriction to the second summand $\ol H^{1-*}\Lambda$ (Morse cohomology of $\Lambda$) is the extended cohomology product. 

Puzzle (c) is resolved by the proof of Theorem~\ref{thm:PD-RFH-intro}, which constructs the extended cohomology product $\boldlambda^\vee$ as a secondary product derived from a vanishing primary product.  

Puzzle (d) proved the most tricky one.
If $H^{n-1}(M)=0$, then the extension $\bar\lambda$ of the loop coproduct to reduced loop homology is canonical and the pair $(\bar\mu,\bar\lambda)$ satisfies Sullivan's relation (see~\cite{CHO-reducedSH}).
In general (e.g.~for $M=S^1$ in~\S\ref{sec:spheres}), however, the extensions are non-canonical and Sullivan's relation involves an extra term arising from the unit $\bar\eta$, 
\begin{equation}\label{eq:gen-Sullivan}
\bar\lambda\bar\mu = (1\otimes\bar\mu)(\bar\lambda\otimes 1) + (\bar\mu\otimes 1)(1\otimes\bar\lambda) - (\bar\mu\otimes \bar\mu)(1\otimes \bar\lambda\bar\eta\otimes 1).
\end{equation}
This failure of Sullivan's relation
led to the discovery of the graded Frobenius algebra structure on Rabinowitz Floer 
homology (Theorem~\ref{thm:Frob-RFH-intro}), from which the generalized Sullivan relation~\eqref{eq:gen-Sullivan}
derives algebraically. So, contrarily to its appearance, Sullivan's relation (in its generalized form) {\em does} arise from a TQFT after all!


Puzzle (e) is resolved as follows.  
According to Theorem~\ref{thm:PD-RFH-intro} (applied to $S^*M$) the loop product $\bullet$ and
the cohomology product $\oast$ 
have natural extensions to products $\wh\bullet$ and $\wh\oast$ on
Rabinowitz loop homology $\wh{H}_*\Lambda$ and its dual $\wh{H}^*\Lambda$,
respectively, and these extensions are related by the Poincar\'e
duality isomorphism.
Using this, we extend in \S\ref{sec:applications} each pair of
results for $\bullet$ and $\oast$ to a pair of results for
$\wh\bullet$ and $\wh\oast$ which is related via Poincar\'e duality. 
In particular, the result for $\wh\bullet$ implies the classical results
for $\bullet$ and $\oast$. While the latter had topological proofs,
the result for $\wh\bullet$ will be proved in each case by {\em
  symplectic} methods.

{\bf Structure of the paper and relation to other papers. }
This is the ``master ring'' of a series on Poincar\'e duality for loop spaces.
Here we introduce Rabinowitz loop homology, establish its basic properties (Poincar\'e duality, graded Frobenius algebra structure, open-closed TQFT structure) in the context of Liouville domains, and discuss its implications for the study of loop spaces. 
It is related to the other papers~\cite{CHO-index,CHO-MorseFloerGH,CHOS-Cross,CHO-algebra,CO-cones,CO-Tate,CHO-reducedSH} as follows. 

Theorems~\ref{thm:PD-RFH-intro},~\ref{thm:Frob-RFH-intro} and~\ref{thm:TQFT-RFH-intro} on Rabinowitz Floer homology are proved in~\S\ref{sec:PD-RFH}--\S\ref{sec:TQFT} of this paper.
Theorem~\ref{thm:TQFT-loop-intro} is proved in~\S\ref{sec:loop-proofs}, using as an input the results in~\cite{CHO-MorseFloerGH} relating various constructions of secondary coproducts. 
The Splitting Theorem~\ref{thm:splitting-intro} is proved in~\cite{CHO-reducedSH} in the more general context of certain Weinstein domains. 
Paper~\cite{CO-cones} recasts some results of this paper in a more general framework of algebraic structures on cones; it is used as an input for~\cite{CHO-MorseFloerGH} and~\cite{CHO-reducedSH}. 
Papers~\cite{CHO-algebra} and~\cite{CO-Tate} develop the theory of graded Frobenius algebras and open-closed TQFTs in the context of linearly topologized vector spaces.

In \S\ref{sec:applications} we prove several applications of Poincar\'e duality:
In~\S\ref{sec:crit-levels} we generalize results in~\cite{Goresky-Hingston} concerning the behaviour of critical levels with respect to products.
In~\S\ref{sec:Gysin} we derive the Hopf-Freudenthal-Gysin formulas in~\cite{Goresky-Hingston} from the graded open-closed TQFT structure. 
In~\S\ref{sec:point-class} we give a new proof of a result of Tamanoi~\cite{Tamanoi} on the loop product with the point class and compute it for various examples. 
In~\S\ref{sec:Uebele} we use a theorem of Uebele to describe the
Rabinowitz loop homology ring of manifolds all of whose geodesics are closed.
This is applied in upcoming joint work with Shelukhin (\cite{CHOS-Cross}, see~\S\ref{sec:string-point-inv} for a summary) to answer the question of string point invertibility of
constant rank one symmetric spaces, which is in turn related
to resonances and a conjecture of Viterbo concerning spectral norms.
In~\S\ref{sec:index} we prove a duality theorem between index and index+nullity for closed geodesics as a consequence of an iteration formula due to Liu and Long.
As an application, we show in~\S\ref{sec:nilpotency} that two sufficient conditions for the existence of infinitely many geodesics have generalizations related by Poincar\'e duality; the generalized statements are related to the Conley conjecture and will be pursued in~\cite{CHO-index}.

{\bf Acknowledgements. }
The first author thanks Stanford University, Institut Mittag--Leffler, and the Institute for Advanced Study (IAS) for their hospitality over the duration of this project.

The second author is grateful for support over
the years from the IAS, and in particular during
the academic year 2019-20. 

The third author thanks the IAS and Helmut Hofer for their hospitality. He was supported by ANR grants MICROLOCAL 15-CE40-0007, ENUMGEOM 18-CE40-0009, COSY 21-CE40-000, and by the University of Strasbourg Institute for Advanced Study (USIAS).

The authors benefited from discussions with
P.~Albers, 
R.~Cohen,
U.~Frauenfelder,
R.~Kaufmann,
J.~Latschev,
T.~Mazuir,
F.~Naef,
M.~Rivera,\break
P.~Safronov,
N.~Wahl,
and M.~Schwarz who long, long
ago pointed out the concept and importance of nilpotence of products in the
context of the pair-of-pants product.

\section{Applications of Poincar\'e duality in string topology}\label{sec:applications}

In this section we present evidence for Poincar\'e duality that was collected over the past years by the second author (and got this project started). In each case, we present a pair of theorems valid in the classical setting of homology/cohomology of the free or based loop space. We then present an extension of that pair of theorems to $\wh{H}_*$ and $\wh{H}^*$ and we explain that, on the one hand, the extended statements are related via Poincar\'e duality, and on the other hand, the classical statements are implied by the extended statements. 
In all cases, the extended statements have {\em symplectic} proofs. 

\subsection{Length filtration}\label{sec:length-filtration}

In this preparatory section we discuss the compatibility of our algebraic structures with suitable length filtrations.
Let us fix a Riemannian metric on $M$ and denote by $\Lambda^{<a}\subset\Lambda^{\leq a}\subset\Lambda$ the subspaces of loops of length $<a$ resp.~$\leq a$. For $a<b$ we set
$$
   H_*^{(a,b)}\Lambda := H_*(\Lambda^{<b},\Lambda^{\leq a}),\qquad 
   H^*_{(a,b)}\Lambda := H^*(\Lambda^{<b},\Lambda^{\leq a}).  
$$
These groups form a {\em double filtration} in the sense that we have canonical maps $H_*^{(a,b)}\Lambda\to H_*^{(a',b')}\Lambda$ for $a\leq a'$ and $b\leq b'$ with the obvious properties, and similarly on cohomology. 

\begin{theorem}[Filtered Poincar\'e duality for free loop spaces]\label{thm:filtered-PD}
A choice of Riemannian metric on $M$ induces double filtrations $\wh H_*^{(a,b)}\Lambda$ on $\wh H_*\Lambda$ and $\wh H^*_{(a,b)}\Lambda$ on $\wh H^*\Lambda$ which are compatible with
\begin{itemize}
\item the products on homology and cohomology in the sense that
\begin{gather*}
   \boldsymbol{\mu}: \wh{H}_*^{(a,b)}\Lambda\otimes \wh{H}_*^{(a',b')}\Lambda \to \wh{H}_*^{(\max\{a+b',a'+b\},b+b')}\Lambda,\cr
   \boldsymbol{\lambda}^\vee: \wh{H}^*_{(a,b)}\Lambda\otimes \wh{H}^*_{(a',b')}\Lambda \to \wh{H}^*_{(a+a',\min\{a+b',a'+b\})}\Lambda;
\end{gather*}
\item Poincar\'e duality in the sense that it induces isomorphisms
$$
   PD^{(a,b)}: \wh{H}_*^{(a,b)}\Lambda\stackrel\simeq\longrightarrow \wh{H}^{1-*}_{(-b,-a)}\Lambda;
$$
\item the Splitting Theorem~\ref{thm:splitting-intro} in the sense that
$$
   \wh{H}_*^{(a,b)}\Lambda = \ol H_*^{(a,b)}\Lambda\oplus \ol H^{1-*}_{(-b,-a)}\Lambda.
$$
\end{itemize}
\end{theorem}

\begin{remark}
In the context of loop homology and loop cohomology, a statement analogous to Theorem~\ref{thm:filtered-PD}.(a) was proved and plays a key role in~\cite{Goresky-Hingston}. Statements (b) and (c) in Theorem~\ref{thm:filtered-PD} give meaning to the fact that the homological and cohomological analogues of (a) are dual to each other.  
\end{remark}

For $\eps>0$ smaller than the length of the shortest closed geodesic on $M$ the long exact sequence~\eqref{eq:les-intro} becomes (by~\cite{CO} and Poincar\'e duality on $M$) 
\begin{equation} \label{eq:Gysin_sequence}
\xymatrix
@C=20pt
{
   \cdots H^{-*}_{(-\eps,\eps)}\Lambda \ar[r]^-\eps \ar[d]^\cong & H_*^{(-\eps,\eps)}\Lambda \ar[r]^-\iota \ar[d]^\cong & \wh{H}_*^{(-\eps,\eps)}\Lambda \ar[r]^-\pi \ar[d]^\cong & H^{1-*}_{(-\eps,\eps)}\Lambda\cdots \ar[d]^\cong \\
   \cdots H^{-*}(M;\co) \ar[r]^-{\cup[e]} & H^{n-*}(M)\ar[r]^-{p^*} & H^{n-*}(S^*M) \ar[r]^-{p_*} & H^{1-*}(M;\co)\cdots
}
\end{equation}
where the bottom row is the Gysin sequence of the sphere bundle $p:S^*M\to M$ with Euler class $[e]$ and $\co$ is the orientation local system on $M$.
The products on $H_*^{(-\eps,\eps)}\Lambda$ and $\wh{H}_*^{(-\eps,\eps)}\Lambda$ translate into the cup products on $H^{n-*}(M)$ and $H^{n-*}(S^*M)$, for which $p^*$ is a ring map.  
Poincar\'e duality becomes
$$
\xymatrix
@C=50pt
{
   \wh{H}_*^{(-\eps,\eps)}\Lambda \ar[r]^{PD^{(-\eps,\eps)}}_\cong \ar[d]_\cong & \wh{H}^{-*+1}_{(-\eps,\eps)}\Lambda \ar[d]^{\cong}\\
   H^{n-*}(S^*M)\ar[r]_{PD}^\cong & H_{*+n-1}(S^*M)
}
$$
where the bottom horizontal arrow is the classical Poincar\'e duality isomorphism. 
This supports our previous claim that the Poincar\'e duality theorem~\ref{thm:PD-RFH-intro} should be seen as generalizing Poincar\'e duality for $S^*M$ (which is always an orientable manifold), and not as generalizing Poincar\'e duality on $M$ itself. 

All the results in this subsection have obvious counterparts for the based loop space $\Om$.

\subsection{Critical levels}\label{sec:crit-levels} 

The length filtrations of the previous subsection yield an increasing filtration on $H_*\Lambda$ by 
$$
   H_*^{(-\infty,a)}\Lambda = H_*(\Lambda^{<a}),
$$
and a decreasing filtration on $H^*\Lambda$ by 
$$
   H^*_{(a,\infty)}\Lambda = H^*(\Lambda,\Lambda^{\le a}). 
$$

\begin{definition}[Critical levels] \label{defi:critical-levels-H}
(1) For a homology class $X\in H_*\Lambda$ denote 
$$
   \Cr(X) = \inf \{a\in\R \mid X\in \im(H_*^{(-\infty,a)}\Lambda \to H_*\Lambda)\}.
$$
In other words, $\Cr(X)$ is the infimum of the values of $a$ such that $X$ is represented by a cycle contained in $\Lambda^{<a}$, i.e., $X$ is supported in $\Lambda^{<a}$. 

(2) For a cohomology class $x\in H^*\Lambda$ denote 
$$
   \Cr(x) = \sup \{a\in\R \mid x\in \im(H^*_{(a,\infty)}\Lambda\to H^*\Lambda)\}.
$$
In other words, $\Cr(x)$ is the supremum of the values of $a$ such that $x$ is represented by a cochain that vanishes on all chains contained in $\Lambda^{\le a}$, i.e., $x$ is supported in $\Lambda^{> a}$.   
\end{definition}

\begin{theorem}[Goresky-Hingston~\cite{Goresky-Hingston}] \label{thm:critical-values-H}
(1) For any two homology classes $X,Y\in H_*\Lambda$ we have 
$$
   \Cr(X\bullet Y)\le \Cr(X) + \Cr(Y). 
$$
(2) For any two cohomology classes $x,y\in H^*(\Lambda,\Lambda_0)$ we have 
$$
\Cr(x\oast y )\ge \Cr(x) + \Cr(y). 
$$
\end{theorem}

We now consider an extension of the previous theorem to the setting of $\wh{H}_*$ and $\wh{H}^*$. Recall from Theorem~\ref{thm:filtered-PD}
that our choice of Riemannian metric determines an increasing filtration $\wh{H}_*^{(-\infty.a)}\Lambda$ on $\wh{H}_*\Lambda$, and a decreasing filtration $\wh{H}^*_{(a,\infty)}\Lambda$ on $\wh{H}^*\Lambda$. The following definition is analogous to Definition~\ref{defi:critical-levels-H} above.

\begin{definition}
(1) For a homology class $X\in \wh{H}_*\Lambda$ denote 
$$
   \Cr(X) = \inf \{a\in\R \mid X\in \im(\wh{H}_*^{(-\infty,a)}\Lambda \to \wh{H}_*\Lambda)\}.
$$
(2) For a cohomology class $x\in \wh{H}^*\Lambda$ denote 
$$
   \Cr(x) = \sup \{a\in\R \mid x\in \im(\wh{H}^*_{(a,\infty)}\Lambda\to \wh{H}^*\Lambda)\}.
$$
\end{definition} 

Recall that we denote the product on $\wh{H}_*\Lambda$ by $\wh{\bullet}$ and the product on $\wh{H}^*\Lambda$ by $\wh\oast$. The following extension of Theorem~\ref{thm:critical-values-H} is now an immediate consequence of the compatibility of these products with the filtrations in Theorem~\ref{thm:filtered-PD}.

\begin{theorem}\label{thm:critical-values-checkH}
(1) For any two homology classes $X,Y\in \wh{H}_*\Lambda$ we have 
$$
   \Cr(X\wh\bullet Y)\le \Cr(X) + \Cr(Y). 
$$
(2) For any two cohomology classes $x,y\in \wh{H}^*\Lambda$ we have 
$$
   \Cr(x\wh\oast y )\ge \Cr(x) + \Cr(y). 
$$
\qed
\end{theorem}

Note that each of the statements (1) and (2) is a consequence of the other one via the Filtered Poincar\'e Duality Theorem~\ref{thm:filtered-PD}.    
Moreover, the Splitting Theorem~\ref{thm:splitting-intro} ensures that, except in the homological range $\{0,\dots,n\}$, Theorem~\ref{thm:critical-values-H} is a consequence of either of the statements (1) or (2) in Theorem~\ref{thm:critical-values-checkH}.

\subsection{Hopf-Freudenthal-Gysin formulas}\label{sec:Gysin}\footnote{
{\it Historical note.} The name ``Gysin formulas" seems to have been coined by Fulton in his book on intersection theory~\cite{Fulton-1st-ed} in connection with the geometric interpretation of shriek maps in fibered setups. Gradually, this name came to designate a variety of algebraic relations involving shriek maps. The ones that we prove in this section should more appropriately be called ``Hopf-Freudenthal formulas". The situation is clearly summarized in the Introduction of~\cite{Cohen-Klein} (but the reference to Hopf's paper is wrong). Hopf~\cite{Hopf-Umkehr} first defined a ``reverse" ({\it umkehr}) homomorphism in singular homology associated to a map between manifolds of the same dimension, and Freudenthal~\cite{Freudenthal} made the connection to Poincar\'e duality and henceforth extended the definition to maps between manifolds of any dimension. In particular, Hopf~\cite[(4)]{Hopf-Umkehr} and Freudenthal~\cite[(V)]{Freudenthal} proved the finite dimensional analogues of the second formulas in Theorem~\ref{thm:Gysin}(1) and~(2) below. While these also appear in the later---and unique---paper of Gysin~\cite[(11.1)]{Gysin}, the new contribution of that paper is to construct the ``Gysin long exact sequence" associated to a sphere bundle and its salient point is a geometric interpretation of an {\it umkehr} map in a fibered context. This long exact sequence happens to play an important role in our paper. The {\it umkehr} homomorphisms of Hopf-Freudenthal are also referred to notationally as ``shriek" maps. 
}

Here we discuss the relations between the product structures on the free loop space $\Lambda$ and on the based loop space $\Om=\Om_{q}M$. We denote the Pontrjagin product (of degree $0$) on $H_*\Om$ by $\bullet_\Om$, and the based cohomology product (of degree $n-1$) on $H^*(\Om,q)$ (relative to the constant loop $q$ at the basepoint) by $\oast_\Om$. Recall that the inclusion $i:\Om\into\Lambda$ induces pushforward/pullback maps $i_*:H_*\Om\to H_*\Lambda$ and $i^*:H^*\Lambda\to H^*\Om$, as well as shriek maps (induced by intersection with the codimension $n$ submanifold $\Om\subset\Lambda$) $i_!:H_*\Lambda\to H_{*-n}\Om$ and $i^!:H^*\Om\to H^{*+n}\Lambda$. 

\begin{theorem}[Goresky-Hingston~\cite{Goresky-Hingston}] \label{thm:Gysin}
(a) For all $A,B\in H_*\Lambda$ and $C\in H_*\Omega$ we have 
$$
   i_!(A\bullet B)=i_!A\bullet_\Om i_!B,
$$ 
$$
   (i_*C)\bullet A = i_*(C\bullet_\Om i_!A). 
$$
(b) For all $a,b\in H^*(\Lambda,\Lambda_0)$ and $c\in H^*(\Omega,q)$ we have 
$$
   i^*(a \oast b)=i^*a \oast_\Omega i^*b,
$$
$$
   (i^!c)\oast a = i^!(c\oast_\Omega i^*a).
$$
\end{theorem}

The above theorem extends to $\wh{H}_*$ and $\wh{H}^*$ as follows.  
We denote the products on $\wh{H}_*\Lambda$ by $\wh\bullet$ (of degree $-n$), on $\wh{H}_*\Omega$ by $\wh\bullet_\Om$ (of degree $0$), on $\wh{H}^*\Lambda$ by $\wh\oast$ (of degree $n-1$), and on $\wh{H}^*\Omega$ by $\wh\oast_\Omega$ (of degree $n-1$). The graded open-closed TQFT structures on $(\wh{H}_*\Lambda,\wh{H}_*\Om)$ and $(\wh{H}^*\Lambda,\wh{H}^*\Om)$ from Theorem~\ref{thm:TQFT-RFH-intro} include in particular operations 
\begin{gather*}
   \boldi_!:\wh{H}_*\Lambda\to \wh{H}_{*-n}\Omega,\quad
   \boldi^!:\wh{H}^*\Omega\to \wh{H}^{*+n}\Lambda,\cr
   \boldi_*:\wh{H}_*\Omega\to \wh{H}_*\Lambda,\quad 
   \boldi^*:\wh{H}^*\Lambda\to \wh{H}^*\Omega.
\end{gather*}
Here $\boldi_!$ is the closed-open map defined by the zipper, $\boldi_*$ is the open-closed map defined by the cozipper, and $\boldi^!,\boldi^*$ are their algebraic duals. See Figure~\ref{fig:zipper-cozipper}. 

\begin{figure} [ht]
\centering
\input{zipper-cozipper.pstex_t}
\caption{The zipper and the cozipper.}
\label{fig:zipper-cozipper}
\end{figure}

\begin{theorem} \label{thm:Gysin-check}
(a) For all $A,B\in \wh{H}_*\Lambda$ and $C\in \wh{H}_*\Omega$ we have 
$$
   \boldi_!(A\, \wh\bullet \,  B) = \boldi_!A\, \wh\bullet_\Om \, \boldi_!B,
$$ 
$$
   (\boldi_*C)\, \wh\bullet \, A = \boldi_*(C\, \wh\bullet_\Om \, \boldi_!A). 
$$
(b) For all $a,b\in \wh{H}^*\Lambda$ and $c\in \wh{H}^*\Omega$ we have 
$$
   \boldi^*(a \, \wh\oast \, b) = \boldi^*a \, \wh\oast_\Omega \, \boldi^*b,
$$
$$
   (\boldi^!c)\, \wh\oast \, a = \boldi^!(c\, \wh\oast_\Omega \, \boldi^*a).
$$
\end{theorem}

\begin{proof} 
Part (a) is an immediate consequence of the graded open-closed TQFT structure on $(\wh{H}_*\Lambda,\wh{H}_*\Om)$ from Theorem~\ref{thm:TQFT-loop-intro}: the first relation says that the zipper is an algebra map, which is condition (3) in Definition~\ref{defi:gradedTQFT}, and the second relation says that the cozipper intertwines the canonical right module structures, which is derived from the axioms of a graded open-closed TQFT in~\cite[Lemma 6.6]{CHO-algebra}. See Figure~\ref{fig:Gysin} for picture proofs.
Similarly, part (b) is a consequence of the graded open-closed TQFT structure on $(\wh{H}^*\Lambda,\wh{H}^*\Om)$. 
\end{proof}

\begin{figure} [ht]
\centering
\input{Gysin.pstex_t}
\caption{TQFT proof of the Hopf-Freudenthal-Gysin formulas.}
\label{fig:Gysin}
\end{figure} 

According to Theorem~\ref{thm:TQFT-loop-intro}, statements (a) and (b) in Theorem~\ref{thm:Gysin-check} imply one another via Poincar\'e duality. 
Theorem~\ref{thm:Gysin-check} implies Theorem~\ref{thm:Gysin} via the inclusions in Theorem~\ref{thm:TQFT-loop-intro} and their based loop counterparts, using the fact that the maps $\boldi_!,\boldi^!,\boldi_*,\boldi^*$ agree on $H_*$ and $H^*$ with the topological shriek and pushforward/pullback maps, see~\S\ref{sec:topological-oc-co}.

\subsection{Loop product with the point class}\label{sec:point-class}

In the following discussion we assume that $M$ is oriented and we use $\Z$-coefficients. 
The long exact sequence~\eqref{eq:les-intro} and the fact that $\iota$ is a ring map imply that $\im\eps=\ker\iota$ is an ideal. By the description
of the map $\eps$ in the Introduction we have $\im\eps=\Z\chi(M)[q]$, where $[q]\in H_0\Lambda$ is the class of the constant loop at the basepoint $q\in M$. Since the loop product with $[q]$ is given by the composition
$$
   H_{*+n}\Lambda\stackrel{i_!}\longrightarrow H_*\Om\stackrel{i_*}\longrightarrow H_*\Lambda, 
$$
we recover the following result of Tamanoi. 

\begin{corollary}[Tamanoi~\cite{Tamanoi}]\label{cor:Tamanoi}
If $\chi(M)\neq 0$, then $\Z\chi(M)[q]$ is an ideal in the ring $H_*\Lambda$.  
Thus $\chi(M)i_*i_!a=\chi(M)[q]\bullet a$ is an integer multiple of $\chi(M)[q]\in H_0\Lambda$ for each $a\in H_*\Lambda$, in particular it vanishes whenever $\deg a\neq n$ or $a$ lives in a nontrivial path component of $\Lambda$. 
\end{corollary}

The corollary is derived in~\cite{Tamanoi} from a partial TQFT structure on $H_*\Lambda$. Note that we always have the nontrivial product $[q]\bullet[M]=[q]$. 

\begin{example}\label{ex:projective-spaces}
In this example we use $\Z$-coefficients and the computations of the (degree shifted) loop homology rings in~\cite{Cohen-Jones-Yan}.  

(a) For $M=\C P^n$ we have
$$
   \H_{*}\Lambda\C P^n = \Lambda[w]\otimes\Z[c,s]/\la c^{n+1},(n+1)c^ns,wc^n\ra
$$ 
with $|w|=-1$, $|c|=-2$ and $|s|=2n$. The point class is $[q]=c^n$, the Euler characteristic is $(n+1)$, and we see that $\Z\cdot(n+1)c^n$ is an ideal. Note that $\Z\cdot c^n$ is not an ideal. 

(b) For $M=S^n$ with $n$ even we have
$$
   \H_{*}\Lambda S^n = \Lambda[b]\otimes\Z[a,s]/\la a^2,ab,2as\ra
$$ 
with $|a|=-n$, $|b|=-1$ and $|s|=2n-2$. The point class is $[q]=a$, the Euler characteristic is $2$, and we see that $\Z\cdot 2a$ is an ideal. Note that $\Z\cdot a$ is not an ideal. 

(c) For $M=S^n$ with $n\geq 3$ odd we have
$$
   \H_{*}\Lambda S^n = \Lambda[a]\otimes\Z[u] = \H_{*}(S^n)\otimes H_*\Om S^n
$$ 
with $|a|=-n$ and $|u|=n-1$. The point class is $[q]=a$ and the Euler characteristic is $0$. It follows that in the composition
$$
   \Lambda[a]\otimes\Z[u]\stackrel{i_!}\longrightarrow \Z[u]\stackrel{i_*}\longrightarrow \Lambda[a]\otimes\Z[u]
$$
the first map is the canonical projection and the second one multiplication with $a$, so the map $i_*i_!$ is the projection (with infinite dimensional image)
$$
   H_{n+*}(S^n)\otimes H_*\Om S^n\to aH_*\Om S^n\cong H_*\Om S^n.
$$
(d) Let $M=G$ be a compact Lie group. Then the Euler characteristic is zero, since there always exists a nowhere vanishing left-invariant vector field. To compute the image of the map $i_*i_!$ note that we have $\Lambda G\cong G\times\Om G$, so the K\"unneth formula gives an injection
$$
   H_*G\otimes H_*\Om G\into H_*\Lambda G. 
$$
It follows that in the composition $H_{n+*}\Lambda G\stackrel{i_!}\longrightarrow H_*\Om G\stackrel{i_*}\longrightarrow H_*\Lambda G$ 
the first map is surjective (with right inverse $a\mapsto[G]\otimes a$) and the second one is the canonical inclusion, so the map $i_*i_!$ is the composition (with infinite dimensional image)
$$
   H_{n+*}\Lambda G\supset [M]\otimes H_*\Om G\to H_*\Om G \to H_*\Lambda G.
$$
\end{example}

\subsection{Manifolds all of whose geodesics are closed}\label{sec:Uebele}

Manifolds all of whose geodesics are 
simple and closed of the same primitive length
have a long history of study, see~\cite{Besse}. In~\cite{Goresky-Hingston} it was observed that the loop and cohomology products have special properties on such manifolds. 

\begin{theorem}[{Goresky-Hingston~\cite{Goresky-Hingston}}] 
Let $M$ be a closed Riemannian $n$-manifold all of whose geodesics are simple (i.e.~without self-intersections) and closed of the same primitive length. Let $\lambda$ denote their Morse index and set $b:=n-1+\lambda$.  

(a) Let $\Theta\in H_{2n-1+\lambda}\Lambda$ be the homology class of the cycle determined by all simple closed geodesics. Then the loop product with $\Theta$ defines an injective map
$$
   \bullet\Theta:H_*(\Lambda,\Lambda_0)\to H_{*+b}(\Lambda,\Lambda_0).  
$$
(b) Let $\om\in H^\lambda(\Lambda,\Lambda_0)$ be the Morse cohomology class determined by one simple closed geodesic. Then the cohomology product with $\om$ defines an injective map
$$
   \oast\om:H^*(\Lambda,\Lambda_0)\to H^{*+b}(\Lambda,\Lambda_0).  
$$
\end{theorem}

A common generalization of this pair of results arises from the following special case of a theorem of P.~Uebele on Rabinowitz Floer homology. 

\begin{theorem}[Uebele~\cite{Uebele-products}]\label{thm:Uebele}
Consider a Liouville domain $V$ such that $\p V$ is connected and the Reeb flow on $\p V$ is $T$-periodic. (Here $T$ is the minimal common period, but there can be Reeb orbits of smaller periods.) Assume $2c_1(V)=0$ and we are given a trivialization of the square of the canonical line bundle such that:
\begin{itemize}
\item all closed Reeb orbits on $\p V$ have Conley--Zehnder index $>3-n$. 
\item given $s\in SH_{n+b}(\p V)$ the class 
corresponding to the maximum on the Bott manifold of Reeb orbits of period $T$, we have $b>0$.
\end{itemize} 
Then $b$ is even and the following hold with coefficients in a field $\bk$.

(a) The class $s$ is invertible and makes Rabinowitz Floer homology $SH_*(\p V)$ a free and finitely generated module over the ring of Laurent polynomials $\bk[s,s^{-1}]$. 

(b) This module is (not necessarily freely) generated by 
classes that can be expressed as linear combinations of
Reeb orbits of period at most $T$. 

(c) $SH_*(\p V)$ and $SH_*(V)$ are finitely generated as $\bk$-algebras. 
\end{theorem}

\begin{remark}
In~\cite{Uebele-products} the result is stated with $\Z_2$-coefficients, without reference to a trivialization of the square of the canonical line bundle and under the additional hypothesis $\pi_1(\p V)=0$. The hypothesis $\pi_1(\p V)=0$ was  
imposed in order to have well-defined Conley--Zehnder indices independently of the choice of trivialization. 
The extension to $\bk$-coefficients is straightforward using coherent orientations in Floer theory. The restriction to field coefficients is essential because the proof uses the fact that $\bk[s,s^{-1}]$ is a principal ideal domain. 
\end{remark}

\begin{corollary}\label{cor:Uebele}
Let $M$ be a closed $n$-dimensional Riemanian manifold all of whose geodesics are closed of (not necessarily primitive) length $\ell$. Suppose that all closed geodesics have index $>3-n$. Let $s\in\wh{H}_{n+b}\Lambda$ be the class of a principal closed geodesic, corresponding to the maximum on the Bott manifold of Reeb orbits of period $\ell$. Then $b$ is even and the following hold with coefficients in a field $\bk$.

(a) The class $s$ is invertible and makes Rabinowitz loop homology $\wh{H}_*\Lambda$ a free and finitely generated module over the ring of Laurent polynomials $\bk[s,s^{-1}]$. 

(b) This module is (not necessarily freely) generated by 
classes that can be expressed as linear combinations of closed geodesics of length at most $\ell$. 

(c) In the splitting $\wh{H}_*\Lambda=\ol H_*\Lambda\oplus\ol H^{1-*}\Lambda$, the summand $\ol H_*\Lambda$ inherits the structure of a free and finitely generated $\bk[s]$-module, and $\ol H^{1-*}\Lambda$ inherits the structure of a free and finitely generated $\bk[s^{-1}]$-module.

(d) $\wh{H}_*\Lambda$, $\ol H_*\Lambda$ and $\ol H^{1-*}\Lambda$ are finitely generated as $\bk$-algebras. 
\end{corollary}

\begin{remark}
Although the splitting $\wh{H}_*\Lambda=\ol H_*\Lambda\oplus\ol H^{1-*}\Lambda$ is in general not canonical, the statement and proof of the Corollary should be understood as being valid for {\em any} splitting. Indeed, two splittings differ only in the 0-energy sector of constant loops, which is finite dimensional. 
\end{remark}

\begin{remark}
As shown by the explicit calculations from Example~\ref{ex:spheres} below, one outstanding application of Uebele's theorem is that for manifolds with periodic geodesic flow either of the loop products, homological or cohomological, can be derived from the other one.  
\end{remark}

\begin{proof}[Proof of Corollary~\ref{cor:Uebele}] 
We apply Theorem~\ref{thm:Uebele} to the unit disk cotangent bundle $V=D^*M$. Note that $2c_1(D^*M)=0$ and the square of the canonical line bundle carries a canonical trivialization determined by its canonical structure of Lagrangian fibration. In this trivialization the Conley--Zehnder index equals the Morse index, see also Appendix~\ref{sec:grading1}, and all closed geodesics have Conley--Zehnder index ($=$ Morse index) 
$>3-n$ by assumption (this is automatic if $n>3$ because the Morse index is $\ge 0$, and excludes the case $n=1$, hence $S^*M$ is necessarily connected).
Moreover, $b=n-1+\lambda>0$, where $\lambda>3-n$ is the Morse index of a principal closed geodesic. So the hypotheses of Theorem~\ref{thm:Uebele} are satisfied, and parts (a) and (b) follow immediately. 

Let $\co$ be the orientation local system on $M$. For part (c), we refer to the Gysin sequence~\eqref{eq:Gysin_sequence} and consider the induced splitting 
\begin{equation}\label{eq:splitting-constant}
H^{n-*}(S^*M) = \ol H^{n-*}(M)\oplus \ol H^{1-*}(M;\co)
\end{equation}
on the constant loops, where the first summand $\ol H^{n-*}(M)\simeq \ol H_*(M;\co)$ is contained in $\ol H_*\Lambda$ and the second one in $\ol H^{1-*}\Lambda$. Denote by $H_*^{(0,\ell]}\Lambda\subset\wh{H}_*\Lambda$ the subspace generated by the positively traversed Reeb orbits of period (or action) in $(0,\ell]$. We claim that $\ol H_*\Lambda$ is the $\bk[s]$-submodule of $\wh{H}_*\Lambda$ generated by the $\bk$-vector space 
$$
   V_* := \ol H^{n-*}(M)\oplus H_*^{(0,\ell]}\Lambda.
$$
To see this we use that, by construction of the splitting, 
the summand $\ol H_*\Lambda$ is generated by the positively traversed Reeb orbits and the constant orbits generating $\ol H^{n-*}(M)$, while $\ol H^*\Lambda$ is generated by the negatively traversed Reeb orbits and the constant orbits generating 
$\ol H^*(M;\co)$. 
In particular, $s\in\ol H_*\Lambda$. 
Since $\ol H_*\Lambda\subset\wh{H}_*\Lambda$ is a subring, this proves the inclusion $\bk[s]V_*\subset \ol H_*\Lambda$. 
For the converse inclusion, we use an argument from~\cite{Uebele-products}. By Theorem~\ref{thm:Uebele}(b), the $\bk$-vector space $\wh{H}_*\Lambda$ is generated by elements of the form $s^ka$ with $k\in\Z$ and $a\in H_*^{(0,\ell]}\Lambda$. If $s^ka\in\ol H_*\Lambda$, then by action reasons we must have $k\geq -1$. If $k\geq 0$, then $s^ka\in\bk[s]H_*^{(0,\ell]}\Lambda\subset\bk[s]V_*$. If $k=-1$, then $s^ka$ must belong to the constant part $\ol H^{n-*}(M)\subset V_*$ and the claim is proved.

By the claim, $\ol H_*\Lambda$ is finitely generated as a $\bk[s]$-submodule. It is torsion free because $\wh{H}_*\Lambda$ is torsion free as a $\bk[s,s^{-1}]$-module. Since $\bk[s]$ is a principal ideal domain, it follows that the $\bk[s]$-module $\ol H_*\Lambda$ is free. This proves the assertion on $\ol H_*\Lambda$. An analogous argument gives the assertion on $\ol H^{1-*}\Lambda$, which is the $\bk[s^{-1}]$-submodule of $\wh{H}_*\Lambda$ generated by the $\bk$-vector space 
$$
V^* := H^{1-*}_{(0,\ell]}\Lambda\oplus \ol H^{1-*}(M;\co).
$$
Part (d) is an immediate consequence of part (c). 
\end{proof}

Corollary~\ref{cor:Uebele} requires neither that the geodesics have the same primitive length, nor that they are simple. In order to describe the algebra structure in examples, suppose now that all geodesics are closed with the same primitive length $\ell$. Then the spaces $V_*$ and $V^*$ in the proof of Corollary~\ref{cor:Uebele} can be replaced by $\ol H_{n+*}(M;\co)\oplus s\,\ol H^{1-n-*}(M;\co)$ and $s^{-1}\ol H_{n+*}(M;\co)\oplus \ol H^{1-n-*}(M;\co)$, respectively, and we extract from the proof the following statements:
\begin{itemize}
\item The $\bk[s,s^{-1}]$-module $\wh{H}_*\Lambda$ is (not necessarily freely) generated by 
$$
   H^{-*}(S^*M) = \ol H_{n+*}(M;\co)\oplus \ol H^{1-n-*}(M;\co).
$$ 
\item The $\bk[s]$-submodule $\ol H_{n+*}\Lambda$ is (not necessarily freely) generated by 
$$
   \wt V_* := \ol H_{n+*}(M;\co)\oplus s\,\ol H^{1-n-*}(M;\co).
$$
\item The $\bk[s^{-1}]$-submodule $\ol H^{1-n-*}\Lambda$ is (not necessarily freely) generated by 
$$
   \wt V^* := s^{-1}\ol H_{n+*}(M;\co)\oplus \ol H^{1-n-*}(M;\co).
$$
\end{itemize}
Let us introduce the degree shifted algebra
$$
   \wh{\H}_*\Lambda := \wh{H}_{*+n}\Lambda = \ol H_{*+n}\Lambda\oplus \ol H^{1-n-*}\Lambda,
$$
graded by the shifted degree $|\gamma|=\CZ(\gamma)-n$. Then the product has degree zero and is graded commutative, and the class $s$ above has degree
$$
   |s| = b = n-1+\lambda>0, 
$$
where $\lambda$ is the Morse index of a principal closed geodesic. 

\begin{example}[Spheres]\label{ex:spheres}
For the loop space of $S^n$ the loop product and the cohomology product have been computed in~\cite{Cohen-Jones-Yan} and~\cite{Goresky-Hingston}, respectively. Corollary~\ref{cor:Uebele} provides a simple way to derive one product from the other in the case $n\geq 3$. For this, note first that in this case each closed geodesic has index at least $\lambda=n-1>3-n$, so Corollary~\ref{cor:Uebele} is applicable with the generator $s$ of shifted degree 
$$
   |s|=n-1+\lambda=2n-2.
$$
Now we distinguish two cases. 

{\em The case $n\geq 3$ odd. }
In this case the $\bk[s,s^{-1}]$-module $\wh{\H}_*\Lambda S^n$ is generated by the graded vector space
$$
   H^{-*}(S^*S^n) = {\rm span}_\bk\{1,a,b,ab\}
$$
in degrees 
$$
   |1|=0,\quad |a|=-n,\quad |b|=1-n,\quad |ab|=1-2n,
$$
where $1,a$ generate the first summand and $b,ab$ the second one in the splitting~\eqref{eq:splitting-constant}. For degree reasons there can be no nontrivial relations involving different powers of $s$ in ${\rm span}_\bk\{1,a,b,ab\}\otimes_\bk\bk[s,s^{-1}]$ and we conclude that
$$
   \wh{\H}_*\Lambda S^n = {\rm span}_\bk\{1,a,b,ab\}\otimes_\bk\bk[s,s^{-1}]
$$ 
as a $\bk[s,s^{-1}]$-module. The preceding discussion then gives
$$
   \H_{*}\Lambda S^n = {\rm span}_\bk\{1,a,u,au\}\otimes_\bk\bk[s],\qquad u:=sb,\quad |u|=n-1
$$ 
as a $\bk[s]$-module, and
$$
   \H^{1-2n-*}\Lambda S^n = s^{-1}{\rm span}_\bk\{1,a,u,au\}\otimes_\bk\bk[s^{-1}]
$$ 
as a $\bk[s^{-1}]$-module. Here the reduced (co)homologies are the same because $\chi(S^n)=0$.
To determine the ring structure we use an input from~\cite{Cohen-Jones-Yan} (see also Example~\ref{ex:projective-spaces} above), where it is shown that the ring structure on $\H_{*}\Lambda S^n$ has only one additional relation $u^2=s$, hence
$$
   \H_{*}\Lambda S^n = \bk[a,u]/\la a^2\ra = \Lambda[a,u]
$$ 
as a $\bk$-algebra. Since any relation in $\wh{\H}_*\Lambda S^n$ gives rise under multiplication by a large negative power of $s$ to a relation in $\H_{*}\Lambda S^n$ and vice versa, it follows that 
$$
   \wh{\H}_*\Lambda S^n =\Lambda[a,u,u^{-1}] 
$$ 
as a $\bk$-algebra. This in turn implies that
$$
   \H^{1-2n-*}\Lambda S^n = u^{-1}\Lambda[a,u^{-1}] 
$$ 
as a $\bk$-algebra. Since the classes $u^{-1},u^{-1}a$ correspond to the constant loops, the cohomology relative to the constant loops becomes
$$
   \H^{1-2n-*}(\Lambda S^n,\Lambda_0S^n) = u^{-2}\Lambda[a,u^{-1}] 
$$ 
as a $\bk$-algebra, in accordance with~\cite{Goresky-Hingston}. 

{\em The case $n\geq 3$ even. }  
If $\bk$ has characteristic $2$ the (co)homology rings are exactly as in the case $n$ odd. Suppose now that $\bk$ has characteristic $\neq 2$.  
Then the $\bk[s,s^{-1}]$-module $\wh{\H}_*\Lambda S^n$ is generated by the graded vector space
$$
   H^{-*}(S^*S^n) = {\rm span}_\bk\{1,c\}
$$
in degrees 
$$
   |1|=0,\quad |c|=1-2n,
$$
where $1$ generates the first summand and $c$ the second one in the splitting~\eqref{eq:splitting-constant}. Again there can be 
no nontrivial relations and we conclude that
$$
   \wh{\H}_*\Lambda S^n = {\rm span}_\bk\{1,c\}\otimes_\bk\bk[s,s^{-1}]
$$ 
as a $\bk[s,s^{-1}]$-module, 
$$
   \ol\H_{*}\Lambda S^n = {\rm span}_\bk\{1,b\}\otimes_\bk\bk[s],\qquad b:=sc,\quad |b|=-1
$$ 
as a $\bk[s]$-module, and
$$
   \ol\H^{1-2n-*}\Lambda S^n = s^{-1}{\rm span}_\bk\{1,b\}\otimes_\bk\bk[s^{-1}]
$$ 
as a $\bk[s^{-1}]$-module. Note that the {\em non-reduced} loop space homology 
$$
   \H_{*}\Lambda S^n = \Bigl({\rm span}_\bk\{1,b\}\otimes_\bk\bk[s]\Bigr)\oplus \bk a,\qquad |a|=-n
$$ 
is {\em not} free as a $\bk[s]$-module because $sa=0$. 
The ring structure is again determined in~\cite{Cohen-Jones-Yan} (see Example~\ref{ex:projective-spaces} above), where it is shown that 
$$
   \H_{*}\Lambda S^n = \bk[a,b,s]/\la a^2,ab,b^2,2sa\ra
$$ 
as a $\bk$-algebra. (Here the factor $2$ can be dropped because it is invertible in $\bk$, but the homology as written also gives the correct answer for $\bk$ replaced by $\Z$.) From this we again deduce the $\bk$-algebras
\begin{gather*}
   \wh{\H}_*\Lambda S^n = \bk[b,s,s^{-1}]/\la b^2\ra = \Lambda[b,s,s^{-1}],
   \cr
   \ol\H_{*}\Lambda S^n = \Lambda[b,s] 
   , \qquad
   \ol\H^{1-2n-*}\Lambda S^n = s^{-1} \Lambda[b,s^{-1}] 
   , \cr
   \H^{1-2n-*}(\Lambda S^n,\Lambda_0S^n) =  s^{-2}\Lambda[b,s^{-1}] 
   ,
\end{gather*}
the last one in accordance with~\cite{Goresky-Hingston}. Here we use the notation $\Lambda[b,s,s^{-1}]$ to designate the free commutative algebra with generators $b,s,s^{-1}$, with $b$ of odd degree and $s$ of even degree, and similarly for $\Lambda[b,s]$ and $\Lambda[b,s^{-1}]$.
\end{example}

\subsection{String point invertibility and resonances for CROSS}\label{sec:string-point-inv}

Example~\ref{ex:spheres} can be generalized to all {\em compact rank one symmetric spaces\break (CROSS)}, i.e., the projective spaces $\R P^n$, $\C P^n$, $\H P^n$, and the Cayley plane $CaP^2$. This is carried out in joint work with E.~Shelukhin~\cite{CHOS-Cross} where we compute for each CROSS the Rabinowitz loop homology ring together with its BV operator, and thus by restriction the BV algebra structures on its loop homology and loop cohomology. Moreover, we apply these computations to the following two questions. 

{\em String point invertibility. } 
Consider a closed manifold $M$ and denote by $\{\cdot,\cdot\}$ the Chas-Sullivan loop bracket on $H_*\Lambda$.\footnote{This is defined by 
$\{A,B\}=(-1)^{|A|}\Delta(AB) - (-1)^{|A|}(\Delta A)B - A \Delta B$, with $\Delta$ the BV-operator on $\H_*\Lambda=H_{*+n}\Lambda$ (cf.~\cite{CHOS-Cross}).
}
For any given class $a\in H_*\Lambda$, consider the operator $P_a:H_*(M)\to H_{*+\deg\, a -n+1}(M)$ defined by 
$$
   P_a=\ev_* \circ \{\cdot,a\}\circ i_*,
$$
with $i:M\hookrightarrow \Lambda$ the inclusion of constant loops, and $\ev:\Lambda\to M$ the evaluation. 
We call $M$ \emph{string point invertible} if there exists a coefficient field $\bk$ such that $M$ is $\bk$-orientable, and a collection of classes $a_1,\dots,a_N\in H_*\Lambda$ such that 
$$
   [M]=P_{a_N}\circ\dots\circ P_{a_1}([pt]).
$$
This property was introduced by Shelukhin~\cite{Shelukhin19}, who used it to prove the following conjecture of Viterbo in the case of string point invertible manifolds: {\em The spectral norm of the pair consisting of the zero-section inside $T^*M$ and its image under a Hamiltonian diffeomorphism supported in the unit disc bundle is uniformly bounded.} Shelukhin proved that spheres are string point invertible, and also that string point invertibility is preserved under taking products. Generalizing this, we have

\begin{theorem}\cite{CHOS-Cross}\label{thm:string-point} 
(a) Let $M$ be a CROSS modelled on $\C P^d$, $\H P^d$, or $CaP^2$ (set $d=2$ in this last case). Then $M$ is string point invertible if and only if the Euler characteristic $\chi(M)=d+1$ is prime (with coefficient field $\bk=\Z/(d+1)\Z)$. 

(b) Let $M$ be a CROSS modelled on $\R P^n$, $n\ge 3$. Then $M$ is not string point invertible with $\Z/2$-coefficients.   
\end{theorem}

{\em Resonances. }
Consider a closed Riemannian manifold $M$ and fix a coefficient ring $R$. To (co)homology classes $X\in H_k(\Lambda M)$ and $x\in H^k(\Lambda M)$ we associate their degrees $\deg(X)=k$, $\deg(x)=k$ and critical levels $\Cr(X)$, $\Cr(x)$ as defined in \S\ref{sec:crit-levels}. We say that $M$ is {\em resonant} with $R$-coefficients if there exists a constant $\ol\alpha>0$ such that 
$$
   \deg(X) - \ol\alpha\,\Cr(X)\quad\text{and}\quad \deg(x)-\ol\alpha\,\Cr(x)
$$
are uniformly bounded for all $X\in H_*(\Lambda M;R)$ and $x\in H^*(\Lambda M;R)$.  
This property is introduced in~\cite{Hingston-Rademacher} and its implications for indices and lengths of closed geodesics are discussed. Moreover, it is proved in~\cite{Hingston-Rademacher} that spheres of dimension at least $3$ are resonant with field coefficients. Generalizing this, we have

\begin{theorem}\cite{CHOS-Cross}\label{thm:resonance}
The string point invertible CROSS from the statement of Theorem~\ref{thm:string-point}(a) are resonant with coefficients in the field $\Z/(d+1)\Z$. 
\end{theorem}

\subsection{Index growth} \label{sec:index}

Consider the following result on the index growth of an iterated closed geodesic.

\begin{theorem}[Goresky-Hingston~{\cite[Proposition~6.1]{Goresky-Hingston}}] \label{thm:iteration-Bott} 
Let $\gamma$ be a closed geodesic with index $\lambda$ and (transverse) nullity $\nu$ on a manifold of dimension $n$. Let $\lambda_m$ and $\nu_m$ denote the index and nullity of the $m$-fold iterate $\gamma^m$. Then $\nu_m\le 2(n-1)$ and 
\begin{equation} \label{eq:iteration-index}
|\lambda_m-m\lambda|\le (m-1)(n-1),
\end{equation}
\begin{equation}\label{eq:iteration-index-plus-nullity}
|\lambda_m+\nu_m-m(\lambda+\nu)|\le (m-1)(n-1).
\end{equation}
\end{theorem}

These inequalities follow from standard properties of the Bott function $S^1\to \Z$ determined by the linearization of the geodesic flow along $\gamma$, see~\cite{Bott56,Goresky-Hingston}. In the context of the present paper, we wish to explain that~\eqref{eq:iteration-index} and~\eqref{eq:iteration-index-plus-nullity} are dual statements. We proceed as in the previous sections: first generalize each of these inequalities to a symplectic setting, then prove a duality theorem for the generalized statements.

The linearization of the geodesic flow along $\gamma$ determines a path in $\mathrm{Sp}(2(n-1))$ starting at $\mathrm{Id}$, canonically defined up to conjugation. Based on Bott~\cite{Bott56}, Long~\cite{Long-book} developed an index iteration theory for general paths of symplectic matrices, not necessarily obtained as linearizations of geodesic flows. To any path $P:[0,1]\to \mathrm{Sp}(2N)$ such that $P(0)=\mathrm{Id}$ is assigned its \emph{Bott-Long index} 
$$
i(P)\in\Z.
$$
See~\cite[Definitions~5.2.7 and~5.4.1]{Long-book}. (In the notation of Long~\cite{Long-book} we have $i(P)=i_1(P)$.) This is defined to be the Conley-Zehnder index of the concatenation $P\# P(1)\xi$ where $\xi$ is a ``\emph{minus} curve", i.e. a path of the form $\{t\mapsto \exp(tJS)\, : \, t\in[0,\varepsilon]\}$ with $S$ a symmetric negative definite 
matrix and $\varepsilon>0$ small. 
\footnote{As an example, we have $i(P\equiv \mathrm{Id})=-N$.}
The \emph{nullity} of such a path is
$$
\nu(P)=\nu(P(1))=\dim\ker (P(1)-\mathrm{Id}).
$$ 
The key property that is of interest to us regarding the Bott-Long index is that, if $P=P_\gamma:[0,1]\to\mathrm{Sp}(2(n-1))$ is the linearized transverse geodesic flow along a given geodesic $\gamma$, then $i(P_\gamma)$ equals the Morse index of $\gamma$, cf.~\cite[Theorem~7.3.4]{Long-book} and~\cite[Theorem~A]{Bott56}. Similarly, $\nu(P_\gamma)$ equals the nullity of $\gamma$ (in the transverse direction).

To formulate the iteration inequalities for the Bott-Long index, define $P^m:[0,m]\to \mathrm{Sp}(2N)$ by
$$
P^m(t)=P(t-j)P(1)^j,\qquad j\le t\le (j+1),\qquad j=0,\dots,m-1.
$$
If $P=P_\gamma:[0,1]\to\mathrm{Sp}(2(n-1))$ is the linearized transverse geodesic flow along a given geodesic $\gamma$, 
then $P^m$ is the linearized transverse geodesic flow along the $m$-th iterate $\gamma^m$. 

The following generalization of Theorem~\ref{thm:iteration-Bott} is just a reformulation of a result by Liu and Long. It specializes to Theorem~\ref{thm:iteration-Bott} if $P=P_\gamma$ is the linearized transverse geodesic flow along some geodesic $\gamma$. 

\begin{theorem}[Liu-Long~\cite{Liu-Long2000}]\label{thm:iteration-symplectic}
Let $P:[0,1]\to \mathrm{Sp}(2N)$ be a continuous path with $P(0)=\mathrm{Id}$. Then for all $m\ge 1$ we have 
\begin{equation} \label{eq:symplectic-index}
|i(P^m)-m i(P)|\le (m-1)N,
\end{equation}
\begin{equation} \label{eq:symplectic-index+nullity}
\big| (i(P^m)+\nu(P^m)) - m(i(P)+\nu(P)) \big| \le (m-1)N. 
\end{equation}
\end{theorem}

\begin{proof}
According to~\cite[Theorem~1.2]{Liu-Long2000} (see also~\cite[Theorem~10.1.3]{Long-book}), the following inequalities hold for all $m\ge 1$:  
\begin{eqnarray*}
\lefteqn{m\big(i(P)+\nu(P)-N\big)+N -\nu(P)}\\
& \le & i(P^m) \\
& \le & m\big(i(P)+N\big)-N-(\nu(P^m)-\nu(P)).
\end{eqnarray*}
Theorem~\ref{thm:iteration-symplectic} follows directly from these using the additional obvious inequalities $2N\ge \nu(P^m)\ge \nu(P)\ge 0$. For example, the second inequality implies 
\begin{equation*}
   i(P^m)-m i(P) \le (m-1)N - (\nu(P^m)-\nu(P)) \le (m-1)N.
\end{equation*}
\end{proof}

\begin{remark}
The proof of Theorem~\ref{thm:iteration-symplectic} ultimately relies on properties of the Bott function determined by the path $P$.
\end{remark}

The key definition for the duality statement is the following. 

\begin{definition} \label{defi:reverse-path}
Given a path $P:[0,1]\to \mathrm{Sp}(2N)$ with $P(0)=\mathrm{Id}$, the \emph{reverse path} $\ol{P}:[0,1]\to \mathrm{Sp}(2N)$ is defined by 
$$
\ol{P}(t)=P(1-t) P(1)^{-1},\qquad t\in[0,1]. 
$$
\end{definition}

Note that $\ol{P}(0)=\mathrm{Id}$. The motivation for the definition is the following. Consider a $1$-periodic compactly supported Hamiltonian $H:\R/\Z \times \R^{2N}\to \R$ and denote $\varphi_t$, $t\in\R$ the flow of the Hamiltonian vector field $X_H^t$, $t\in\R/\Z$, which solves the equation $\frac{d}{dt}\varphi_t=X_H^t\circ \varphi_t$ with initial condition $\varphi_0=\mathrm{Id}$. The $1$-periodicity of the Hamiltonian implies 
$\varphi_t\circ\varphi_1=\varphi_1\circ\varphi_t=\varphi_{1+t}$ for all $t\in\R$.   
The \emph{reverse flow} $\ol{\varphi}_t=\varphi_{-t}$
satisfies the equation $\ol{\varphi}_t\circ\varphi_1=\varphi_{1-t}$, and its linearization satisfies the equation $d\ol{\varphi}_t = d\varphi_{1-t}\circ d\varphi_1^{-1}$.  
Thus, reversing the time direction for a Hamiltonian flow corresponds at the linearized level to reversal of the path as in  Definition~\ref{defi:reverse-path}.

\begin{proposition} \label{prop:index_reverse} Given a path $P:[0,1]\to \mathrm{Sp}(2N)$ with $P(0)=\mathrm{Id}$, the index of the reverse path $\ol{P}$ is 
\begin{equation} \label{eq:index_reverse}
i(\ol{P})=-i(P)-\nu(P).
\end{equation}
\end{proposition}

This proposition is proved in Appendix~\ref{sec:index_reverse}. Using it, we can state and prove our duality theorem for the index. 

\begin{theorem}[Duality for the index] \label{thm:duality_index} 
Let $P:[0,1]\to\mathrm{Sp}(2N)$ be a path with $P(0)=\mathrm{Id}$ and $\ol{P}$ be the reverse path. 

(i) The index inequality~\eqref{eq:symplectic-index} for $P$ is equivalent to the index+nullity inequality~\eqref{eq:symplectic-index+nullity} for $\ol{P}$. 

(ii) The index+nullity inequality~\eqref{eq:symplectic-index+nullity} for $P$ is equivalent to the index inequality~\eqref{eq:symplectic-index} for $\ol{P}$.
\end{theorem}

\begin{proof} Since taking the reverse of a path is an involutive operation, assertions (i) and (ii) are equivalent. To prove~(ii) we use Proposition~\ref{prop:index_reverse} and the equality $\ol{P}^m=\ol{P^m}$ to get
{\scriptsize
\begin{eqnarray*} 
& mi(\ol{P})-(m-1)N  \le i(\ol{P}^m)\le m i(\ol{P}) + (m-1)N \\
\Leftrightarrow & m(-i(P)-\nu(P))-(m-1)N  \le -i(P^m)-\nu(P^m) \le m(-i(P)-\nu(P))+(m-1)N \\
\Leftrightarrow & m(i(P)+\nu(P))+(m-1)N \ge i(P^m)+\nu(P^m) 
 \ge m(i(P)+\nu(P)) - (m-1)N.
\end{eqnarray*}
}
\end{proof}

\begin{remark}
To the best of our knowledge, Theorem~\ref{thm:duality_index} did not appear in the literature before. We interpret it as an explanation for the fact that, in various contexts related to the problem of the existence and multiplicity of closed geodesics, the index and the index+nullity play a ``dual'' role, see~\S\ref{sec:nilpotency} below. Note that, given a path $P:[0,1]\to \mathrm{Sp}(2N)$, $P(0)=\mathrm{Id}$ which is the linearization of the geodesic flow along a closed geodesics, the path $\ol{P}$ is not the linearization of the geodesic flow along the same geodesic traversed backwards. Thus, the duality between index and index+nullity is not of a Riemannian nature, but rather of a symplectic nature. 
\end{remark}

\subsection{Level-Potency}\label{sec:nilpotency}

Given a nonconstant closed geodesic $c$ of length $L$, denote $Sc=S^{1}\cdot c \subset \Lambda$ its saturation with
respect to the circle action and define its \textit{local level homology/cohomology} 
$$
H_*(Sc) = \lim_{U\supset Sc} H_*(\Lambda \cap U,\Lambda ^{<L}\cap U),\
H^*(Sc) = \lim_{U\supset Sc} H^{\ast }(\Lambda \cap U,\Lambda ^{<L}\cap U),
$$
where $U$ is an open set in $\Lambda$ containing $Sc$. 
Here and in the rest of this subsection we use rational coefficients.

For a closed orientable Riemannian $n$-manifold $M$,
consider the following ensemble of dual results, essentially contained in~\cite{Hingston93,Hingston97,Goresky-Hingston}. 

\begin{theorem}
\label{thm:level-potency-Hingston}
Any of the following conditions implies the existence of infinitely many closed geodesics on $M$:

(a1) There exists a nonzero class $x\in H^*(\Lambda,\Lambda_0)$ such that $\Cr(x^m)=m\Cr(x)$ for all $m\ge 1$ (we say that $x$ is \emph{level-potent}).

(a2) There exists an isolated closed geodesic $c$ and $x\in H^*(Sc)$ such that $x^m\in H^*(Sc^m)$ is nonzero for all $m\ge 1$.

(a3) There exists an isolated closed geodesic $c$ of index $\lambda$ such that $H^\lambda(Sc)\neq 0$ and
$$
   \text{index}(c^m) \geq m\lambda + (m-1)(n-1)\quad\text{for all }m\geq 1.
$$
 
(b1) There exists a nonzero class $X\in H_*\Lambda$ such that $\Cr(X^m)=m\Cr(X)$ for all $m\ge 1$ (we say that $X$ is \emph{level-potent}).
   
(b2) There exists an isolated closed geodesic $c$ and $X\in H_*(Sc)$ such that $X^m\in H_*(Sc^m)$ is nonzero for all $m\ge 1$.

(b3) There exists an isolated closed geodesic $c$ of index $\lambda$ and nullity $\nu$ such that $H_{\lambda+\nu+1}(Sc)\neq 0$ and
$$
   (\text{index+nullity})(c^m) \leq m(\lambda+\nu) - (m-1)(n-1)\quad\text{for all }m\geq 1.
$$
\end{theorem}

We defer the proof of this theorem to Appendix~\ref{app:level-potency}. 
Note that, by Theorem~\ref{thm:iteration-Bott}, part (a3) corresponds to the {\em fastest} possible growth of the index, while part (b3) corresponds to the {\em slowest} possible growth of the index + nullity under iteration. 

Our goal is to generalize (a1-3) and (b1-3) to statements on Rabinowitz loop homology that are related by Poincar\'e duality. For this, consider $S^*M$ with its canonical contact form $\alpha$. By a {\em generalized closed Reeb orbit} $\gamma$ we mean a closed Reeb orbit (i.e. the lift of a closed geodesic on $M$) or its backward parametrization. Recall that, together with the constants on $S^*M$, after a Morse perturbation these are the generators of Rabinowitz Floer homology $SH_*(S^*M)=\wh{H}_*\Lambda$. 

Let $\gamma$ be a generalized closed Reeb orbit of action $\int_\gamma\alpha = L\neq 0$ and $S\gamma=S^1\cdot\gamma$ its saturation. 
By \textit{local level Rabinowitz homology/cohomology}, denoted $\wh{H}_{\ast }(S\gamma)$ and $\wh{H}^*(S\gamma)$, we mean symplectic homology/cohomology of $S^*M$ localized near the isolated set $S\gamma$ in the following sense: we choose a Hamiltonian $H$ as in the definition of $SH_*(S^*M)$ with negative slope $-\mu$ and positive slope $\tau$ such that $L\in (-\mu,\tau)$, and consider local Floer (co)homology of the isolated fixed point set which corresponds to $S\gamma$ in the convex region of $H$. This is a mild variation on McLean's definition of local symplectic homology in~\cite{McLean-LocalFH}.
In case $\gamma$ is the lift of a nonconstant closed geodesic $c$, the Viterbo isomorphism (see~\cite{Viterbo-cotangent,AS2,CHO-MorseFloerGH} and references therein) specializes to isomorphisms which intertwine the product structures
\begin{equation}\label{eq:Viterbo-local}
\wh{H}_*(S\gamma)\simeq H_*(Sc), \qquad \wh{H}^*(S\gamma)\simeq H^*(Sc).
\end{equation} 

\begin{proposition}\label{prop:level-potency}
For $\textrm{i}=1,2,3$ the following conditions (A\textrm{i}) and (B\textrm{i}) are equivalent under the Poincar\'e duality isomorphism, and either of them is equivalent to the conditions (a\textrm{i}) and (b\textrm{i}) in Theorem~\ref{thm:level-potency-Hingston}: 

(A1) There exists $x\in \wh{H}^*\Lambda$ 
such that $\Cr(x^m)=m\Cr(x)$ for all $m\ge 1$. 

(A2) There exists an isolated generalized closed Reeb orbit $\gamma$ and $x\in \wh{H}^*(S\gamma)$ such that $x^m\in \wh{H}^*(S\gamma^m)$ is nonzero for all $m\ge 1$. 
  
(A3) $S^*M$ carries an isolated generalized closed Reeb orbit $\gamma$ of index $\lambda$ with fastest possible index growth and $\wh{H}^\lambda(S\gamma)\neq 0$. 

(B1) There exists $X\!\in\!\wh{H}_*\Lambda$ 
such that $\Cr\!(X^m)=m\Cr\!(X)$ for all $m\ge 1$. 

(B2) There exists an isolated generalized closed Reeb orbit $\gamma$ and $X\in \wh{H}_*(S\gamma)$ such that $X^m\in \wh{H}_*(S\gamma^m)$ is nonzero for all $m\ge 1$. 

(B3) $S^*M$ carries an isolated generalized closed Reeb orbit $\gamma$ of index $\lambda$ and nullity $\nu$ with slowest possible index+nullity growth and $\wh{H}_{\lambda+\nu+1}(S\gamma)\neq 0$.
\end{proposition}

\begin{proof}
(A1) $\Leftrightarrow$ (B1) follows from the compatibility of Poincar\'e duality with the products and with the length filtrations, see~\S\ref{sec:length-filtration}.

Let $\gamma$ be an isolated generalized closed Reeb orbit and $\ol\gamma$ its backward parametrization. 
Then (A2) for $\gamma$ is equivalent to (B2) for $\ol\gamma$ because  
Poincar\'{e} duality in Theorem~\ref{thm:PD-RFH-intro} specializes to an isomorphism
\begin{equation*}
   \wh{H}^{*}(S\gamma) \cong \wh{H}_{1-*}(S\ol\gamma).
\end{equation*}
Condition (A3) for $\gamma$ is equivalent to (B3) for $\ol\gamma$: By Proposition~\ref{prop:index_reverse}, the indices and nullities are related by $i(\gamma)=-i(\ol\gamma)-\nu(\ol\gamma)$, hence $\wh{H}^{i(\gamma)}(S\gamma) \cong \wh{H}_{1-i(\gamma)}(S\ol\gamma) = \wh{H}_{i(\ol\gamma)+\nu(\ol\gamma)+1}(S\ol\gamma)$. 
On the other hand, by the proof of Theorem~\ref{thm:duality_index}, fastest possible index growth for $\gamma$ is equivalent to slowest possible index+nullity growth for $\ol\gamma$. 

To prove the equivalence with conditions (a1-3) and (b1-3), let $\gamma$ be the lift of an isolated closed geodesic $c$. Then by~\eqref{eq:Viterbo-local} and the above we have the following equivalences for $\mathrm{i}=1,2,3$: 
\begin{gather*}
  \text{condition (ai) for $c$} \Longleftrightarrow \text{condition (Ai) for $\gamma$} \Longleftrightarrow \text{condition (Bi) for $\ol\gamma$},\cr 
  \text{condition (bi) for $c$} \Longleftrightarrow \text{condition (Bi) for $\gamma$} \Longleftrightarrow \text{condition (Ai) for $\ol\gamma$}.
\end{gather*}
\end{proof}

Proposition~\ref{prop:level-potency} suggests a generalization of Theorem~\ref{thm:level-potency-Hingston} in terms of local level Rabinowitz homology and closed Reeb orbits, which we will pursue in~\cite{CHO-index}. It will rely on ideas from the proof of the Conley conjecture in~\cite{Hingston_Conley_conjecture,Ginzburg_Conley_conjecture,Ginzburg_Gurel_Conley_revisited} and the study of symplectically degenerate maxima in~\cite{GG-localFH,GHHM,HHM}.

\section{Poincar\'e duality in Rabinowitz Floer homology}\label{sec:PD-RFH}

Poincar\'e duality $H_*(\Sigma)\cong H^{m-*}(\Sigma)$ for an $m$-dimensional
closed oriented manifold $\Sigma$ is known to be induced by a canonical chain isomorphism\break 
$MC_*(f)\cong MC^{m-*}(-f)$ between the Morse chain complex of a Morse
function $f:\Sigma\to\R$ and the Morse cochain complex of $-f$. In this
section we show that this isomorphism canonically extends to
Rabinowitz Floer homology if $\Sigma$ is the boundary of a Liouville domain. 

\subsection{Rabinowitz Floer homology}

Recall from~\cite{Cieliebak-Frauenfelder} the definition of Rabinowitz Floer homology.
Consider the completion $(\wh V,\lambda)$ of a Liouville domain
$(V,\lambda)$ with boundary $\Sigma=\p V$. We abbreviate by
$\mathcal{L}=C^\infty(S^1,\wh V)$ the free loop space of $\wh V$,
where $S^1=\R/\Z$. A {\em defining Hamiltonian} for $\Sigma$ is a smooth
function $H:\wh V\to\R$ with 
regular level set $\Sigma=H^{-1}(0)$ whose Hamiltonian vector field $X_H$
(defined by $i_{X_H}\om=-dH$) has compact support and agrees with the
Reeb vector field $R$ along $\Sigma$. Given such a Hamiltonian, the {\em
Rabinowitz action functional} is defined by
$$
   A_H: \mathcal{L} \times \R \to \R, \qquad 
   A_H(x,\eta):= \int_0^1 x^*\lambda - \eta \int_0^1 H(x(t))dt.
$$
Critical points of $A_H$ are pairs $(x,\eta)$ such that $x:S^1\to \Sigma$
solves 
\begin{equation*}
   \dot x = \eta X_H(x) = \eta R(x). 
\end{equation*}
So there are three types of critical points: closed Reeb orbits on $\Sigma$
which are positively parametrized
and correspond to $\eta>0$, closed Reeb orbits on $\Sigma$ which
are negatively parametrized and correspond to $\eta<0$, and constant
loops on $\Sigma$ which correspond to $\eta=0$. The action of a critical
point $(x,\eta)$ is $A_H(x,\eta)=\eta$. Because of general homotopy invariance properties of Floer homology, for the purpose of defining Rabinowitz Floer homology we can assume without loss of generality that the Reeb flow on $\Sigma$ is nondegenerate. In this case the critical points of $A_H$ form Morse-Bott manifolds: on the one hand $\Sigma$, and on the other hand one circle for each closed Reeb orbit, parametrized positively or negatively.

Pick a smooth family $(J_t)_{t\in S^1}$ of compatible almost complex
structures on $\wh V$ that are cylindrical at infinity. It induces 
a metric $g=g_J$ on $\mathcal{L}\times \mathbb{R}$ which at a point
$(x,\eta) \in \mathcal{L}\times \mathbb{R}$
and two tangent vectors $(\hat{x}_1,\hat{\eta}_1),(\hat{x}_2,\hat{\eta}_2)
\in T_{(x,\eta)}(\mathcal{L}\times \mathbb{R})=\Gamma(S^1,x^*T\wh V)
\times \mathbb{R}$ is given by
$$
   g_{(x,\eta)}\big((\hat{x}_1,\hat{\eta}_1),(\hat{x}_2,\hat{\eta}_2)\big)
=\int_0^1\omega\big(\hat{x}_1(t),J_t(x(t))\hat{x}_2(t)\big)dt+
\hat{\eta}_1\cdot\hat{\eta}_2.
$$
Positive gradient flow lines of the Rabinowitz action functional
$A_{H}$ with respect to this metric are solutions
$(x,\eta)\in C^\infty(\mathbb{R}\times S^1,\widehat{V}) \times
C^\infty(\mathbb{R},\mathbb{R})$ of the Rabinowitz Floer equation
\begin{equation}\label{eq:RF}
\left\{\begin{array}{c}
\partial_s x+J_t(x)\big(\partial_t x-\eta X_{H}(x)\big)=0,\\
\partial_s\eta + \int_0^1H(x(t))dt=0.
\end{array}\right.
\end{equation}
We fix action values $-\infty<a<b<\infty$ outside the action
spectrum of $A_H$, and we pick an additional small Morse
function $f$ on the critical manifold $\Crit(A_H)$. Thus $f$
consists of a Morse function $f_\Sigma:\Sigma\to\R$ and Morse functions
$f_{\gamma^k}:\mathrm{im}(\gamma)\to\R$ for each simple closed Reeb orbit $\gamma$ and
$k\in\Z\setminus\{0\}$, where we assume the latter to have unique
minima $m_{\gamma^k}$ and maxima $M_{\gamma^k}$.

We fix a commutative unital coefficient ring $R$.
Then $RFC_*^{(a,b)}(H,f;J)$ is the free $R$-module generated by the
critical points of $f$ with action in $(a,b)$, and the boundary operator 
$$
   \p: RFC_*^{(a,b)}(H,f;J)\to RFC_{*-1}^{(a,b)}(H,f;J)
$$
counts cascades as in~\cite{Cieliebak-Frauenfelder}. They combine the {\em negative}
gradient flow of $A_H$ with respect to the metric $g$ and the {\em negative} gradient flow of $f$
with respect to some metric on $\Crit(A_H)$. 
As grading we use the integer grading obtained by shifting up by $\frac 12$ the half-integer grading defined in~\cite{Cieliebak-Frauenfelder}.
The resulting {\em filtered Rabinowitz Floer homology} groups 
$$
   RFH_*^{(a,b)}(\Sigma,\lambda):=RFH_*^{(a,b)}(H,f;J),
$$
are well-defined and do not depend on the choice of $J$, $H$ and $f$,
though they do depend on the contact form $\lambda|_\Sigma$. 
The {\em Rabinowitz Floer homology} of $\Sigma$ is defined as the limit
$$
   RFH_*(\Sigma) := \lim_{\stackrel \longrightarrow b} \lim_{\stackrel
   \longleftarrow a} RFH_*^{(a,b)}(\Sigma,\lambda), \qquad
   a\to-\infty,\ b\to\infty.
$$
By~\cite[Theorem~A]{CF}, this definition is equivalent to the
original one in~\cite{Cieliebak-Frauenfelder}. 
By similar direct-inverse limits one defines as in~\cite{CO} the groups
$$
   RFH_*^\heartsuit(\Sigma),\qquad \heartsuit\in\{\varnothing, >0,\ge 0, =0, \le 0, <0\},
$$
with the meaning that $RFH_*^\varnothing=RFH_*$. 

We define the Rabinowitz Floer cohomology groups by a similar procedure using the dual complex $RFC^*_{(a,b)}(H,f;J)=RFC_*^{(a,b)}(H,f;J)^\vee$. The \emph{filtered Rabinowitz Floer cohomology} groups are defined as 
$$
RFH^*_{(a,b)}(\Sigma,\lambda):= RFH^*_{(a,b)}(H,f;J), 
$$
and the \emph{Rabinowitz Floer cohomology} of $\Sigma$ is the limit 
$$
RFH^*(\Sigma):= \lim_{\stackrel \longrightarrow a} \lim_{\stackrel
   \longleftarrow b}  RFH^*_{(a,b)}(\Sigma,\lambda),
$$
with variants $RFH^*_\heartsuit(\Sigma)$ for $\heartsuit\in\{\varnothing, >0,\ge 0, =0, \le 0, <0\}$. The grading is the same as for homology.  This definition is analogous to~\cite[Definition~3.1]{CO}, given in a symplectic homology framework.

These Rabinowitz Floer co/homology groups depend on the contact structure
on $\Sigma$ (though not on the contact form) as well as the Liouville filling $V$ (up
to Liouville homotopy). 
As in~\cite{CO} we do not indicate the filling in the notation since it will always be clear from the context.

\subsection{Poincar\'e duality}

Poincar\'e duality results from the following observation: 
Under the canonical involution
$$
   {\mathcal L}\times\R \to {\mathcal L}\times\R, \quad
   (x,\eta)\mapsto(\bar x,\bar\eta),\quad \bar x(t)=x(-t),\quad \bar\eta=-\eta
$$
the Rabinowitz action functional changes sign,
$$
   A_H(\bar x,\bar\eta) = -A_H(x,\eta).
$$
It follows that the involution maps positive gradient lines of $A_H$
to negative gradient lines of $A_H$, provided that we also replace the
family $J_t$ by $\bar J_t:=J_{-t}$ (and the resulting metric
accordingly). In other words, if $(x,\eta)\in C^\infty(\mathbb{R}\times
S^1,\widehat{V}) \times 
C^\infty(\mathbb{R},\mathbb{R})$ solves the Rabinowitz Floer
equation~\eqref{eq:RF} with positive asymptotic $(x_+,\eta_+)$ and
negative asymptotic $(x_-,\eta_-)$, then $(\bar x,\bar\eta)$ defined by
$$
   \bar x(s,t):=x(-s,-t),\quad \bar\eta(s):=-\eta(-s)
$$
solves equations~\eqref{eq:RF} with positive asymptote $(\bar x_-,\bar\eta_-)$ and
negative asymptote $(\bar x_+,\bar\eta_+)$. 

When applying the involution, we also replace the Morse functions
$f_\Sigma:\Sigma\to\R$ and $f_{\gamma^k}:\mathrm{im}(\gamma)\to\R$ by
$$
   \bar f_\Sigma:=-f_\Sigma,\quad \bar f_{\gamma^k}:=-f_{\gamma^k}. 
$$
Then the preceding discussion shows that the involution
$(x,\eta)\mapsto(\bar x,\bar\eta)$ defines a chain isomorphism
between the Rabinowitz Floer chain and cochain groups
$$
   RFC_*^{(a,b)}(H,f;J)\stackrel{\cong}\longrightarrow
   RFC^{1-*}_{(-b,-a)}(H,\bar f;\bar J).
$$
Therefore, we have shown

\begin{theorem}[Poincar\'e duality in Rabinowitz Floer homology]\label{thm:PD-RFH}
The involution $(x,\eta)\mapsto(\bar x,\bar\eta)$ induces isomorphisms
between filtered Rabinowitz Floer homology and cohomology groups
$$
   PD: RFH_*^{(a,b)}(\Sigma,\lambda)\stackrel{\cong}\longrightarrow RFH^{1-*}_{(-b,-a)}(\Sigma,\lambda),
$$
and between the Rabinowitz Floer homology and cohomology groups
$$
   PD: RFH_*^\heartsuit(\Sigma)\stackrel{\cong}\longrightarrow RFH^{1-*}_{-\heartsuit}(\Sigma)
$$
for $\heartsuit\in\{\varnothing, >0,\ge 0, =0, \le 0, <0\}$. \hfill \qed
\end{theorem}
 
Given $\heartsuit\in\{\varnothing, >0,\ge 0, =0, \le 0, <0\}$, the meaning of $-\heartsuit$ is that equalities are preserved and inequalities are reversed, e.g., if $\heartsuit=\varnothing$ then $-\heartsuit=\varnothing$, and if $\heartsuit=``\ge 0"$ then $-\heartsuit=``\le 0"$. 

\begin{remark}
With the half-integer grading from~\cite{Cieliebak-Frauenfelder}, Poincar\'e duality would take the more symmetric form $RFH_*\simeq RFH^{-*}$. Our use of the shifted integer grading is motivated by~\eqref{eq:iso-RFH-SH} below, which ensures that Rabinowitz Floer homology as defined in this section is isomorphic to the one from the Introduction. Also, this grading relates more easily to the gradings in singular cohomology and loop space homology. 
For example, with $a=-\eps$, $b=\eps$ and $\eps>0$ smaller than the period of the shortest closed Reeb orbit (``energy 0''), the filtered Poincaré duality isomorphism for Rabinowitz Floer homology from Theorem~\ref{thm:PD-RFH} becomes the usual Poincaré duality isomorphism $H_{*+n-1}(\Sigma)\simeq H^{-*+n}(\Sigma)$.
\end{remark}

\begin{remark} Rabinowitz Floer homology has a Lagrangian version defined in terms of Lagrange multipliers as above, see~\cite{Merry}. Theorem~\ref{thm:PD-RFH} has a straightforward counterpart in that setting. 
\end{remark}

\section{Poincar\'e duality in symplectic homology} \label{sec:PD-SH}

In this section we discuss Poincar\'e duality in (V-shaped) symplectic homology. While being more involved than the one in Rabinowitz Floer homology, the description in symplectic homology has three additional features: it is directly related to loop space (co)homology in the case of cotangent bundles; it is compatible with the pair-of-pants products and coproducts; and it carries non-canonical splittings in the case of cotangent bundles. 

This section is concerned with general Liouville domains and leads to the proof of Theorem~\ref{thm:PD-RFH-intro}.
For general background on symplectic homology we refer to~\cite{CO}. 
We use coefficients in a commutative unital ring $R$.

\subsection{Recollections on Poincar\'e duality and exact sequences} \label{sec:recollections-PD}

The main result in~\cite{Cieliebak-Frauenfelder-Oancea} states that
\begin{equation} \label{eq:iso-RFH-SH}
   RFH_*^\heartsuit(\p V) \cong SH_*^\heartsuit(\p V),
\end{equation}
where $V$ is a Liouville domain of dimension $2n$ and 
$SH_*^\heartsuit(\p V)$ denotes the symplectic homology of the
trivial cobordism $[\frac 1 2,1]\times \p V$ in the sense of~\cite{CO}. 
We assume that $2c_1(V)=0$ and that the square of the canonical bundle of $TV$ is trivialized, so that Floer chain complexes are  canonically $\Z$-graded by the Conley-Zehnder indices of periodic orbits, see Appendix~\ref{sec:grading}. This assumption is used only for grading purposes and has no essential bearing on the sequel arguments. 
Then Theorem~\ref{thm:PD-RFH} has the following formulation in symplectic homology. 

\begin{theorem}[{Poincar\'e duality in symplectic homology of a trivial cobordism~\cite[Theorem 9.4]{CO}}]\label{thm:PD-SH}
There exist canonical isomorphisms between the symplectic homology
and cohomology groups of a trivial cobordism
$$
   PD: SH_*^\heartsuit(\p V)\stackrel{\cong}\longrightarrow SH^{1-*}_{-\heartsuit}(\p V)
$$
for $\heartsuit\in\{\varnothing, >0,\ge 0, =0, \le 0, <0\}$. \hfill\qed
\end{theorem}

This isomorphism is constructed in~\cite{CO} without relying on the
connection between symplectic homology and Rabinowitz Floer
homology. We expect the isomorphisms of Theorems~\ref{thm:PD-RFH}
and~\ref{thm:PD-SH} to be compatible with~\eqref{eq:iso-RFH-SH}. 
In the rest of the paper we will refer to Poincar\'e duality for a
trivial cobordism as being the isomorphism in Theorem~\ref{thm:PD-SH}.   

The Poincar\'e duality isomorphism has the following properties, see~\cite{CO}. 

{\em (A) Compatibility with the exact sequence of the pair $(V,\p V)$.}

\begin{theorem}[{\cite[Theorem 9.5]{CO}}]\label{thm:PD-LES}
For every Liouville domain $V$ and $\heartsuit\in\{\varnothing, >0,\ge
0, =0, \le 0, <0\}$ there exists a commuting diagram
\begin{equation}\label{eq:PD-LES}
{\scriptsize
\xymatrix{
   \dots SH_*^\heartsuit(V,\p V) \ar[r] \ar[d]^\cong_{PD} & SH_*^\heartsuit(V)
   \ar[r] \ar[d]^\cong_{PD} & SH_*^\heartsuit(\p V) \ar[r] \ar[d]^\cong_{PD} &
   SH_{*-1}^\heartsuit(V,\p V) \dots \ar[d]^\cong_{PD} \\
   \dots SH^{-*}_{-\heartsuit}(V) \ar[r] & SH^{-*}_{-\heartsuit}(V,\p
   V) \ar[r] & SH^{1-*}_{-\heartsuit}(\p V) \ar[r] & SH^{1-*}_{-\heartsuit}(V) \dots \\
}
}
\end{equation}
where the rows are the long exact sequences of the pair $(V,\p V)$
from~\cite{CO} and the vertical arrows are the Poincar\'e duality
isomorphisms from Theorem~\ref{thm:PD-SH} (the third one) and for
pairs as defined in~\cite{CO} (the other ones). 
Moreover, the Poincar\'e duality isomorphisms are compatible with 
filtration exact sequences. \hfill\qed
\end{theorem}

{\em (B) Relation to singular cohomology. }
Recall from~\cite{CO} that at action zero symplectic homology specializes to singular
cohomology, 
$$
   SH_*^{=0}(V) \cong H^{n-*}(V),
$$
and similarly for the other versions. Therefore, we obtain

\begin{corollary}[{\cite[Corollary 9.7]{CO}}]\label{cor:PD-LES-zero}
The commuting diagram in Theorem~\ref{thm:PD-LES} specializes at
action zero to

\begin{equation}\label{eq:PD-LES-zero}
{\scriptsize
\xymatrix{
   \dots H^{n-*}(V,\p V) \ar[r] \ar[d]^\cong_{PD} & H^{n-*}(V)
   \ar[r] \ar[d]^\cong_{PD} & H^{n-*}(\p V) \ar[r] \ar[d]^\cong_{PD} &
   H^{n-*+1}(V,\p V) \dots \ar[d]^\cong_{PD} \\
   \dots H_{n+*}(V) \ar[r] & H_{n+*}(V,\p
   V) \ar[r] & H_{n+*-1}(\p V) \ar[r] & H_{n+*-1}(V) \dots \\
}
}
\end{equation}

where the rows are the long exact sequences of the pair $(V,\p V)$
and the vertical arrows are the Poincar\'e duality
isomorphisms for the closed manifold $\p V$ (the third one) and the
manifold-with-boundary $V$ (the other ones). \hfill\qed
\end{corollary}

{\em (C) Description of the first map in~\eqref{eq:PD-LES}. }
The following result is simply a restatement of~\cite[Proposition 1.3]{Cieliebak-Frauenfelder-Oancea}.

\begin{proposition} \label{prop:map}

(a) The map $SH_*(V,\p V)\to SH_*(V)$ in~\eqref{eq:PD-LES} (for
$\heartsuit=\varnothing$) fits into a commutative diagram
\begin{equation} \label{eq:map}
\xymatrix
@R=15pt
{
SH_*(V,\p V) \ar[r] \ar[d]_{c^*} & SH_*(V) \\
H^{n-*}(V,\partial V) \ar[r] & H^{n-*}(V) \ar[u]_{c_*}
}
\end{equation}
in which the bottom arrow is the restriction map and the vertical arrows
are compositions of action zero isomorphisms with action truncation maps 
\begin{gather*}
   c_*: H^{n-*}(V)\cong SH_*^{=0}(V)\to SH_*^{\geq 0}(V)=SH_*(V), \cr
   c^*: SH_*(V,\p V)=SH_*^{\leq 0}(V,\p V)\to SH_*^{=0}(V,\p V)\cong H^{*+n}(V,\partial V).
\end{gather*}

(b) The map $SH^{-*}(V)\to SH^{-*}(V,\p V)$ in~\eqref{eq:PD-LES} (for
$\heartsuit=\varnothing$) fits into a commutative diagram
\begin{equation} \label{eq:map-coh}
\xymatrix
@R=15pt
{
SH^{-*}(V) \ar[r] \ar[d]_{c^*} & SH^{-*}(V,\p V) \\
H_{n+*}(V) \ar[r] & H_{n+*}(V,\p V) \ar[u]_{c_*}
}
\end{equation}
in which the bottom arrow is induced by inclusion and the vertical arrows
are compositions of action zero isomorphisms with action truncation maps 
\begin{gather*}
   c_*: H_{n+*}(V,\p V)\cong SH^{-*}_{=0}(V,\p V)\to SH^{-*}_{\leq 0}(V,\p V)=SH^{-*}(V,\p V), \cr
   c^*: SH^{-*}(V)=SH^{-*}_{\geq 0}(V)\to SH^{-*}_{=0}(V)\cong H_{n+*}(V).
\end{gather*} 

(c) The above two diagrams are isomorphic via the Poincar\'e duality isomorphism $PD$. 
\hfill\qed
\end{proposition}

\begin{remark}[Lagrangian case] The previous results 
have Lagrangian counterparts, with symplectic homology replaced by Lagrangian symplectic homology, or wrapped Floer homology, cf.~\cite{CO}. We spell out some of these statements, mainly with the purpose of explaining the effect of the grading convention described in Appendix~\ref{sec:grading}.  

Given an exact $n$-dimensional Lagrangian $L$ which is conical near its boundary $\p L$, the Poincar\'e duality isomorphisms from~\cite{CO} read
$$
SH_*(L,\p L)\simeq SH^{n-*}(L),
$$
and
\begin{equation} \label{eq:PD-Lag-SH}
SH_*(\p L)\simeq SH^{n-*+1}(\p L). 
\end{equation}
This fixes the grading for the Lagrangian counterpart of~\eqref{eq:PD-LES}. 

The action zero part of Lagrangian symplectic homology and cohomology is expressed in topological terms as 
$$
SH_*^{=0}(\cL)\simeq H^{n-*}(\cL), \qquad SH^*_{=0}(\cL)\simeq H_{n-*}(\cL),
$$
where $\cL$ stands for any of the symbols $L$, $(L,\p L)$, or $\p L$. This fixes the grading for the Lagrangian counterpart of~\eqref{eq:PD-LES-zero}. 

The Lagrangian analogue of~\eqref{eq:map} is 
$$
\xymatrix
@R=15pt
{
SH_*(L,\p L) \ar[r] \ar[d]_{c^*} & SH_*(L) \\
H^{n-*}(L,\partial L) \ar[r] & H^{n-*}(L) \ar[u]_{c_*}.
}
$$
In particular, if $L$ is a disc the bottom map vanishes for degree reasons, hence the map $SH_*(L,\p L)\to SH_*(L)$ vanishes as well. 
\end{remark}

\subsection{TQFT operations on Floer homology}\label{sec:TQFT-Floer}

In this subsection we recall the definition of TQFT operations on
Hamiltonian Floer homology from~\cite{Seidel07,Ritter}, see
also~\cite{CO,Ekholm-Oancea}. 
Consider a punctured Riemann surface $S$ with $p$ negative and $q$
positive punctures and chosen cylindrical ends
$Z_i^-=(-\infty,0]\times S^1$ near the negative punctures $z_i^-$,
resp.~$Z_j^+=[0,\infty)\times S^1$ near the positive punctures $z_j^+$. 
Let $H:S\times\wh V\to\R$ be an $S$-dependent 
Hamiltonian on a completed Liouville domain $\wh V$ which is linear
outside a compact subset of $\wh V$, and $S$-independent equal to
$H_\ell^\pm$ near each puncture $z_\ell^\pm$ of $S$. Pick 
nonzero weights $A_\ell^\pm\in\R^*$ (one usually considers positive weights, but we will also need to allow negative ones)
and a $1$-form $\beta$ on $S$ with the following properties: 
\begin{enumerate}
\item[(i)] $d_S(H\beta)\leq 0$;
\item[(ii)] $\beta=A_\ell^\pm dt$ in cylindrical coordinates $(s,t)\in
  Z_\ell^\pm$ near the puncture $z_\ell^\pm$.
\end{enumerate}
We consider maps $u:S\to\wh V$ that are perturbed holomorphic in the sense that
$(du-X_H\otimes\beta)^{0,1}=0$ and have finite energy
$E(u)=\frac{1}{2}\int_S|du-X_H\otimes\beta|^2{\rm vol}_S$. They
converge at the punctures to $1$-periodic orbits
$x_\ell^\pm$ of $A_\ell^\pm H_\ell^\pm$ and satisfy the energy estimate
\begin{equation}\label{eq:energy}
   0 \leq E(u) \leq \sum_{j=1}^qA_{A_j^+H_j^+}(x_j^+) -
   \sum_{i=1}^pA_{A_i^-H_i^-}(x_i^-).
\end{equation}
The signed count of such maps yields an operation
$$
   \psi_S:\bigotimes_{j=1}^qFH_*(A_j^+H_j^+) \to
   \bigotimes_{i=1}^pFH_*(A_i^-H_i^-).  
$$
of degree $n(2-2g-p-q)$ which does not increase the action. These
operations are graded commutative if the degrees are shifted by $-n$ and
satisfy the usual TQFT composition rules. As they respect the action,
the operations descend to operations between suitable filtered Floer
homology groups and thus to sympletic homology. 

Suppose now that $H:\wh V\to\R$ is $S$-independent. 
Then $\beta$ and the weights are related by Stokes' theorem
$$
   \sum_{j=1}^qA_j^+ - \sum_{i=1}^pA_i^- = \int_Sd\beta. 
$$
Conversely, if the quantity on the left-hand side is nonnegative (resp. zero,
resp. nonpositive), then we find a $1$-form $\beta$ with properties (i) and
(ii) such that $d\beta\geq 0$ (resp.~$=0$, resp.~$\leq0$). Thus for
$S$-independent $H$ we can arrange
conditions (i)--(ii) in the following situations:
\begin{enumerate}
\item[(a)] $H$ arbitrary, $d\beta\equiv 0$, $p,q\geq 1$;
\item[(b)] $H\geq 0$, $d\beta\leq 0$, $p\geq 1$;
\item[(c)] $H\leq 0$, $d\beta\geq 0$, $q\geq 1$.
\end{enumerate}
Let us now specialize to the case that $S$ is a pair-of-pants with two
positive punctures and one negative puncture. Then the operation
$\psi_S$ induces a pair-of-pants product on filtered Floer homology
\begin{align*}\psi_S:FH_i^{(a_1,b_1)}(A_1^+H)\otimes FH_j^{(a_2,b_2)} & (A_2^+H) \\
   & \longrightarrow
   FH_{i+j-n}^{(\max\{a_1+b_2,a_2+b_1\},b_1+b_2)}(A_1^-H). 
\end{align*}
By taking suitable inverse and direct limits (see~\cite{CO}) this leads to degree $-n$ pair-of-pants products on $SH_*(V)$, $SH_*(V,\p V)\cong SH^{-*}(V)$, $SH_*(\p V)$, and $SH_*(W,\p W)\cong SH^{-*}(\p V)$, where $W=[\frac12,1]\times\p V$. 

More generally, one can consider open-closed 
TQFT operations in Floer homology, which mix together Hamiltonian Floer homology inputs and outputs and Lagrangian Floer homology inputs and outputs. The discussion about weights carries over without modification, and we refer to Appendix~\ref{sec:grading} for a discussion of gradings in this setting. 

\subsection{Products and the mapping cone} \label{sec:prod-algebraic} 

We consider in this section an algebraic setup that will help us organize the subsequent geometric arguments regarding the compatibility of the Poincar\'e duality isomorphism with product structures.  

Consider a chain complex $(C,\p)$ of the form 
$$
   C=C^-\oplus C^+,\qquad \p = \left(\begin{array}{cc}\p^- & f \\0 & \p^+ \end{array}\right).
$$
Thus $(C^\pm,\p^\pm)$ are chain complexes, $C^-$ is a subcomplex of $C$, and $C$ is the cone of the chain map $f:C^+[1]\to C^-$. We use the conventions of~\cite{CO} for cones and degree shifts: $C^+[1]_i=C_{i+1}$ and $\p_{C^+[1]}=-\p^+$. 

Assume now that {\em $C$ is acyclic} as a consequence of $\mathrm{Id}_C$ being homotopic to zero, i.e., there exists $K:C\to C[1]$ such that 
$$
   \mathrm{Id}_C=\p K + K\p. 
$$
Writing $K$ with respect to the decomposition $C=C^-\oplus C^+$ as 
$$
  K=\left(\begin{array}{cc}k^- & h \\g & k^+ \end{array}\right),
$$
this is equivalent to the system of four equations 
$$
\left\{\begin{array}{rcl}
\p^+g+g\p^- & = & 0, \\
\p^-k^- + fg + k^-\p^- & = & \mathrm{Id}_{C^-}, \\
\p^+k^+ + gf + k^+\p^+ & = & \mathrm{Id}_{C^+}, \\
\p^-h + fk^+ + k^-f + h\p^+ & = & 0.
\end{array}\right.
$$
The first three equations amount to the fact that 
$$
   g:C^-\to C^+[1]
$$
is a chain map which is a homotopy inverse for $f$, the homotopies between $fg$ and $gf$ and the corresponding identity maps being given by 
$$
   k^-:C^-\to C^-[1],\qquad k^+:C^+\to C^+[1]. 
$$
The fourth equation gives extra information which we will not use.

Let now $(\ol C,\ol\p)$ be another chain complex of the form
$$
   \ol C=\ol C^-\oplus \ol C^+,\qquad \ol\p = \left(\begin{array}{cc}\ol\p^- & \ol f \\0 & \ol\p^+ \end{array}\right).
$$
The complex $\ol C$ need not be acylic, though an important special case is $\ol C=C$. 
In the following we will adopt the following
conventions for a linear map $\phi:C\otimes C\to\ol C$: 
\begin{itemize}
\item $\phi^{+-}_+$ is the part of $\phi$ mapping $C^+\otimes C^-\to
  \ol C^+$, etc;
\item we abbreviate $\phi^+:=\phi^{++}_+$ and $\phi^-:=\phi^{--}_-$;
\item $[\p,\phi]:=\ol\p\circ\phi - (-1)^{|\phi|}\phi\circ(\p\otimes 1+1\otimes\p)$.
\end{itemize}
Suppose now that we are given a ``product'' $\mu:C\otimes C\to\ol C$ of
degree zero satisfying
$$
   [\p,\mu]=0 \quad\text{and}\quad \mu^{--}_+ = \mu^{-+}_+ = \mu^{+-}_+ = 0.
$$
Thus $\mu$ descends to homology and the pair of subcomplexes $(C^-,\ol C^-)$ is a
``two-sided ideal pair'' with respect to $\mu$.
\footnote{I.e., whenever one of the inputs of $\mu$ is in $C^-$, the output is in $\ol C^-$.}  
It follows that 
$$
   0=[\p,\mu]^{--}_-=[\p^-,\mu^{--}_-] \quad\text{and}\quad 0=[\p,\mu]^{++}_+=[\p^+,\mu^{++}_+],
$$
so $\mu^-=\mu^{--}_-:C^-\otimes C^-\to\ol C^-$ and
$\mu^+=\mu^{++}_+:C^+\otimes C^+\to\ol C^+$ descend to products on homology. 

\begin{remark}
The case of a product of any degree can always be reduced to the
degree zero case by a shift of the grading on $C$ and $\ol C$.
\end{remark}

The following lemma is a key result for this section.

\begin{lemma}\label{lem:secondary-product}
For $(C,\p)$, $K$, $(\ol C,\ol\p)$ and $\mu$ as above the following hold. 

(a) The negative parts $\nu^-,\wt\nu^-:C^-\otimes C^-\to\ol C^-[1]$ of the
  maps $\nu:=\mu(K\otimes 1):C\otimes C\to\ol C[1]$ and
  $\wt\nu:=\mu(1\otimes K)$ satisfy
$$
   \mu^- = [\p^-,\nu^-] = [\p^-,\wt\nu^-].  
$$
In particular, the {\em primary product} $\mu^-$ vanishes on the homology of $C^-$. 

(b) The negative part $\sigma^-:C^-[-1]\otimes C^-[-1]\to\ol C^-[-1]$ of the
difference $\sigma:=\wt\nu-\nu$ satisfies $[\p^-,\sigma^-]=0$ and thus
descends to a {\em secondary product} on the homology of $C^-[-1]$. 

(c) The negative part $\eta^-=\eta^{--}_-:C^-[-1]\otimes C^-[-1]\to\ol C^-$ of
$\eta:=\mu(K\otimes K)$ satisfies
$$
   [\p^-,\eta^-] = \sigma^--\ol f\mu^+(g\otimes g).
$$
Thus the secondary product $\sigma^-$ agrees on homology with the
product $\mu^+$ transferred from $C^+$ to $C^-[-1]$ via the chain map $\ol f$
and the homotopy inverse $g$ of $f$.
\end{lemma}

\begin{proof}
For part (a) we compute, using $[\p,\mu]=0$ and $[\p,K]=1$,
\begin{align*}
   [\p,\nu] 
   &= \ol\p\mu(K\otimes 1) + \mu(K\otimes 1)(\p\otimes 1+1\otimes\p) \cr
   &= \mu\bigl([\p\otimes 1+1\otimes\p,K\otimes 1]\bigr) \cr
   &= \mu\bigl([\p,K]\otimes 1\bigr) = \mu.
\end{align*}
We now use $\nu^{--}_+=0$ and then take the $C^-$ part to get
$$
   \mu^- = [\p,\nu]^{--}_- = [\p^-,\nu^-].
$$
Part (b) follows directly from part (a). 
For part (c) we compute, using $[\p,\mu]=0$ and $[\p,K]=1$,
\begin{align*}
   [\p,\eta] 
   &= \ol\p\mu(K\otimes K) + \mu(K\otimes K)(\p\otimes 1+1\otimes\p) \cr
   &= \mu\bigl([\p\otimes 1+1\otimes\p,K\otimes K]\bigr) \cr
   &= \mu\bigl([\p,K]\otimes K-K\otimes[\p,K]\bigr) \cr
   &= \mu(1\otimes K-K\otimes 1).
\end{align*}
Taking the $C^-$ part the right hand side becomes
$\wt\nu^--\nu^-=\sigma^-$, while the left hand side becomes
$$
   [\p,\eta]^{--}_- = [\p^-,\eta^-] + \ol f\eta^{--}_+ = [\p^-,\eta^-] +
   \ol f\mu^+(g\otimes g).
$$
\end{proof}

\subsection{Poincar\'e duality with products}\label{sec:PD-products}

As in the previous sections $V$ is a Liouville domain of dimension $2n$, and $W=[\frac 1 2, 1]\times \p V\subset V$ is the trivial cobordism realized by a collar neighbourhood of $\p V$ in $V$. Recall from \S\ref{sec:TQFT-Floer} that $SH_*(W,\p W)$ carries a product of degree $-n$ determined by counts of rigid pairs-of-pants in combination with suitable action truncations. We refer to this product as the \emph{primary product on $SH_*(W,\p W)$}. By Poincar\'e duality we have $SH_*(W,\p W)\simeq SH^{-*}(W)$, and we refer to the corresponding degree $n$ pair-of-pants product on $SH^*(W)$ as the \emph{primary product on $SH^*(W)=SH^*(\p V)$}. 

The following result corresponds to Theorem~\ref{thm:PD-RFH-intro} from the Introduction. Recall that $SH_*(\p V)$ carries a unital product $\boldmu$ of degree $-n$.  

\begin{theorem} \label{thm:coh-product}
Let $V$ be a Liouville domain of dimension $2n$. Then:

(a) The primary product on $SH^*(\p V)$ vanishes. As a consequence, $SH^*(\p V)$ carries a degree $n-1$ secondary product $\boldlambda^\vee \tau$ which is associative, graded commutative, and unital (if degrees are shifted by $1-n$). 

(b) The Poincar\'e duality isomorphism 
$$
   PD:SH_*(\p V)\stackrel\simeq \longrightarrow SH^{1-*}(\p V)
$$
is a ring homomorphism, where $SH_*(\p V)$ is endowed with its degree $-n$ pair-of-pants product $\boldsymbol{\mu}$ and $SH^*(\p V)$ is endowed with its degree $n-1$ secondary product $\boldsymbol{\lambda}^\vee \tau$. 
\end{theorem}

Part (b) shows that, in contrast to the secondary product on $SH^*_{>0}(V)$ defined in~\cite{CO}, the secondary product on $SH^*(\p V)$ has a unit. 

It is interesting to note that our proof of (a) is inseparable from the proof of (b). 
In particular, unitality of $\boldlambda^\vee \tau$ is only implicitly inferred from unitality of $\boldmu$ via the isomorphism in (b). We refer to~\cite{CO-cones} for an alternative, more direct description of the unit for $\boldlambda^\vee \tau$ within the framework of multiplicative structures on cones.

Before giving the proof of Theorem~\ref{thm:coh-product}, we state its Lagrangian counterpart. 
We consider Maslov 0 exact Lagrangians $L\subset V$ with boundary such that $\p L=L\cap \p V$ and $L$ is conical near $\p L$. 

\begin{theorem} \label{thm:coh-product-Lag}
Let $V$ be a Liouville domain of dimension $2n$ and $L\subset V$ be an exact Lagrangian as above. Then:

(a) The primary product on $SH^*(\p L)$ vanishes. As a consequence, $SH^*(\p L)$ carries a secondary product $\boldsymbol{\lambda}^\vee \tau$ of degree $-1$, which has a unit in degree $1$. 

(b) The Poincar\'e duality isomorphism 
$$
   PD:SH_*(\p L)\stackrel\simeq \longrightarrow SH^{n-*+1}(\p L)
$$
is a ring homomorphism, where $SH_*(\p L)$ is endowed with its degree $-n$ pair-of-pants product $\boldsymbol{\mu}$ and $SH^*(\p L)$ is endowed with its degree $-1$ secondary product $\boldsymbol{\lambda}^\vee \tau$. 
\end{theorem}

\begin{proof}
The proof is up to notation the same as that of Theorem~\ref{thm:coh-product}. Gradings are discussed in  Appendix~\ref{sec:grading}. 
\end{proof}

Before giving with the proof of Theorem~\ref{thm:coh-product}, it is useful to recall the definition of $SH_*(\p V)$. 
Denote the radial coordinate in the conical part of the symplectic completion $\wh V=V\cup [1,\infty)\times \p V$ by $r\in[1,\infty)$. Rabinowitz Floer homology $SH_*(\p V)$ is defined using the family  
$$
\cH=\{H_{\lambda,\mu}\, : \, \lambda,\mu\in\R\}
$$ 
of Hamiltonians on $\wh V$, equal to $0$ on $\p V$, linear in $r$ of slope $\mu$ on $[1,\infty)\times \p V$, linear in $r$ of slope $\lambda$ on $[1/2,1]\times \p V$, and constant equal to $-\lambda/2$ on $\{r\le 1/2\}$. See Figure~\ref{fig:Ham_RFH}.

\begin{figure} [ht]
\centering
\input{Ham_RFH.pstex_t}
\caption{Hamiltonian profile for $SH_*(\p V)$.}
\label{fig:Ham_RFH}
\end{figure}

Given $-\infty<a<b<\infty$ we define
$$
SH_*^{(a,b)}(\p V)=\lim^{\longrightarrow}_{\mu\to\infty,\, \lambda\to-\infty}FH_*^{(a,b)}(H_{\lambda,\mu}),
$$
and further 
$$
SH_*(\p V)=\lim^{\longrightarrow}_{b\to\infty}\lim^{\longleftarrow}_{a\to-\infty}SH_*^{(a,b)}(\p V).
$$

\begin{proof}[Proof of Theorem~\ref{thm:coh-product}]
Let $\wh V$ be the completion of $V$, with a canonical embedding of the symplectization $(0,\infty)\times\p V\subset \wh V$. Under this embedding the level $\{1\}\times\p V$ is canonically identified with $\p V$, and the restriction of this embedding to $(0,1]\times \p V$ takes values in $V$. Denote $V_\delta= \wh V\setminus (\delta,\infty)\times \p V$ for $\delta>0$. 

Key to the proof is the following construction.  Given $\tau>0$ define the function $\ell_{-\tau}:(0,\infty)\to\R$ constant equal to $3\tau/4$ on the interval $(0,\frac{1}{4}]$ and linear of slope $-\tau$ on $[\frac{1}{4},\infty)$, so that $\ell_{-\tau}(1)=0$. 
Fix now parameters $\tau'>\tau>0$ and $0<\epsilon<\frac 1 2$. (In the constructions that follow the parameter $\tau'$ will eventually be chosen very close to $\tau$.) Let $\epsilon'=\epsilon(\tau'-\tau)/(\tau'+\tau)$, so that $0<\epsilon'<\epsilon$. Consider the continuous piecewise linear function $h=h_{-\tau',\tau,\epsilon,\epsilon'}:(0,\infty)\to\R$ defined by the following conditions:
\begin{itemize}
\item $h$ coincides with $\ell_{-\tau}$ on $(0,1-\epsilon]$. 
\item $h$ is linear of slope $-\tau'$ on the interval $[1-\epsilon,1-\epsilon']$;
\item $h$ is linear of slope $\tau$ on the interval $[1-\epsilon',1]$;
\item $h(1)=0$;
\item $h$ coincides with $\ell_{-\tau}$ on the interval $[1,\infty)$. 
\end{itemize}
We call $h=h_{-\tau',\tau,\epsilon,\epsilon'}$ \emph{the $(-\tau',\tau,\epsilon,\epsilon')$-dent on the function $\ell_{-\tau}$}. See Figure~\ref{fig:dent}. 
\begin{figure} [ht]
\centering
\input{dent.pstex_t}
\caption{Hamiltonian profile $h_{-\tau',\tau,\epsilon,\epsilon'}$ for Poincar\'e duality.}
\label{fig:dent}
\end{figure}
Let $\tilde h$ be a smoothing of $h$ with the following properties: 
\begin{itemize}
\item Outside a small neighbourhood of $r=\frac 1 4$ and a small neighbourhood of the closed interval $[1-\epsilon,1]$, the function $\tilde h$ coincides with $h$. In particular, $\tilde h$ is constant equal to $3\tau/4$ for $r\le \frac 1 4-\delta$ and is linear of slope $-\tau$ for $r\in[\frac 1 4+\delta,1-\epsilon-\delta]$ and $r\ge 1+\delta$ for some small $\delta>0$.
\item Inside the neighbourhood of the closed interval $[1-\epsilon,1]$, and outside small neighbourhoods of $r=1-\epsilon$, $r=1-\epsilon'$, and $r=1$, the function $\tilde h$ is linear of negative slope $-\tau'_1$ smaller but close to $-\tau'$, then linear of positive slope $\tau_1$ larger but close to $\tau$.  
\item Inside the neighbourhood of $r=\frac 1 4$ where its derivative lies in $(-\tau,0)$, the function $\tilde h$ is strictly concave. 
\item Inside the neighbourhood of $r=1-\epsilon$ where its derivative lies in $(-\tau'_1,-\tau)$, the function $\tilde h$ is strictly concave. 
\item In a neighbourhood of $r=1-\epsilon'$ the function $\tilde h$ is constant equal to $-\epsilon'\tau$ on some open interval, and is strictly convex as its slope varies in $(-\tau'_1,0)$ and $(0,\tau_1)$.
\item In a neighbourhood of $r=1$ the function $\tilde h$ is constant equal to $0$ on some open interval, and is strictly concave as its slope varies in $(0,\tau_1)$ and $(-\tau,0)$.  
\end{itemize}

A function $\tilde h$ as above can be interpreted as a Hamiltonian $\tilde H:\wh V\to\R$ by extending it as constant equal to $3\tau/4$ over $\wh V\setminus (0,\infty)\times\p V$. 

Assume now that $\tau$ does not belong to the action spectrum of $\p V$. Since the latter is a closed set in $\R$, we can choose the other parameters $\tau',\tau_1,\tau'_1$ such that $\tau<\tau_1<\tau'<\tau'_1$ and such that the whole interval $[\tau,\tau'_1]$ does not intersect the spectrum. (In our later proof of Proposition~\ref{prop:c-iso} we will have to diminish even further the parameter $\tau'_1>\tau$.) The parameter $\epsilon$ is chosen arbitrarily, while the parameter $\epsilon'$ is determined by $\epsilon$, $\tau$, and $\tau'$, and is typically much smaller than $\epsilon$. The $1$-periodic orbits of $\tilde H$ fall then into three groups. 
\begin{itemize}
\item Group~I consists of constants in $V_{\frac 1 4}$ and nonconstant orbits in the concavity region near $r=\frac 1 4$. 
\item Group~II consists of orbits located in a neighbourhood of $r=1-\epsilon'$, and these are themselves of three types: constants on the trivial cobordism which constitutes the minimal level (type $II^0$), nonconstant orbits in the convexity region near its negative boundary (type $II^-$), nonconstant orbits in the convexity region near its positive boundary (type $II^+$). 
\item Group~III consists of orbits located in a neighbourhood of $r=1$, and these are again of three types: constants on the trivial cobordism which constitutes the maximal level (type $III^0$), nonconstant orbits in the concavity region of positive slope near its negative boundary (type $III^-$), nonconstant orbits in the concavity region of negative slope near its positive boundary (type $III^+$). 
\end{itemize}

Let $H=H_\tau$ be a $C^2$-small perturbation of $\tilde H$, time-dependent and supported near the nonconstant periodic orbits in the concavity or convexity regions, time-independent Morse in the flat regions, and such that in each flat region the gradient of the Morse perturbation is pointing outwards along the boundary if the latter is adjacent to a convexity region, and is pointing inwards along the boundary if the latter is adjacent to a concavity region. The orbits of $H$ naturally fall into classes $I$, $II$, $III$ as above, and their action is close to the action of the corresponding orbits of $\tilde H$.

Let $(a,b)$ be a fixed action window with $-\infty<a<0<b<\infty$, and assume $\tau$ is large enough so that $-3\tau/4 < a$. Then the action of all orbits in group $I$ lies below the action window $(a,b)$. Denote by
$$
   (C_{a,b},\p) := FC_*^{(a,b)}(H)
$$
the Floer complex in the action window $(a,b)$. Note that, up to canonical chain homotopy equivalence, this complex does not depend on $H=H_\tau$ as long as $\tau>\max\{4|a|/3,b\}$.  
The generators of $C_{a,b}$ are orbits of types $II$ and $III$, and we can write 
$$
   C_{a,b} = C_{a,b}^-\oplus C_{a,b}^+,
$$
where $C_{a,b}^-$ is the submodule generated by orbits of type $III$, and $C_{a,b}^+$ is the submodule generated by orbits of type $II$. 

It follows from~\cite[Lemmas~2.2, 2.3, and 2.5]{CO} that $C_{a,b}^-$ is a subcomplex if the size of the perturbation $H-\tilde H$ is small enough. In particular we can think of $C_{a,b}$ as the cone $C(f)$, where $f:C_{a,b}^+\to C_{a,b}^-[-1]$ is the part of the differential which maps elements of $C_{a,b}^+$ to elements of $C_{a,b}^-$. 

The Hamiltonian $H$ is homotopic by a monotone homotopy supported in an open neighbourhood of the region $\{1-\epsilon\le r\le 1\}$ to the Hamiltonian $L=L_\tau$ which coincides with $H$ in an open neighbourhood of $V_{\frac 1 4}$ and which is linear of slope $-\tau$ outside that neighbourhood. This homotopy defines a chain map $FC_*^{(a,b)}(H)\to FC_*^{(a,b)}(L)=0$. 

Choose now $\epsilon>0$ small enough so that the Hamiltonians $H$ and
$L$ are $C^0$-close. By~\cite[Lemma 7.2]{CO}, the reverse homotopy from $L$ to $H$ then induces a chain map $0=FC_*^{(a,b)}(L)\to FC_*^{(a,b)}(H)$ which is a homotopy inverse for the previous map. The homotopy of homotopies between the composition of these two homotopies and the constant homotopy for $H$ induces a chain homotopy 
$$
   K:C_{a,b}\to C_{a,b}[1]
$$
such that 
$$
   \mathrm{Id}_{C_{a,b}}= [\p,K]. 
$$
It follows from the discussion in \S\ref{sec:prod-algebraic} and from the definition of symplectic (co)homology in~\cite{CO} that the map $f:C_{a,b}^+\to C_{a,b}^-[-1]$ above induces on homology the Poincar\'e duality isomorphism
$$
   f_*:SH_*^{(a,b)}(\p V)\stackrel{\cong}\longrightarrow SH^{1-*}_{(-b,-a)}(\p V)
$$
from symplectic homology to cohomology in truncated action. (Alternatively, this follows from~\cite{CO-cones}.)
Consider now the pair-of-pants product
$$
   C_{a,b}\otimes C_{a,b}\to FC_*^{(a+b,2b)}(2H)
$$
as in \S\ref{sec:TQFT-Floer}, where we use a $1$-form $\beta$
satisfying $d\beta=0$ with positive weights $1$ and negative weight $2$. 
The Hamiltonian $2H$ is of the type $H_{2\tau}$, $C^0$-close to the corresponding linear Hamiltonian $L_{2\tau}$. 
Thus $FC_*^{(a+b,2b)}(2H)$ can be identified with $C_{a+b,2b}$ and we obtain a degree $-n$ product
$$
   \mu:C_{a,b}\otimes C_{a,b}\to C_{a+b,2b}. 
$$
Note that the target also splits as 
$C_{a+b,2b}=C_{a+b,2b}^-\oplus C_{a+b,2b}^+$, with $C_{a+b,2b}^-$ a subcomplex, 
and in the notation of~\S\ref{sec:prod-algebraic} we have 
$$
   [\p,\mu]=0\qquad \mbox{and}\qquad \mu_+^{--}=\mu_+^{-+}=\mu_+^{+-}=0,
$$
again as a consequence of~\cite[Lemmas~2.2, 2.3, and 2.5]{CO}. Hence we are in the situation of \S\ref{sec:prod-algebraic} with $C=C_{a,b}$ and $\ol C=C_{a+b,2b}$. In the notation of that section, it follows that $\mu^+=\mu^{++}_+$ descends to a product on homology
$$
   \mu^+_*:SH_*^{(a,b)}(\p V)\otimes SH_*^{(a,b)}(\p V)\stackrel{[-n]}\longrightarrow SH_*^{(a+b,2b)}(\p V).
$$
Moreover, Lemma~\ref{lem:secondary-product} implies:

(a) The primary product $\mu^-=\mu^{--}_-$ vanishes on homology (cohomology after dual identification), i.e.
$$
   \mu^-_*=0:SH^{-*}_{(-b,-a)}(\p V)\otimes SH^{-*}_{(-b,-a)}(\p V)\to SH^{-*}_{(-2b,-a-b)}(\p V).
$$
(b) The vanishing of the primary product in two ways gives rise to a secondary product 
$$
   \sigma^-_*:SH^{1-*}_{(-b,-a)}(\p V)\otimes SH^{1-*}_{(-b,-a)}(\p V)\stackrel{[-n]}\longrightarrow SH^{1-*}_{(-2b,-a-b)}(\p V).
$$
(c) The products $\mu^+_*$ and $\sigma^-_*$ on homology are related via the Poincar\'e duality isomorphism $f_*$ and its inverse $g_*$ as
$$
\xymatrix
@C=30pt
@R=20pt
{
   SH^{(a,b)}_*(\p V)\otimes SH^{(a,b)}_*(\p V) \ar[r]^{\ \ \ \ \ \ \mu^+_*} & SH^{(a+b,2b)}_*(\p V) \ar[d]^{f_*}_\cong \\
   SH_{(-b,-a)}^{1-*}(\p V)\otimes SH_{(-b,-a)}^{1-*}(\p V) \ar[u]_{g_*\otimes g_*}^\cong \ar[r]^{\ \ \ \ \ \ \sigma^-_*} & SH_{(-2b,-a-b)}^{1-*}(\p V).
}
$$
By construction, the maps in this diagram are compatible with the action filtrations, so for $a<a'$ and $b<b'$ the diagram is related to the corresponding diagram for $a',b'$ via the action truncation maps
$$
   SH^{(a,b)}_*(\p V)\to SH^{(a',b')}_*(\p V),\qquad SH_{(-b,-a)}^{1-*}(\p V)\to SH_{(-b',-a')}^{1-*}(\p V).
$$
Passing first to the inverse limit as $a\to-\infty$ and then to the direct limit as $b\to\infty$, we obtain a degree $-n$ product $\boldsymbol{\mu}=\mu^+_*$ on $SH_*(\p V)$ and a degree $n-1$ secondary product $\boldsymbol{\lambda}^\vee\tau=\sigma^-_*$ on $SH^*(\p V)$ which are related via the Poincar\'e duality isomorphism $PD=f_*$ as
$$
\xymatrix
@C=30pt
@R=20pt
{
   SH_*(\p V)\otimes SH_*(\p V) \ar[r]^{\ \ \ \ \ \ \boldsymbol{\mu}} \ar[d]_{PD\otimes PD}^\cong & SH_*(\p V) \ar[d]^{PD}_\cong \\
   SH^{1-*}(\p V)\otimes SH^{1-*}(\p V)  \ar[r]^{\ \ \ \ \ \ \ \ \ \boldsymbol{\lambda}^\vee\tau} & SH^{1-*}(\p V).
}
$$
This concludes the proof of Theorem~\ref{thm:coh-product}.
\end{proof}

\subsection{Poincar\'e duality with coproducts}\label{sec:PD-coproducts}

We keep the setup of~\S\ref{sec:PD-products}, with $V$ a Liouville domain of dimension $2n$. The degree $-n$ product on $SH_*(\p V)$ can be rephrased as a degree $n$ coproduct on $SH^*(\p V)$, see~\S\ref{sec:TQFT-Floer}. On the other hand $SH_*(\p V)$ carries a coproduct of degree $-n$ determined by counts of rigid pairs-of-pants in combination with suitable action truncations. We refer to this coproduct as the \emph{primary coproduct on $SH_*(\p V)$}. 

In \S\S\ref{sec:PD-coproducts}--\ref{ss:ord-hom} we work with coefficients in a field, so that the K\"unneth isomorphism holds. 
This ensures that chain-level coproducts $C\to C\otimes C$, which a priori induce in homology maps $H(C)\to H(C\otimes C)$, factor through coproducts $H(C)\to H(C)\otimes H(C)$. All our formulas actually hold with arbitrary coefficients if the targets of the homological coproducts are set to $H(C\otimes C)$ instead of $H(C)\otimes H(C)$.

\begin{theorem} \label{thm:coh-coproduct}
Let $V$ be a Liouville domain of dimension $2n$. Then:

(a) The primary coproduct on $SH_*(\p V)$ vanishes. As a consequence, the group $SH_*(\p V)$ carries a secondary coproduct $\boldsymbol{\lambda}$ of degree $-n+1$ which is coassociative, graded cocommutative, and counital (if degrees are shifted by $n-1$). 

(b) The Poincar\'e duality isomorphism 
$$
   PD:SH_*(\p V)\stackrel\simeq \longrightarrow SH^{1-*}(\p V)
$$
is a homomorphism of coalgebras, where $SH_*(\p V)$ is endowed with its degree $-n+1$ secondary coproduct and $SH^*(\p V)$ is endowed with its degree $n$ pair-of-pants coproduct. 
\end{theorem}

\begin{proof}[Proof of Theorem~\ref{thm:coh-coproduct}]
The proof is entirely analogous to that of Theorem~\ref{thm:coh-product}. For finite action intervals $(a,b)$, it uses the dual version of Lemma~\ref{lem:secondary-product} for coproducts (obtained by applying Lemma~\ref{lem:secondary-product} to the dual chain complexes), and the Hamiltonian profiles in Figure~\ref{fig:dent} turned upside-down (i.e., with $h$ replaced by $-h$). The resulting coproduct in the limit as $a\to-\infty$ and $b\to\infty$ lands in a suitable completed tensor product described in~\S\S\ref{sec:Tate}--\ref{sec:coFrob}.
\end{proof}

Theorem~\ref{thm:coh-coproduct} has a Lagrangian counterpart. Given the Liouville domain $V$, we consider Maslov 0 exact Lagrangians $L\subset V$ with boundary such that $\p L=L\cap \p V$ and $L$ is conical near $\p L$. 

\begin{theorem} \label{thm:coh-coproduct-Lag}
Let $V$ be a Liouville domain of dimension $2n$ and $L\subset V$ be an exact Lagrangian as above. Then:

(a) The primary coproduct on $SH_*(\p L)$ vanishes. Therefore $SH_*(\p L)$ carries a secondary coproduct of degree $1$ which is coassociative and counital (if degrees are shifted by $-1$). 

(b) The Poincar\'e duality isomorphism 
$$
   PD:SH_*(\p L)\stackrel\simeq \longrightarrow SH^{n-*+1}(\p L)
$$
is a homomorphism of coalgebras, where $SH_*(\p L)$ is endowed with its degree $1$ secondary coproduct and $SH^*(\p L)$ is endowed with its degree $n$ pair-of-pants coproduct. 
\end{theorem}

\begin{proof}
The proof is up to notation the same as that of Theorem~\ref{thm:coh-coproduct}. Gradings are discussed in  Appendix~\ref{sec:grading}. 
\end{proof}

\subsection{Relation to ordinary symplectic homology}\label{ss:ord-hom}
  
Recall from~\cite{CO} the pair-of-pants product $\mu$ on symplectic homology $SH_*(V)$ and the secondary pair-of-pants coproduct $\lambda$ on positive symplectic homology $SH_*^{>0}(V)$, as well as their algebraic duals $\mu^\vee$ and $\lambda^\vee$. The next result relates these to the product $\boldmu$ and coproduct $\boldlambda$ on $SH_*(\p V)$ defined above. 

\begin{theorem}[Relation to ordinary symplectic homology] \label{thm:RFH-SH}
There exists a commuting diagram with exact row
\begin{equation*}\label{eq:RFH-SH-intro}
\xymatrix
@C=22pt
{
   & & & \bigl(SH_{>0}^{1-*}(V),\lambda^\vee \tau\bigr) \ar[dl]_-{i} \ar[d]^{j} & \\
   \ar[r]^-\eps & \bigl(SH_*(V),\mu\bigr) \ar[r]^-{\iota} \ar[d]^-q & 
   \bigl(SH_*(\p V),\boldsymbol{\mu},\boldsymbol{\lambda}\bigr) \ar[r]^-{\pi} \ar[dl]_-{p} & 
   \bigl(SH^{1-*}(V), \tau \mu^\vee\bigr) \ar[r]^-\eps & \\
   & \bigl(SH_*^{>0}(V),\lambda\bigr) &
}
\end{equation*}
in which
\vspace{-5pt}
\begin{itemize} 
\item the maps $\iota$ and $i$ intertwine the products $\mu$,
  $\boldsymbol{\mu}$ and $\lambda^\vee \tau$;
\item the maps $p$ and $\pi$ intertwine the coproducts $\lambda$,
  $\boldsymbol{\lambda}$ and $\tau \mu^\vee$.
\end{itemize} 
Dualizing the diagram
and applying Poincar\'e duality $SH_*(\p V)\cong \break SH^{1-*}(\p V)$ 
reproduces the same diagram reflected at its center. 
\end{theorem}

\begin{proof}
By~\cite[Proposition 7.19]{CO} there exists a commuting diagram
\begin{equation*}
\xymatrix
@C=22pt
{
   & 0=SH_*^{<0}(V) \ar[r] & SH_*^{<0}(\p V) \ar[r]^-k_-\cong \ar[d]_-{i'} & SH^{<0}_{*-1}(V,\p V) \ar[d]^{j} \ar[r] & 0 \\
   \ar[r]^-\eps & SH_*(V) \ar[r]^-{\iota} \ar[d]^-q & 
   SH_*(\p V) \ar[r]^-{\pi} \ar[d]_-{p'} & 
   SH_{*-1}(V,\p V) \ar[r]^-\eps & \\
   0 \ar[r] & SH_*^{>0}(V) \ar[r]^-r_-\cong & SH_*^{>0}(\p V) \ar[r] & SH_{*-1}^{>0}(V,\p V)=0 
}
\end{equation*}
where the rows are exact sequences of the pair $(V,\p V)$, and the vertical maps are parts of the tautological sequences given by action truncation. (Note that the columns are {\em not} tautological sequences and thus not exact). Setting $i=i'k^{-1}$ and $p=r^{-1}p'$, in view of the canonical ring isomorphisms $SH_{*-1}(V,\p V)\cong SH^{1-*}(V)$ and $SH^{<0}_{*-1}(V,\p V)\cong SH_{>0}^{1-*}(V)$ this yields the desired commuting diagram. By~\cite[Theorem 10.2]{CO}, the maps $\iota$ and $i$ intertwine the respective pair-of-pants products. By Poincar\'e duality, which dualizes and reflects the diagram, this implies that the maps $p$ and $\pi$ intertwine the respective coproducts. 
\end{proof}

\begin{remark} \label{rmk:RFH-SH-Lag} Theorem~\ref{thm:RFH-SH} has a Lagrangian counterpart. The statement is entirely analogous and we omit it. 
\end{remark}

\section{Graded Frobenius algebra structure}\label{sec:sec-bialg}

In this section we prove that the Poincar\'e duality isomorphism from~\S\ref{sec:PD-SH} intertwines naturally defined graded Frobenius algebra structures. Up to the discussion of graded open-closed TQFT structures from~\S\ref{sec:TQFT}, this proves Theorems~\ref{thm:Frob-RFH-intro} and~\ref{thm:TQFT-RFH-intro} from the Introduction. 
Throughout this section we use coefficients in a field $\bk$.

Given a Liouville domain $V$ of dimension $2n$, we denote 
$$
   S\H_*(\p V)=SH_{*+n}(\p V),\qquad S\H^*(\p V)=SH^{*+n}(\p V)
$$
the shifted Rabinowitz Floer, or symplectic, (co)homology of $\p V$.

\begin{theorem}[Graded Frobenius algebra structure on symplectic homology]\label{thm:PD-sec-bialg-SH}
The product $\boldsymbol{\mu}$ and coproduct $\boldsymbol{\lambda}$ of~\S\ref{sec:PD-SH} make $S\H_*(\p V)$ into a commutative cocommutative graded Frobenius algebra in the sense of Definition~\ref{defi:biunital-coFrobenius-bialgebra}.
The dual product $\boldsymbol{\lambda}^\vee \tau$ and coproduct $\tau \boldsymbol{\mu}^\vee$ make
$S{\H}^*(\p V)$ into a commutative cocommutative graded Frobenius algebra. 
Poincar\'e duality yields an isomorphism of such (bi)algebras
$$
   PD: (S{\H}_*(\p V),\boldsymbol{\mu},\boldsymbol{\lambda})
   \stackrel\simeq\longrightarrow 
   (S{\H}^{1-2n-*}(\p V),\boldsymbol{\lambda}^\vee \tau, \tau \boldsymbol{\mu}^\vee).
$$
\end{theorem}

Given the Liouville domain $V$ of dimension $2n$, let $L\subset V$ be a Maslov 0 exact Lagrangian with boundary such that $\p L=L\cap \p V$ and $L$ is conical near $\p L$.  We denote 
$$
   S\H_*(\p L)=SH_{*+n}(\p L),\qquad S\H^*(\p L)=SH^{*+n}(\p L) 
$$
the shifted Rabinowitz Floer homology, or symplectic, or wrapped Floer (co)homology of $\p L$.

\begin{theorem}[Graded Frobenius algebra structure on symplectic homology, Lagrangian case]\label{thm:PD-sec-bialg-SH-Lag}
The product $\boldsymbol{\mu}$ and coproduct $\boldsymbol{\lambda}$ of~\S\ref{sec:PD-SH} make $S\H_*(\p L)$ into a graded Frobenius algebra in the sense of Definition~\ref{defi:biunital-coFrobenius-bialgebra}. 
The dual product $\boldsymbol{\lambda}^\vee \tau$ and coproduct $\tau \boldsymbol{\mu}^\vee$ make
$S{\H}^*(\p L)$ into a graded Frobenius algebra. 
Poincar\'e duality yields an isomorphism of such (bi)algebras
$$
   PD: (S{\H}_*(\p L),\boldsymbol{\mu},\boldsymbol{\lambda})
   \stackrel\simeq\longrightarrow 
   (S{\H}^{1-n-*}(\p L),\boldsymbol{\lambda}^\vee \tau, \tau \boldsymbol{\mu}^\vee).
$$
\end{theorem}

The explanation and proof of these theorems will occupy the remainder of this section. 
In~\$\ref{sec:Tate} and~\S\ref{sec:coFrob} we recall from~\cite{CHO-algebra} the notion of a graded Tate vector space and a graded Frobenius algebra and state an algebraic version of Poincar\'e duality. In the sequel subsections we give a direct definition of the secondary coproduct in symplectic homology and prove the relations of a graded Frobenius algebra. In the process, we obtain a more direct description of the Poincar\'e duality isomorphism in terms of the copairing (Proposition~\ref{prop:c-iso}). The closing subsection wraps up the ensemble and summarizes the proofs of Theorems~\ref{thm:PD-sec-bialg-SH} and~\ref{thm:PD-sec-bialg-SH-Lag}.

\subsection{Tate vector spaces~\cite{CHO-algebra,CO-Tate}}\label{sec:Tate}

We equip the field $\bk$ with the discrete topology. 
A {\em linearly topologized vector space} is a $\bk$-vector space $A$
with a Hausdorff topology which is translation invariant and has a
basis of open neighbourhoods of $0$ consisting of linear subspaces. 
Note that an open linear subspace $U$ is also closed
because its set-theoretic complement can be written as a union of translates of $U$.
Linearly topologized vector spaces form a category $\Top$ whose morphisms are
continuous linear maps.
The kernel and cokernel of a morphism $f:A\to B$ are defined as
$$
  \ker(f) = \{v\in A\mid f(v)=0\}, \qquad 
  \coker(f) = B/\ol{f(A)}. 
$$
The {\em completion} of a linearly topologized vector space $A$ is
$$
   \wh A := \lim_{U\in\mathcal{U}} A/U,
$$
where $\mathcal{U}$ denotes the collection of open linear subspaces in $A$.
The space $A$ is called {\em complete} if the canonical map $A \to \wh A$ is an isomorphism.

A linearly topologized vector space $A$ is called
\vspace{-5pt}
\begin{itemize}
\item \emph{discrete} if $\{0\}$ is an open set;
\item \emph{linearly compact} if it is complete and $\dim(A/U)<\infty$ for each open linear subspace $U\subset A$;
\item \emph{Tate}\footnote{
These spaces were introduced by Lefschetz~\cite[II]{Lefschetz-book} under the name of {\em locally linearly compact vector spaces}; see~\cite{CO-Tate} for further discussion and references.}
if it is complete and admits an open linearly compact subspace.
\end{itemize}

For example, the vector space $\bk[t^{-1},t]]$ of Laurent power series, with a basis of neighborhoods of $0$ given by $t^n\bk[[t]]$, $n\in\Z$, is Tate. Its open subspace $\bk[[t]]$ is linearly compact, and its complementary closed subspace $t^{-1}\bk[t^{-1}]$ is discrete. 
More generally, a linearly topologized vector space $A$ is Tate iff it is a topological direct sum $A= L\oplus D$ with $L$ linearly compact and $D$ discrete. 
It is finite dimensional iff it is discrete and linearly compact. 

The categorical product in $\Top$ of a family $\{V_i\}_{i\in I}$ is the \emph{direct product} $\prod_i V_i$ endowed with the initial topology with respect to the canonical projections $proj_i:\prod_i V_i\to V_i$. As such, a product of linearly compact vector spaces is linearly compact. The categorical coproduct in $\Top$ of a family $\{V_i\}_{i\in I}$ is the \emph{direct sum} $\oplus_i V_i$, endowed with the final topology with respect to the canonical inclusions $incl_i:V_i\to \oplus_i V_i$. As such, the coproduct of discrete spaces is discrete~\cite[\S2.1]{CO-Tate}.

{\em Duality. }
For linearly topologized vector spaces $A,B$ we denote by $\Hom(A,B)$ the space of continuous linear maps $A\to B$ equipped with the {\em compact-open topology}, i.e., the linear topology whose neighbourhood basis of the origin is given by the linear subspaces $\{f\in \Hom(A,B)\mid f(L)\subset V\}$ for $L\subset A$ linearly compact and $V\subset B$ linear open. 
An important special case of this is the topological dual $A^\vee=\Hom(A,\bk)$, whose neighbourhood basis of the origin is given by the linear subspaces
$$
   L^\perp=\{\alpha\in A^\vee\mid \alpha|_L=0\},\qquad
   L\subset A \text{ linearly compact}.
$$ 
The association $A\mapsto A^\vee$ is functorial: a continuous linear map $f:A\to B$ induces a continuous linear map $f^\vee:B^\vee\to A^\vee$, $f^\vee\beta=\beta\circ f$. 

\begin{theorem}[Lefschetz-Tate duality~{\cite[(II.28.2-29.1)]{Lefschetz-book}}]\label{thm:Tate-duality}\quad 

(a) If $A$ is discrete, then $A^\vee$ is linearly compact. 

(b) If $A$ is linearly compact, then $A^\vee$ is discrete.
  
(c) If $A$ is Tate, then so is $A^\vee$ and the canonical map $A\to
  A^{\vee\vee}$ is a topological isomorphism. \qed
\end{theorem}

E.g., $\bigoplus_\Z \bk$ is discrete, $\prod_\Z\bk$ is linearly compact, and these spaces are topological duals to each other. See~\cite[Lemma~4.3(c)]{CO-Tate} and~\cite[Remark~2.18]{CHO-algebra}. More generally, the dual of an arbitrary coproduct is the product of the duals, and the dual of an arbitrary product is the coproduct of the duals.

{\em Tensor products. }
Let $A,B$ be linearly topologized vector spaces. Beilinson describes in~\cite[\S1.1]{Beilinson} two topologies on the algebraic tensor product $A\otimes B$ that are uniquely characterized by the following universal properties:
\begin{itemize}
\item The $*$ topology is the finest linear topology such that the canonical bilinear map $A\times B\to A\otimes B$ is continuous;
\item Let $B(A^\vee\times B^\vee,\bk)$ be the space of continuous bilinear maps $A^\vee\times B^\vee\to \bk$, endowed with the compact-open topology.\footnote{See~\cite[\S4.3]{CO-Tate} for a discussion of spaces of bilinear maps. Note the slight ambiguity in the notation: the space $B(A^\vee\times B^\vee,\bk)$ depends on each of the factors $A^\vee$ and $B^\vee$, and not just on the product $A^\vee\times B^\vee$.} For $A$ and $B$ Tate, the $!$ topology is the coarsest topology such that the linear map 
$$
A\otimes B\to B(A^\vee\times B^\vee,\bk),\qquad a\otimes b \mapsto \big( (\varphi,\psi)\mapsto \varphi(a)\psi(b)\big)
$$
is continuous. 
\end{itemize}
We denote $A\hatotimes^*B$ and $A\hatotimes^!B$ the completions of $A\otimes B$ with respect to the $*$ and $!$ 
topologies. Since the $!$ topology is coarser than the $*$ topology, the identity map induces a continuous linear map
$$
  A\hatotimes^* B\to A\hatotimes^! B.
$$

Given Tate vector spaces $A,B$ we have canonical isomorphisms~\cite{Esposito-Penkov23,CO-Tate}
$$
(A\otimes^* B)^\vee\simeq A^\vee\otimes^! B^\vee,\qquad 
(A\otimes^! B)^\vee\simeq A^\vee\otimes^* B^\vee. 
$$
A key aspect is that neither $A\otimes^* B$ nor $A\otimes^! B$ need be Tate. Following Esposito-Penkov~\cite{Esposito-Penkov22,Esposito-Penkov23} and~\cite[\S\S5.5-6]{CO-Tate}, this ``monoidal'' duality statement can be understood in terms of the categories $\cI$ of ind-linearly compact vector spaces and $\cP$ of pro-discrete vector spaces. The category $\cI$ is closed under the $\otimes^*$ tensor product, the category $\cP$ is closed under the $\otimes^!$ tensor product, and topological duality induces an involutive equivalence $\cI\simeq \cP$ that intertwines the two tensor products. The category $Tate$ of Tate vector spaces is such that $Tate\subset\cI\cap \cP$. Under suitable countability conditions on $Tate$, $\cI$, and $\cP$, this inequality becomes an equality~\cite[Proposition~6.5]{CO-Tate}.

{\em Graded Tate vector spaces. }
A {\em ($\Z$-)graded linearly topologized vector space} is a direct sum $A=\bigoplus_{i\in\Z}A_i$ of linearly topologized vector spaces. It is a \emph{graded Tate vector space} if $A_i$ is Tate for each $i\in\Z$. The morphism spaces in the category of graded linearly topologized vector spaces are by definition  
\begin{align*}
  \Hom(A,B) = \bigoplus_{k\in\Z}\Hom_k(A,B),\quad
  \Hom_k(A,B) = \prod_i \Hom(A_i,B_{i+k}).
\end{align*}
(The product used in the definition of $\Hom_k(A,B)$ is the usual categorical product in $\Top$.) The degree of a homogeneous element $a\in A$ is denoted $|a|$. A morphism $\phi:A\to B$ is \emph{homogeneous of degree $k$} if $|\phi(a)|=|a|+k$ for homogeneous elements $a\in A$. A homogeneous morphism of degree $k$ is therefore a collection $\phi=(\phi_i)$ with $\phi_i\in\Hom(A_i,B_{i+k})$. The degree of a homogeneous morphism $\phi$ is denoted $|\phi|$. 

The \emph{dual} of a graded linearly topologized vector space is defined to be  
$$
  A^\vee = \bigoplus_{i\in\Z}A^\vee_i,\qquad A^\vee_i=\Hom(A_{-i},\bk).
$$
The $*$ and $!$ tensor products of graded linearly topologized vector spaces are the graded linearly topologized vector spaces 
\begin{gather*}
A\otimes^* B = \bigoplus_{k\in\Z}(A\otimes^* B)_k,\qquad 
  (A\otimes^* B)_k = \bigoplus_{i+j=k}A_i\hatotimes^* B_j, \cr
A\otimes^! B = \bigoplus_{k\in\Z}(A\otimes^! B)_k,\qquad 
  (A\otimes^! B)_k = \prod_{i+j=k}A_i\hatotimes^! B_j.
\end{gather*}
(The product used in the definition of $(A\otimes^! B)_k$ is again the usual categorical product in $\Top$.)

\begin{proposition} \label{prop:comm-ass-graded}
For graded linearly topologized vector spaces $A,B,C$ 
we have canonical isomorphisms
\begin{gather*}
  A\otimes^* B\simeq B\otimes^* A, \qquad
  (A\otimes^* B)\otimes^* C \simeq A\otimes^*(B\otimes^* C), \cr
  A\otimes^! B\simeq B\otimes^! A, \qquad
  (A\otimes^! B)\otimes^! C \simeq A\otimes^!(B\otimes^! C).
\end{gather*}
Moreover, the identity induces a continuous linear map
$$
   A\otimes^*(B\otimes^!C) \to (A\otimes^*B)\otimes^!C.
$$
\qed
\end{proposition}

\begin{theorem}
\label{thm:duality*!-Tate-graded}
For graded Tate vector spaces $A,B$ we have canonical isomorphisms
$$
(A\otimes^* B)^\vee\simeq A^\vee\otimes^! B^\vee,\qquad 
(A\otimes^! B)^\vee\simeq A^\vee\otimes^* B^\vee. 
$$
\qed
\end{theorem}

The following result fits Rabinowitz Floer homology into this framework. Here we view both homology and cohomology as graded topological vector spaces, i.e, direct sums of their pure degree components.

\begin{theorem}[\cite{CO-Tate}]
Rabinowitz Floer homology and cohomology of any Liouville domain $V$ are
graded Tate vector spaces such that $SH^*(\p V)$ is the topological dual of
$SH_*(\p V)$, and Poincar\'e duality (see Theorem~\ref{thm:coh-product}) defines an isomorphism of graded Tate vector spaces. \qed
\end{theorem}

\subsection{Graded Frobenius algebras~\cite{CHO-algebra}} \label{sec:coFrob}

Let $A=\bigoplus_{i\in\Z}A_i$ be a graded Tate vector space.
We denote the degree of a homogeneous element $a\in A$ by $|a|$. A continuous linear map $\phi:A\to B$ is homogeneous of degree $d$ if
$|\phi(a)|=|a|+d$ for homogeneous elements $a\in A$, and we denote its degree by $|\phi|$. We will use the following conventions.

{\it Identity map.} The identity map of $A$ is denoted $1:A\to A$. 

{\it Twist.} The twist $\tau:A\otimes A\to A\otimes A$ defined by $\tau(a\otimes b)=(-1)^{|a||b|}b\otimes a$ induces maps $A\otimes^* A\to A\otimes^* A$ and $A\otimes^! A\to A\otimes^! A$ still denoted $\tau$.  

{\it Product.} A morphism $\boldmu:A\otimes^* A\to A$ is called a \emph{product}.  
We say that $\boldmu$ is \emph{commutative} if 
$$
\boldmu \tau = (-1)^{|\boldmu|}\boldmu.
$$
We say that $\boldmu$ is \emph{associative} if 
$$
\boldmu(\boldmu\otimes 1) = (-1)^{|\boldmu|}\boldmu(1\otimes \boldmu).
$$
An element $\eta\in A$ is called a \emph{unit} for $\boldmu$ if
$$
(-1)^{|\boldmu|} \boldmu(\boldeta\otimes 1) = 1 = \boldmu(1\otimes \boldeta).
$$
{\it Coproduct.} A morphism $\boldlambda:A\to A\otimes^! A$ is called a \emph{coproduct}. 
We say that $\boldlambda$ is \emph{cocommutative} if 
$$
\tau \boldlambda = (-1)^{|\boldlambda|}\boldlambda. 
$$
We say that $\boldlambda$ is \emph{coassociative} if 
$$
(\boldlambda\otimes 1)\boldlambda = (-1)^{|\boldlambda|}(1\otimes \boldlambda)\boldlambda. 
$$
A morphism $\boldeps:A\to \bk$ is called a \emph{counit} for $\boldlambda$ if 
$$
(\boldeps\otimes 1)\boldlambda = 1 = (-1)^{|\boldlambda|}(1\otimes\boldeps)\boldlambda.
$$
 
\begin{definition} \label{defi:biunital-coFrobenius-bialgebra} 
A \emph{graded Frobenius algebra} is a graded Tate vector space $A$ endowed with a degree zero product $\boldmu:A\otimes^* A\to A$, a coproduct $\boldlambda:A\to A\otimes^! A$, and elements $\boldeta\in A$, $\boldeps:A\to \bk$ that satisfy the following relations:
\begin{itemize}
\item {\sc (unit)} the element $\boldeta$ is the unit for the product $\boldmu$.
\item {\sc (counit)} the element $\boldeps$ is the counit for the coproduct $\boldlambda$. 
\item {\sc (associativity)} the product $\boldmu$ is associative. 
\item {\sc (coassociativity)} the coproduct $\boldlambda$ is coassociative. 

\medskip 

\noindent Moreover, defining the \emph{copairing} by 
$$
\boldc = \boldlambda\boldeta,
$$
and the \emph{pairing} by 
$$
\boldp = (-1)^{|\boldlambda|}\boldeps\boldmu,
$$
we have: 

\item {\sc (Frobenius)} 
$$
\boldlambda=(1\otimes \boldmu)(\boldc\otimes 1) = (\boldmu\otimes 1)(1\otimes \boldc)
$$
and
$$
\boldmu = (-1)^{|\boldlambda|}(\boldp\otimes 1)(1\otimes \boldlambda) = (1\otimes \boldp)(\boldlambda\otimes 1).
$$
\item {\sc (symmetry)}
$$
\tau \boldc = (-1)^{|\boldlambda|} \boldc
$$
and
$$
\boldp \tau  = \boldp. 
$$
\end{itemize}
\end{definition}

In~\cite{CHO-algebra}, a graded Frobenius algebra is also called a {\em biunital coFrobenius bialgebra}. In contrast to~\cite{CHO-algebra}, we assume in this paper that the product has degree zero (which can always be achieved by a degree shift). We depict the Frobenius relations in Figure~\ref{fig:Frobenius}, where the operations have their inputs at the top and their outputs at the bottom. 

\begin{figure}
\begin{center}
\includegraphics[width=\textwidth]{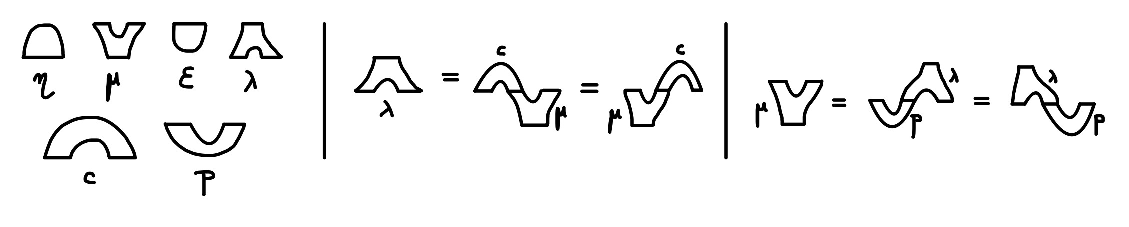}
\caption{The Frobenius relations.}
\label{fig:Frobenius} 
\end{center}
\end{figure}

A graded Frobenius algebra is called {\em commutative and cocommutative} if $\boldmu$ is commutative and $\boldlambda$ is cocommutative.
  
{\it Duality.}
Given a homogeneous morphism $\varphi:A\longrightarrow B$, its \emph{dual morphism} $\varphi^\vee:B^\vee\longrightarrow A^\vee$ is defined by $\varphi^\vee g= (-1)^{|g|  |\varphi|}g\circ \varphi$. 

Any operation with inputs/outputs in $A$ is equivalent to one with outputs/inputs in $A^\vee$. To navigate between these two points of view we use the \emph{evaluation} and \emph{coevaluation} maps. We have a degree $0$ continuous bilinear map $\langle\cdot,\cdot\rangle:A^\vee\times A\to \bk$ given by evaluation $\langle f,a\rangle = f(a)$~\cite[Proposition~4.12]{CO-Tate}. This induces a morphism $\ev:A^\vee\otimes^* A\to \bk$, which we call \emph{evaluation}. The dual morphism $\ev^\vee:\bk\to (A^\vee\otimes^* A)^\vee\cong A\otimes^! A^\vee$ is called \emph{coevaluation}. These two maps satisfy $(1\otimes \ev)(\ev^\vee\otimes1)=1_A$ and $(\ev\otimes1)(1\otimes\ev^\vee)=1_{A^\vee}$.  

Given a copairing $\boldc:\bk\to A\otimes^! A$ we denote 
$$
\vec\boldc:A^\vee\to A
$$
the map $\vec\boldc=(\ev\otimes 1)(1\otimes\boldc)$, i.e., $\vec\boldc(f)=(-1)^{|f||\boldc|}(f\otimes 1)\boldc$. 
Given a pairing $\boldp:A\otimes^* A\to \bk$ we denote 
$$
\vec\boldp:A\to A^\vee
$$
the map $\vec\boldp=(\boldp\otimes 1)(1\otimes\ev^\vee)$. This is equivalent to $\ev(\vec\boldp\otimes 1)=\boldp$, i.e., $\vec\boldp(a)(b)=\boldp(a\otimes b)$.

We say that a homogeneous morphism $\varphi:A\to B$ \emph{intertwines products} $\boldmu_B,\boldmu_B$ if 
$$
\varphi\boldmu_A = (-1)^{|\varphi||\boldmu_A|}\boldmu_B\varphi^{\otimes 2}.
$$ 
If this holds and $\boldmu_A$ is unital with unit $\boldeta_A$, then $\boldmu_B$ is unital with unit 
$$
\boldeta_B=(-1)^{|\varphi|}\varphi\boldeta_A.
$$
We say that $\varphi$ \emph{intertwines coproducts} $\boldlambda_A,\boldlambda_B$ if
$$
\varphi^{\otimes 2} \boldlambda_A = (-1)^{|\varphi||\boldlambda_A|}\boldlambda_B\varphi.
$$ 
If this holds and $\boldlambda_B$ is counital with counit $\boldeps_B$, then $\boldlambda_A$ is counital with counit 
$$
\boldeps_A=\boldeps_B\varphi. 
$$
With this terminology, we have

\begin{theorem}[Algebraic Poincar\'e duality~\cite{CHO-algebra}] \label{thm:algebraic_Poincare_duality}
Let $(A,\boldmu,\boldlambda,\boldeta,\boldeps)$ be a graded Frobenius algebra with pairing $\boldp$ and copairing $\boldc$.
Then the maps
$$
\vec\boldp:A\to A^\vee,\qquad \vec\boldc:A^\vee\to A
$$
are mutually inverse isomorphisms of graded Frobenius algebras
\begin{align*}
  (A,\boldmu,\boldlambda, \boldeta,\boldeps)
  \simeq (A^\vee,(-1)^{|\boldlambda|}\boldlambda^\vee\tau, (-1)^{|\boldlambda|}\tau\boldmu^\vee,\boldeps^\vee,(-1)^{|\boldlambda|}\boldeta^\vee),
\end{align*}
i.e. they intertwine the products, they intertwine the coproducts, and they preserve the units and counits. 
\end{theorem} 

Thus Poincar\'e duality holds automatically on each Frobenius algebra. The following criterion will be useful.

\begin{proposition}[\cite{CHO-algebra}]\label{prop:coFrob-criterion}
Let $(A,\boldmu,\boldlambda,\boldeta)$ with $\boldc = \boldlambda\boldeta$ satisfy the following conditions:
\begin{itemize}
\item {\sc (unit)} the element $\boldeta$ is the unit for the degree zero product $\boldmu$.
\item {\sc (associativity)} the product $\boldmu$ is associative. 
\item {\sc (unital coFrobenius)} 
$$
\boldlambda=(1\otimes \boldmu)(\boldc\otimes 1) = (\boldmu\otimes 1)(1\otimes \boldc).
$$
\item {\sc(symmetry)}
$$
\tau \boldc = (-1)^{|\boldlambda|} \boldc. 
$$
\item {\sc(isomorphism)} the induced map $\vec\boldc:A^\vee\to A$ is an isomorphism. 
\end{itemize}
Then $\boldeps=(-1)^{|\boldlambda|}\boldeta^\vee\vec\boldc^{-1}$ makes $(A,\boldmu,\boldlambda,\boldeta,\boldeps)$ a graded Frobenius algebra.  
\end{proposition}

\subsection{Product and unit on $SH_*(\p V)$} \label{sec:product}
  
To prove Theorems~\ref{thm:PD-sec-bialg-SH} and~\ref{thm:PD-sec-bialg-SH-Lag}, we will construct a product $\boldmu$ with unit $\boldeta$ and a coproduct $\boldlambda$ on $SH_*(\p V)$ and verify the relations in Proposition~\ref{prop:coFrob-criterion}.

The product $\boldmu$ is induced by the usual pair-of-pants product, and $\boldeta$ is induced by the count of spheres with one negative puncture of weight $1$. It is well-known that $\boldmu$ is commutative and associative with unit $\boldeta$. 
By~\cite{CO-Tate} the product defines a morphism $\boldmu:SH_*(\p V)\otimes^* SH_*(\p V)\to SH_*(\p V)$.

\subsection{The primary coproduct on $SH_*(\p V)$ vanishes} \label{sec:primary_coproduct}
  
We give in this subsection a second proof of the vanishing of the primary coproduct on $SH_*(\p V)$ from Theorem~\ref{thm:coh-coproduct}. This is useful for the direct definition of the secondary coproduct which we give in~\S\ref{sec:secondary_coproduct}. 

Let $H=H_{\lambda,\mu}$ with $\lambda<0<\mu$ as in~\S\ref{sec:PD-products} and $a<b$. The genus zero Riemann surface $\Sigma$ with 1 positive puncture, 2 negative punctures, and fixed cylindrical ends at the punctures, together with the choice of a non-positive $1$-form with weights $(1;2,2)$ at the punctures in the sense of~\S\ref{sec:TQFT-Floer}, defines a degree $-n$ primary coproduct 
$$
\boldsymbol{\lambda}^{\text{primary}}:FC_*(H)\to FC_*(2H)\otimes FC_*(2H).
$$
Since the sum of Hamiltonian actions at the output is less than or equal to the action at the input, it acts with respect to the action filtrations as
$$
\boldsymbol{\lambda}^{\text{primary}}:FC_*^{\le a}(H)\to FC_*^{\le \frac a 2}(2H)\otimes FC_*(2H) + FC_*(2H)\otimes FC_*^{\le \frac a 2}(2H)
$$
and
$$
\boldsymbol{\lambda}^{\text{primary}}:FC_*^{<b}(H)\to \sum_{b_1+b_2=b}FC_*^{<b_1}(2H)\otimes FC_*^{<b_2}(2H). 
$$
Combining these for $a<b$, we get
$$
\boldsymbol{\lambda}^{\text{primary}}:FC_*^{(a,b)}(H)\to \sum_{b_1+b_2=b}FC_*^{(\frac a 2,b_1)}(2H)\otimes FC_*^{(\frac a 2,b_2)}(2H).
$$ 
Passing to homology and in the limit over $H$ we obtain the degree $-n$ primary coproduct  
$$
\boldsymbol{\lambda}^{\text{primary}}:SH_*^{(a,b)}(\p V)\to \sum_{b_1+b_2=b}SH_*^{(\frac a 2,b_1)}(\p V)\otimes SH_*^{(\frac a 2,b_2)}(\p V).
$$

\begin{proposition} \label{prop:vanishinglambdaprimary}
The primary coproduct
$$
\boldsymbol{\lambda}^{\text{primary}}:SH_*^{(a,b)}(\p V)\to \sum_{b_1+b_2=b}SH_*^{(\frac a 2,b_1)}(\p V)\otimes SH_*^{(\frac a 2,b_2)}(\p V).
$$
vanishes for all $a<b$. 
\end{proposition}

\begin{proof} Given $H=H_{\lambda,\mu}$ we denote $L=H_{\lambda,\lambda}$. The homology coproduct 
$$
\boldsymbol{\lambda}^{\text{primary}}:FH_*^{(a,b)}(H)\to \sum_{b_1+b_2=b}FH_*^{(\frac a 2,b_1)}(2H)\otimes FH_*^{(\frac a 2,b_2)}(2H)
$$
does not depend on the choice of non-positive $1$-form on the Riemann surface $\Sigma$. In particular, the latter can be chosen as follows: it has weights $(1;-1,2)$ away from a neighborhood of $-\infty$ in the cylindrical end at the first negative puncture, it has weight $-1$ on a long neck in that cylindrical end, and it has weight $2$ at $-\infty$ in that cylindrical end. As such, we obtain a factorization
$$
\boldsymbol{\lambda}^{\text{primary}}=(c\otimes 1)\boldsymbol{\lambda}^{\prime},
$$
where at chain level
$$
\boldsymbol{\lambda}^{\prime}:FC_*^{(a,b)}(H)\to \sum_{b_1+b_2=b}FC_*^{(\frac a 2,b_1)}(-H)\otimes FC_*^{(\frac a 2,b_2)}(2H)
$$
is a primary coproduct and $c:FC_*^{(\frac a 2,b_1)}(-H)\to FC_*^{(\frac a 2,b_1)}(2H)$ is a continuation map.
Let us now choose the parameters $\lambda$ and $\mu$ such that $-\mu\le \lambda<0<\mu$. We then have $-H=H_{-\lambda,-\mu}\le L=H_{\lambda,\lambda}\le H=H_{\lambda,\mu}$, and therefore the continuation map $c$ factors as $FC_*^{(\frac a 2,b_1)}(-H)\to FC_*^{(\frac a 2,b_1)}(L)\to FC_*^{(\frac a 2,b_1)}(2H)$. Since $FC_*^{(\frac a 2,b_1)}(L)=0$ for $\lambda<a$, we obtain $c=0$ in this range. Passing to homology and the limit over $H=H_{\lambda,\mu}$ as $\lambda\to-\infty$ and $\mu\to\infty$ thus yields the desired vanishing result.
\end{proof}

\subsection{Secondary coproduct on $SH_*(\p V)$}  \label{sec:secondary_coproduct}

We give in this subsection a direct definition of the degree $-n+1$ secondary coproduct $\boldsymbol{\lambda}$ on $SH_*(\p V)$. It follows from~\cite{CO-cones} that this is dual to the degree $n-1$ secondary product on Rabinowitz Floer cohomology from Theorem~\ref{thm:coh-product}. 

The vanishing of $\boldsymbol{\lambda}^{\text{primary}}$ was proved above by factoring out a continuation map at the first negative puncture. The same argument could have been carried at the second negative puncture. The secondary coproduct $\boldsymbol{\lambda}$  is obtained by interpolating between these two constructions. 

Consider again the genus $0$ Riemann surface $\Sigma$ with 1 positive puncture, 2 negative punctures and fixed cylindrical ends at the punctures, with coordinates $z=s+it$, $t\in S^1$, $s\in (-\infty,0]$ at the negative punctures. We pick a smooth $1$-parameter family of Floer data $H_\tau,\beta_\tau$, $\tau\in (0,1)$, on $\Sigma$ as follows (see Figure~\ref{fig:Coproduct-RFH-limit}).

\begin{figure}
\begin{center}
\includegraphics[width=.7\textwidth]{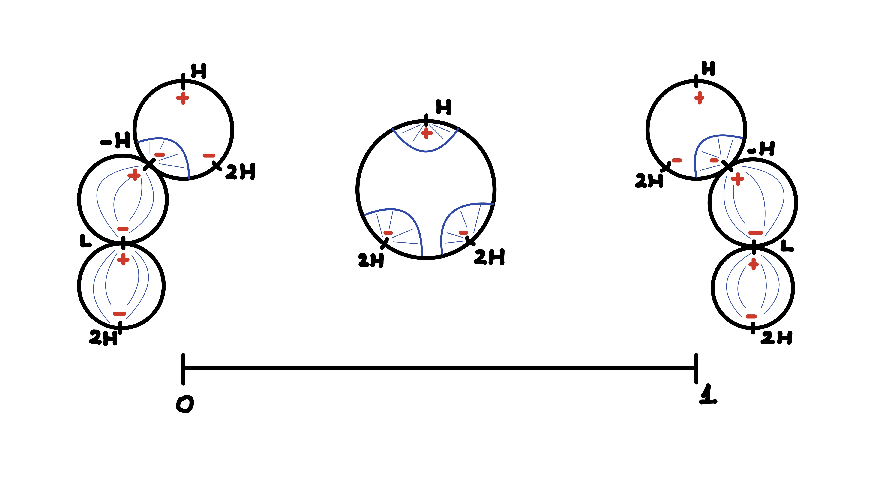}
\caption{Definition of the coproduct with broken homotopies at the endpoints.}
\label{fig:Coproduct-RFH-limit}
\end{center}
\end{figure}

{\em Outside the negative cylindrical ends} $H_\tau=H$ as in the previous subsection, and the $1$-forms $\beta_\tau\in\Omega^1(\Sigma)$ satisfy $d\beta_\tau= 0$ and are equal to $dt$ on the positive cylindrical end, and to $f(\tau)dt$ resp.~$(1-f(\tau))dt$ near the negative boundaries. Here $f:[0,1]\to [-1,2]$ is a smooth function which equals $-1$ on $[0,\delta]$ and $2$ on $[1-\delta,1]$, for some $\delta>0$. 

{\em On the first negative cylindrical end $(-\infty,0]\times S^1$} we write $H_\tau\beta_\tau=H_{\tau,s}dt$ for a family $H_{\tau,s}$ of $s$-dependent Hamiltonians with $\frac{\p H_{\tau,s}}{\p s}\leq 0$ and the following properties:
\begin{itemize}
\item $H_{\tau,s}$ is equal to $f(\tau)H$ near $s=0$, and to $2H$ for $s\leq -R(\tau)$;
\item $H_{\tau,s}=L$ for $\tau\in[0,\delta]$ and $s\in[-R(\tau)+1,-1]$. 
\end{itemize}
Here $R:(0,1]\to[0,\infty)$ is a nonincreasing smooth function with $R(1)=0$, $R(\delta)=3$, and $\lim_{\tau\to 0}R(\tau)=\infty$. 

{\em On the second negative cylindrical end $(-\infty,0]\times S^1$} we write $H_\tau\beta_\tau=H_{\tau,s}dt$ for a family $H_{\tau,s}$ of $s$-dependent Hamiltonians with $\frac{\p H_{\tau,s}}{\p s}\leq 0$ and the following properties:
\begin{itemize}
\item $H_{\tau,s}$ is equal to $(1-f(\tau))H$ near $s=0$, and to $2H$ for $s\leq -R(1-\tau)$;
\item $H_{\tau,s}=L$ for $\tau\in[1-\delta,1]$ and $s\in[-R(1-\tau)+1,-1]$. 
\end{itemize}

Note that at $\tau=0$ the Riemann surface $\Sigma$ degenerates into a nodal Riemann surface consisting of a 3-punctured sphere with asymptotics $(H;-H,2H)$, and a chain of two cylinders with asymptotics $(-H;L)$ and $(L;2H)$ attached at the first negative puncture. A similar degeneration occurs at $\tau=1$ at the second negative puncture, see Figure~\ref{fig:Coproduct-RFH-limit}. 

By construction, the Floer data $H_\tau\beta_\tau$ satisfy conditions (i) and (ii) in \S\ref{sec:TQFT-Floer}. So the moduli spaces of pairs $(\tau,u)$, with $\tau\in [0,1]$ and $u:\Sigma\to \wh V$ solving $(du-X_{H_\tau}\otimes d\beta_\tau)^{0,1}=0$ with given asymptotics, are compact up to Floer breaking. The signed count of $0$-dimensional such moduli spaces therefore defines a degree $-n+1$ map 
$$
\boldsymbol{\lambda}:FC_*(H)\to FC_*(2H)^{\otimes 2},
$$
where the $+1$ in the degree comes from the parameter $\tau$. 
As in the case of $\boldsymbol{\lambda}^{\text{primary}}$, for $a<b$ action considerations yield
$$
\boldsymbol{\lambda}:FC_*^{(a,b)}(H)\to \sum_{b_1+b_2=b}FC_*^{(\frac a 2,b_1)}(2H)\otimes FC_*^{(\frac a 2,b_2)}(2H).
$$ 
Again, the key point now is that for fixed parameters $a<b$ we have $FC_*^{(\frac a 2, b)}(L)=0$ if the slope of $L$ satisfies $\lambda <a$. Since the contributions to $\boldlambda$ at $\tau=0,1$ factor through $FC_*^{(\frac a 2, b)}(L)$, this implies that $\boldlambda$ is a chain map. 
Passing to homology and in the limit over $H$ we obtain  
\begin{equation*} 
\boldsymbol{\lambda}:SH_*^{(a,b)}(\p V)\to \sum_{b_1+b_2=b}SH_*^{(\frac a 2,b_1)}(\p V)\otimes SH_*^{(\frac a 2,b_2)}(\p V).
\end{equation*}
By~\cite[Theorem 8.1]{CO-Tate}, passing to the direct limit over $b\to \infty$ and then to the inverse limit over $a\to-\infty$ we obtain the degree $-n+1$ {\em secondary coproduct} as a morphism
$$
\boldsymbol{\lambda}:SH_*(\p V)\to SH_*(\p V)\,\otimes^!\, SH_*(\p V).
$$ 
Standard arguments show that this coproduct is canonical, i.e., it does not depend on the various choices involved in the construction.
It follows directly from the definition that $\boldlambda$ is cocommutative on $S\H_*(\p V)$.

\subsection{The unital coFrobenius relation}

Let $\boldmu,\boldlambda,\boldeta$ be the operations on $S\H_*(\p V)$ constructed in the preceding subsections and define
$
  \boldc = \boldlambda\boldeta.
$
A standard neck-stretching argument shows that $\boldc$ is induced by counting configurations as in Figure~\ref{fig:Coproduct-RFH-limit} without the positive puncture. This directly implies the {\sc(symmetry)} $\tau \boldc = (-1)^{|\boldlambda|} \boldc$.
  
\begin{proposition}\label{prop:unital-coFrob}
The operations $\boldmu$ and $\boldc$ on $S\H_*(\p V)$ satisfy the {\sc (unital coFrobenius)} relation
$$
  \boldlambda=(1\otimes \boldmu)(\boldc\otimes 1) = (\boldmu\otimes 1)(1\otimes \boldc).
$$
\end{proposition}

\begin{proof}
For $\lambda<0<\mu$ and $|\lambda|\le \mu$ recall the Hamiltonians $H=H_{\lambda,\mu}$ and $L=H_{\lambda,\lambda}$ from above, which satisfy $-H\le L\le H$. Consider the new Hamiltonians $S=H_{-\lambda,2\mu}$ and $T=H_{2\lambda,-\mu}$ in Figure~\ref{fig:ST}. Recalling the notation $V_\delta$, $\delta>0$ from the beginning of the proof of Theorem~\ref{thm:coh-product}, their explicit description is the following: the Hamiltonian $S$ is constant equal to $\lambda/2$ on $V_{\frac12}$, it is linear of slope $-\lambda$ on $V\setminus V_{\frac 12}$, and it is linear of slope $2\mu$ outside of $V$; the Hamiltonian $T$ is constant equal to $\lambda$ on $V_{\frac 12}$, it is linear of slope $2\lambda$ on $V\setminus V_{\frac 12}$, and it is linear of slope $-\mu$ outside $V$. 
\begin{figure}
\begin{center}
\includegraphics[width=\textwidth]{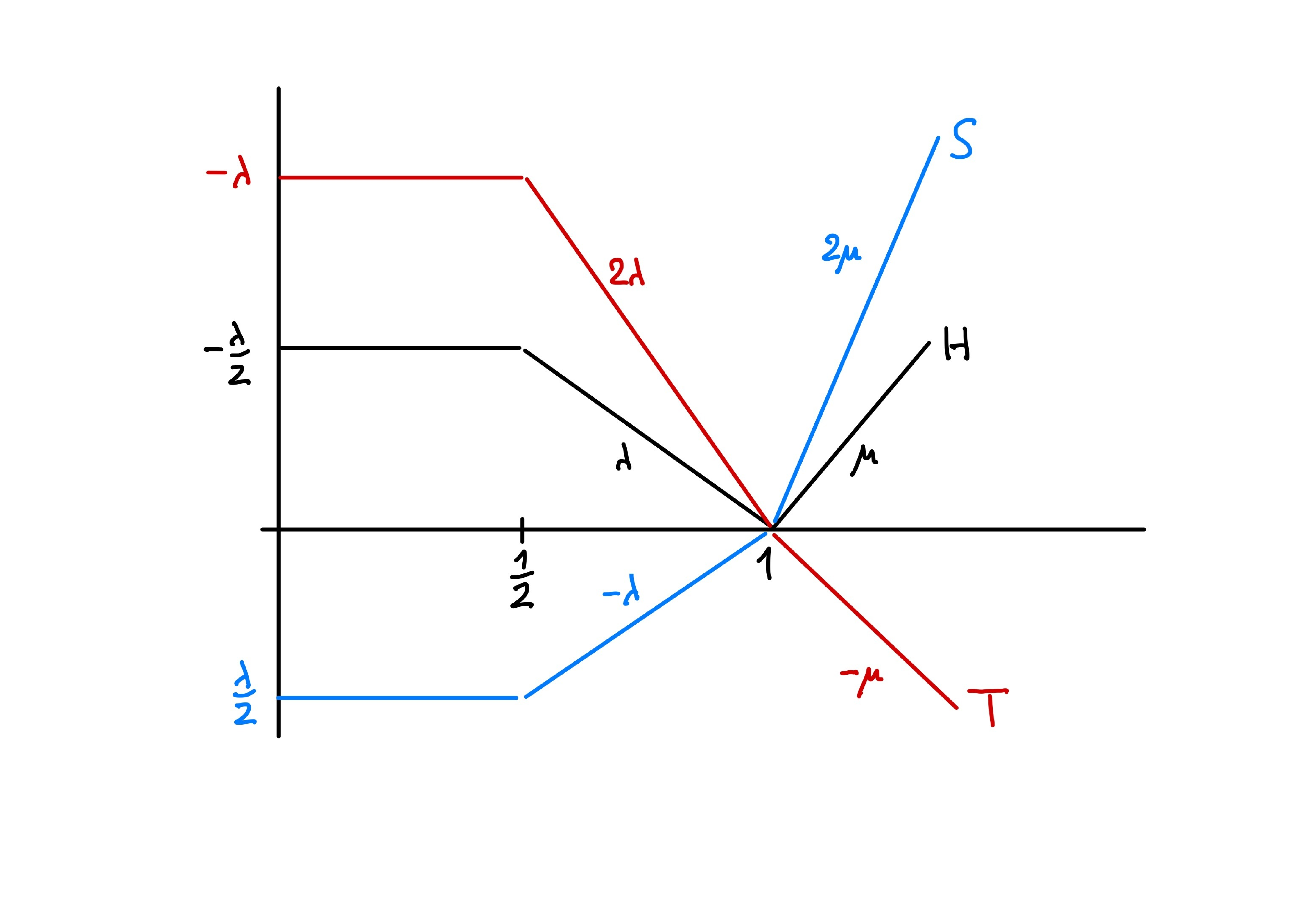}
\caption{The Hamiltonians $S$ and $T$.}
\label{fig:ST}
\end{center}
\end{figure}
They have the following properties:
\begin{enumerate}
\item $H=S+T$ and $-H\leq S,T\leq 2H$.
\item The count of Floer spheres with one positive puncture at $H$ and two negative punctures at $S$ and $T$ is well-defined. 
\item For fixed $-\infty<a<b<\infty$ and $\mu,-\lambda$ sufficiently large, $S$ and $T$ have no $1$-periodic orbits with action in $[a,b]$; the same holds for $(1-\tau)S-\tau L$ and $(1-\tau)T+\tau L$ with $\tau\in[0,1]$. 
\end{enumerate}
Consider now the moduli space of punctured Floer spheres shown in Figure~\ref{fig:coFrob1}.
\begin{figure}
\begin{center}
\includegraphics[width=.8\textwidth]{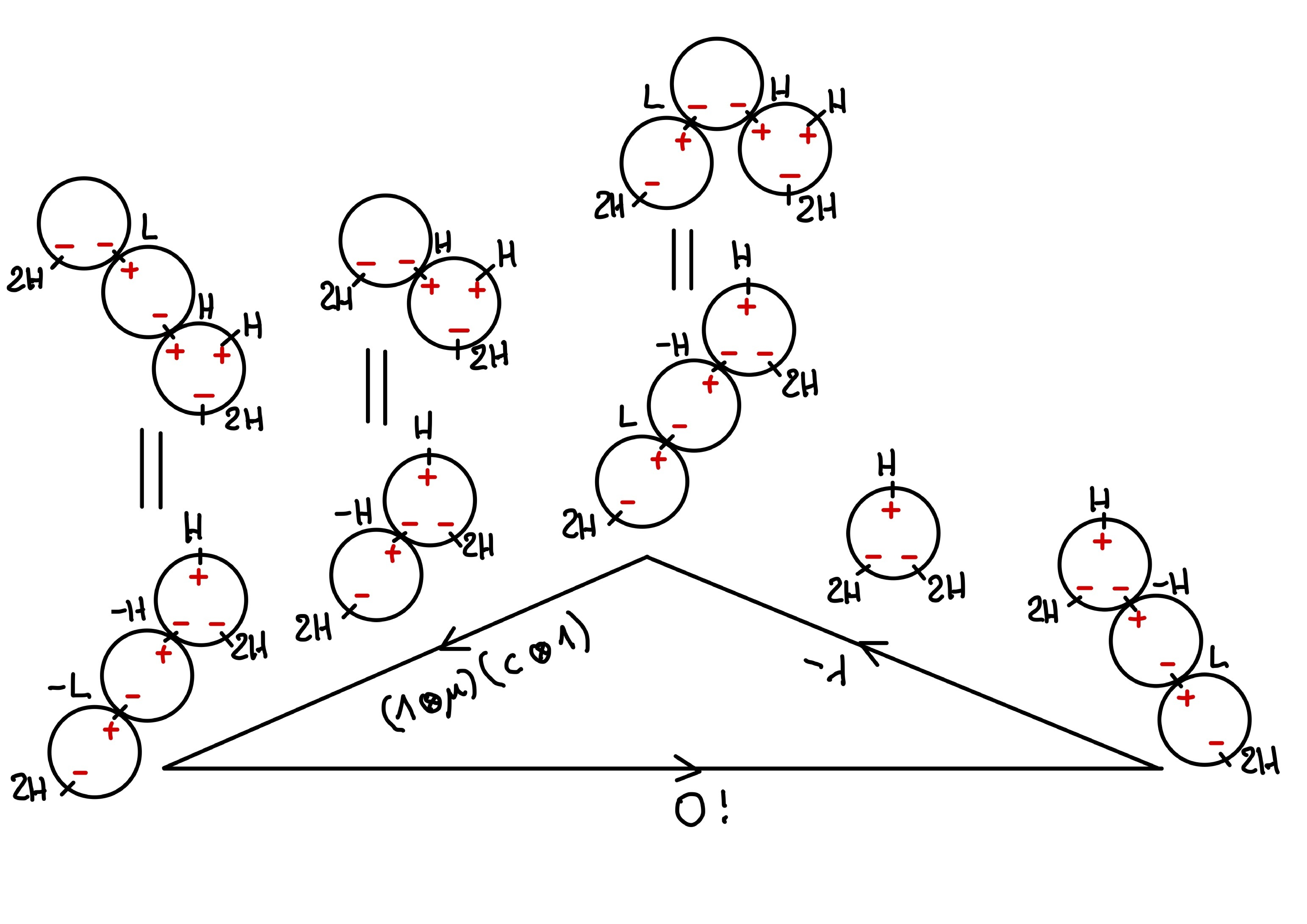}
\caption{Proof of the unital coFrobenius relation.}
\label{fig:coFrob1}
\end{center}
\end{figure}
Here the upper right side describes $-\boldlambda$ and the upper left side the composition $(1\otimes\boldmu)(\boldc\otimes 1)$. The lower side is described in Figure~\ref{fig:coFrob2}. 
\begin{figure}
\begin{center}
\includegraphics[width=.8\textwidth]{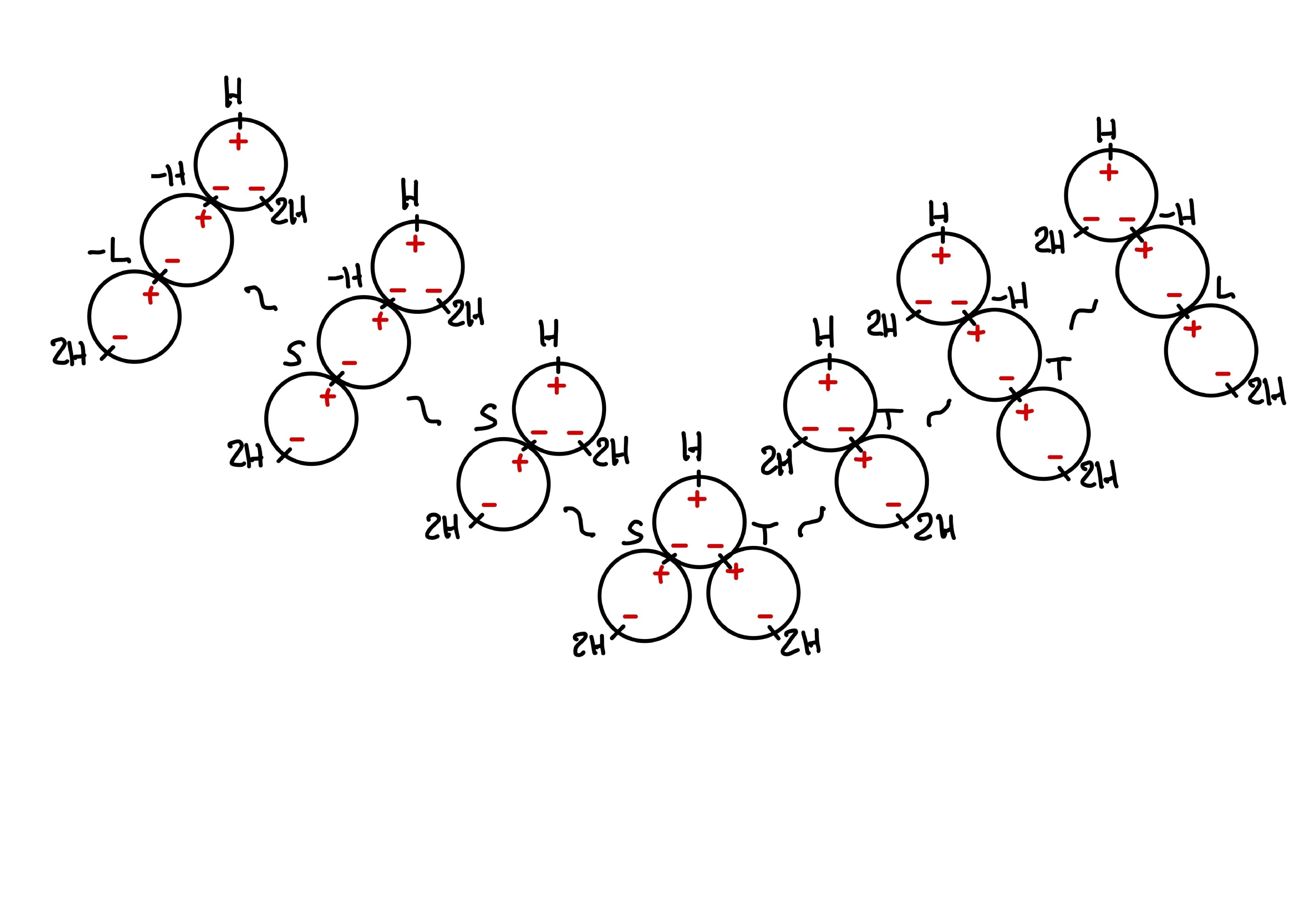}
\caption{Vanishing of the lower side.}
\label{fig:coFrob2}
\end{center}
\end{figure}
Here (going from left to right) the first deformation linearly interpolates between $-L$ and $S$, the second one glues the puncture at $-H$, the third one splits off the puncture at $T$, the fourth one glues the puncture at $S$, the fifth one splits off the puncture at $-H$, and the sixth one linearly interpolates between $T$ and $L$. In view of properties (1) and (2) above all the corresponding counts are well-defined. For a finite action window $(a,b)$ and $\mu,-\lambda$ sufficiently large, all these counts vanish by property (3) because we always have at least one puncture at $S$, $T$, $(1-\tau)S-\tau L$, or $(1-\tau)T+\tau L$ with $\tau\in[0,1]$. 
We extend the configurations over the interior of the triangle in Figure~\ref{fig:coFrob1} by 
first gluing everything in a small tubular neighborhood of the boundary - and then convexly interpolating the forms on the disc (all having $H$ as input and $2H$ as output).
This proves the first equality of the {\sc (unital coFrobenius)} relation in a fixed finite action window, and by taking inverse and direct limits it follows on $S\H_*(\p V)$.

The second equality can be proved analogously. Alternatively, it can be formally deduced from the first one by verifying the identity
$$
  \tau\boldlambda = (\boldmu\tau\otimes 1)(1\otimes\tau\boldc).
$$
and applying the symmetries
$$
  \tau\boldlambda=(-1)^{|\boldlambda|}\boldlambda,\qquad 
  \tau\boldc=(-1)^{|\boldlambda|}\boldc,\qquad 
  \boldmu\tau=(-1)^{|\boldmu|}\boldmu.
$$
\end{proof}

\subsection{The copairing is the inverse of the Poincar\'e duality isomorphism}

Let $\boldc=\boldlambda\boldeta$ be as in the preceding subsection and recall the notation $\vec\boldc=(\ev\otimes 1)(1\otimes\boldc)$. We shall sometimes refer to $\vec\boldc$ as being \emph{the secondary continuation map}. See~\cite{CHO-reducedSH} for a related discussion. 

\begin{proposition}\label{prop:c-iso}
The map $-\vec\boldc:SH^{1-*}(\p V)\to SH_*(\p V)$ is the inverse of the Poincar\'e duality isomorphism in Theorem~\ref{thm:coh-product}. 
\end{proposition}

\begin{proof}
Let $H=H_{\lambda,\mu}$ be a V-shaped Hamiltonian as in Figure~\ref{fig:Ham_RFH} and $L=H_{\lambda,\lambda}$ the corresponding linear Hamiltonian. By construction, the map $\vec\boldc$ counts cylinders in the $1$-parametric moduli space shown in Figure~\ref{fig:cvec}.

\begin{figure}
\begin{center}
\includegraphics[width=.7\textwidth]{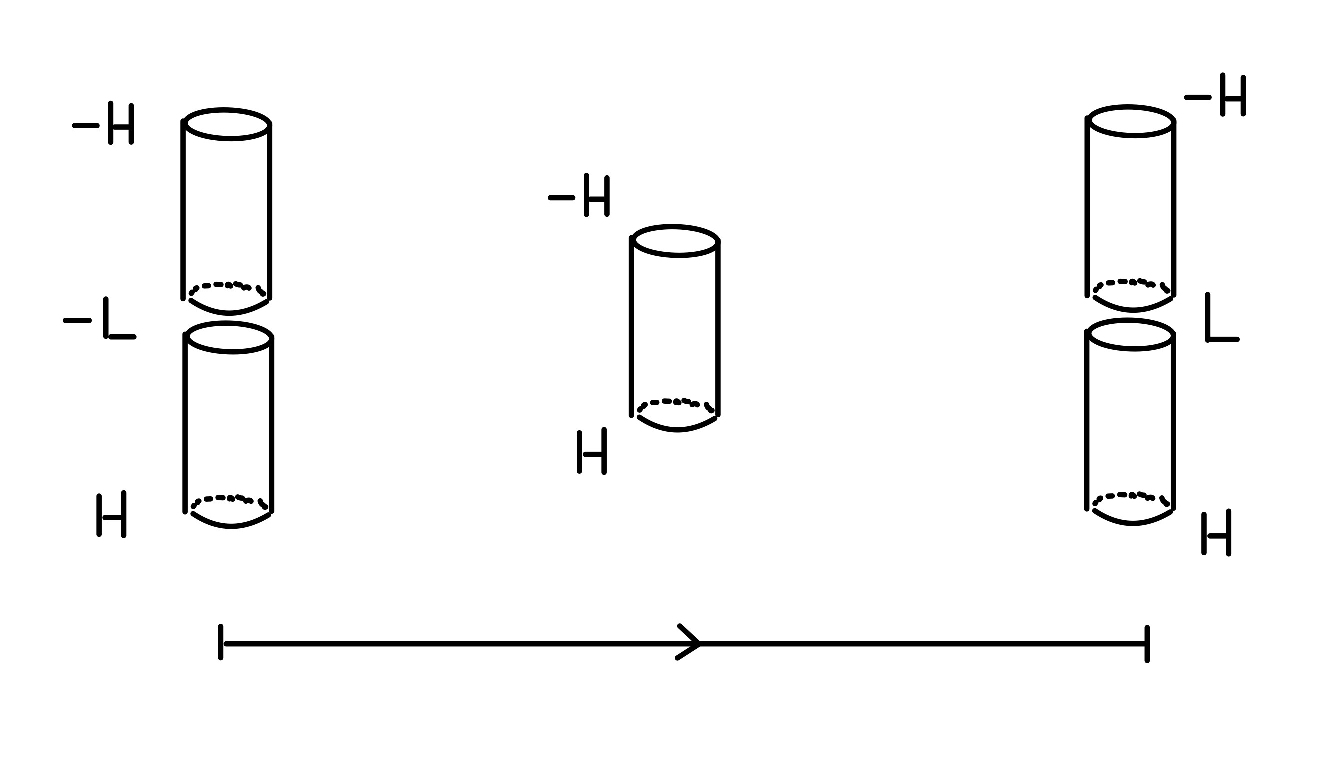}
\caption{The map $\vec\boldc$.}
\label{fig:cvec}
\end{center}
\end{figure}

As in the proof of Theorem~\ref{thm:coh-product}, we fix a finite action window $(a,b)$ and consider for $\tau\gg 0$ the dented Hamiltonian in Figure~\ref{fig:dent}, redrawn in Figure~\ref{fig:Z}. Let us denote this Hamiltonian by $Z$. Recall that its Floer chain complex in the action window $(a,b)$
has the form
$$
  FC^{(a,b)}(Z) = C^- \oplus C^+,\qquad \p_Z =
  \begin{pmatrix} \p^- & f \\ 0 & \p^+
  \end{pmatrix},  
$$
where $C^-$ and $C^+$ are generated by orbit groups III and II, respectively, and $f:C^+\to C^-$ is a map of degree $-1$. 

Let $K$ be the Hamiltonian which agrees with $Z$ to the left of orbit group III and with $H$ to its right, see Figure~\ref{fig:Z}. 
\begin{figure}
\begin{center}
\includegraphics[width=0.9\textwidth]{Z-evennewer.eps}\\ \qquad \\
\Hamlambdamuaxis{$-\tau$}{$\tau$} \qquad 
\HamLlambdaaxis{$-\tau$} \qquad 
\HamZaxis{$-\tau$}{$-\tau'$}{$\tau$} \qquad 
\HamKaxis{$-\tau$}{$-\tau'$}{$\tau$} \\ \qquad \\ \qquad 
\caption{The Hamiltonians $H,L,Z,K$.} 
\label{fig:Z}
\end{center}
\end{figure}
Monotone homotopies $Z\leadsto L\leadsto H$ and $Z\leadsto K\leadsto H$ give rise to a diagram of chain maps
$$
\xymatrix
@C=30pt
@R=20pt
{
  & FC^{(a,b)}(K) \ar[rd] &\\
  FC^{(a,b)}(Z) \ar[ru] \ar[rd] && FC^{(a,b)}(H) \\
  & FC^{(a,b)}(L) \ar[ru] &
}
$$
which commutes up to a degree $1$ chain homotopy
$$
  \Phi:\bigl(FC^{(a,b)}(Z),\p_Z\bigr) \to \bigl(FC^{(a,b)}(H),\p_H\bigr).
$$
This means that 
$$
  \p_H\Phi+\Phi\p_Z = F_1-F_0
$$
for the compositions
\begin{align*}
  &F_0:FC^{(a,b)}(Z)\to FC^{(a,b)}(L)\to FC^{(a,b)}(H),\cr
  &F_1:FC^{(a,b)}(Z)\to FC^{(a,b)}(K)\to FC^{(a,b)}(H).
\end{align*}
Writing $\Phi=(\Phi^-,\Phi^+)$ and $F_i=(F_i^-,F_i^+)$ with respect to the splitting $FC^{(a,b)}(Z)=C^-\oplus C^+$, the relation reads
$$
\left\{
\begin{aligned}
  \p_H\Phi^+ + \Phi^+\p^+ + \Phi^-f &= F_1^+ - F_0^+, \cr
  \p_H\Phi^- + \Phi^-\p^- &= F_1^- - F_0^-.
\end{aligned}
\right.
$$
The following 3 observations describe the chain maps $F_i^\pm$.  

(1) We have $F_0^\pm=0$ because the map $F_0$ factors through $FC^{(a,b)}(L)=\{0\}$.

(2) We have $F_1^-=0$. To see this, we show that the continuation map $C^-\to FC^{(a,b)}(K)$ vanishes. Consider an $s$-dependent Hamiltonian $H_s$ which agrees with $K$ near $-\infty$ and with $Z$ near $+\infty$ and satisfies $\p_sH_s\leq 0$. Suppose that there exists a Floer cylinder $u$ for $H_s$ with asymptotic orbits $x_+=(r_+,\gamma_+)$ in group III and $x_-$ in group II. 
If $\gamma_+$ is not constant, then, since $Z$ is concave near orbit group III, $u$ attains $r$-values $>r_+$ by~\cite[Lemma~2.3]{CO}. This is a contradiction to~\cite[Lemma~2.2]{CO}, which is applicable because in orbit group III the slope of $h_s$, and therefore also the action $A_{h_s}(r)$ at fixed $r$, decreases in $s$. If $\gamma_+$ is constant, then, since $Z$ is concave near orbit group III, the same argument applies using~\cite[Lemma~2.5]{CO}. 

(3) The map 
$$
  F_1^+:C^+
  \longrightarrow FC^{(a,b)}(K)
  \longrightarrow FC^{(a,b)}(H)
$$
is the identity under the canonical identifications of the chain complexes for a suitable choice of the homotopies. This is proved by analyzing separately each of these continuation maps. As we now explain, we only deform the Hamiltonians and keep the almost complex structure constant. 
\begin{itemize}
\item The first continuation map $C^+\to FC^{(a,b)}(K)$ can be chosen to be determined by a monotone homotopy $Z\leadsto K$ that is constant to the left of orbit group III. The almost complex structure is kept constant along the homotopy. By~\cite[Lemma~2.2]{CO}, the Floer continuation cylinders relevant for this continuation map must be entirely contained in the region where the homotopy $Z\leadsto K$ is constant, and they must therefore be constant. Since constant cylinders are regular for a constant almost complex structure, this continuation map is the identity. 
\item For the second continuation map, we keep the almost complex structure constant along the homotopy and we choose the slope $\tau'$ very close to $\tau$. Then $K$ and $H$ are $C^2$-close, and the following hold: 
\begin{itemize}
\item The generators of $FC(K)$ and $FC(H)$ are canonically identified. 
\item The continuation map induced by a generic $C^2$-small homotopy from $K$ to $H$ induces the canonical identification between $FC(K)$ and $FC(H)$. 
\end{itemize}
This is a consequence of the continuity and persistence, under $C^1$-small perturbations, of the zeroes of a transverse index 0 proper Fredholm section of a Banach bundle. For the first statement, the fixed points of a Hamiltonian diffeomorphism are intersections of its graph with the diagonal, and these can be seen locally as zeroes of an index 0 section of the normal bundle to the diagonal. For the second statement, the continuation Floer trajectories that connect two orbits with index difference 0 are the zeroes of the index 0 section given by the Cauchy-Riemann equation defined on a suitable space of maps. 
\end{itemize}

As a consequence of these observations, the second relation above says that $\Phi^-$ is a chain map
and the first relation becomes $\p_H\Phi^+ + \Phi^+\p^+ + \Phi^-f = 1$, so on homology the composition in the upper row of
$$
\xymatrix
@C=30pt
@R=20pt
{
  H_*(C^+) \ar@{=}[d] \ar[r]^{f_*} & H_{*-1}(C^-) \ar@{=}[d] \ar[r]^{\Phi^-_*} & FH_*^{(a,b)}(H) \ar@{=}[d] \\
  SH_*^{(a,b)}(\p V) \ar[r]^-{PD} & SH^{1-*}_{(-b,-a)}(\p V) \ar[r]^-{-\vec\boldc} & SH_*^{(a,b)}(\p V)
}
$$
is the identity. Commutativity of the first square follows from the definition of the Poincar\'e duality isomorphism in the proof of Theorem~\ref{thm:coh-product}. The second square commutes because the map $\Phi^-:C^-\to FC^{(a,b)}(H)$ interpolates between the factorizations through $FC^{(a,b)}(L)$ and $FC^{(a,b)}(K)$. 
We claim that this interpolation agrees with the definition of the secondary continuation map $-\vec\boldc$ (see Figure~\ref{fig:cvec}).
To prove the claim, consider the diagram of chain maps 
$$
\xymatrix
@C=30pt
@R=20pt
{ & & FC^{(a,b)}(L) \ar[dr] & \\
  & C^-\subset FC^{(a,b)}(Z) \ar[rd] \ar[ur] & -\Phi^- & FC^{(a,b)}(H) \\
  FC^{(a,b)}(-H) \ar@{=}[ru] \ar[rd] & 0 & FC^{(a,b)}(K) \ar@{=}[ur]\\
  & FC^{(a,b)}(-L) \ar[ru] &
}
$$
The identification $FC^{(a,b)}(K) \equiv FC^{(a,b)}(H)$ is the one described above. The image of $FC^{(a,b)}(-H)\to FC^{(a,b)}(Z)$ is contained in $C^-$ because the corresponding homotopy $-H\leadsto Z$ can be chosen to be constant to the right of orbit group II and we can apply~\cite[Lemmas~2.2, 2.3, 2.5]{CO}. 
As a consequence, the upper-right diamond describes $-\Phi^-$.
That the chain homotopy for commutativity of the bottom-left diamond is identically zero follows from the proof of observation (2) above, since the monotone homotopies in this $1$-parametric family differ from the one in (2) only to the left of orbit group 
III, which does not affect the argument. 
The claim follows because the bottom and top compositions from left to right in the diagram are the endpoints of $\vec\boldc$. 

We obtain $(-\vec\boldc)\circ PD = 1$ on $SH_*^{(a,b)}(\p V)$. In the inverse-direct limit we obtain that $(-\vec\boldc)\circ PD = 1$ on $SH_*(\p V)$, i.e. $-\vec\boldc$ is the inverse of the isomorphism $PD$. 
\end{proof}

\subsection{Proof of Theorems~\ref{thm:PD-sec-bialg-SH} and~\ref{thm:PD-sec-bialg-SH-Lag}}

\begin{proof}[Proof of Theorem~\ref{thm:PD-sec-bialg-SH}]
By \S\ref{sec:product}, the degree $0$ product $\boldmu$ on $S\H_*(\p V)$ is commutative and associative with unit $\boldeta$. The degree $1-2n$ coproduct $\boldlambda$ is defined in \S\ref{sec:secondary_coproduct}. By construction, $\boldlambda$ is cocommutative and the associated copairing $\boldc = \boldlambda\boldeta$ is symmetric. By Proposition~\ref{prop:unital-coFrob} the {\sc (unital coFrobenius)} relation holds, and by Proposition~\ref{prop:c-iso} the map $-\vec\boldc:S\H^{1-2n-*}(\p V)\to S\H_*(\p V)$ is the inverse of the Poincar\'e duality isomorphism. Thus $(\boldmu,\boldlambda,\boldeta)$ satisfies the conditions in Proposition~\ref{prop:coFrob-criterion} and it follows that $\boldeps=-\boldeta^\vee\vec\boldc^{-1}$ makes $(S\H_*(\p V),\boldmu,\boldlambda,\boldeta,\boldeps)$ a graded Frobenius algebra. Finally, it follows from Theorem~\ref{thm:algebraic_Poincare_duality} that $-\vec\boldc$ and its inverse $-\vec\boldp = PD$ intertwine the graded Frobenius algebra structures on $S\H_*(\p V)$ and $S\H^{-*}(\p V)$. 
\end{proof}

\begin{proof}[Proof of Theorem~\ref{thm:PD-sec-bialg-SH-Lag}]
The proof is up to notation the same as that of Theorem~\ref{thm:PD-sec-bialg-SH}, using the 
gradings in Appendix~\ref{sec:grading}. Note that commutativity of $\boldmu$ and cocommutativity of $\boldlambda$ need not hold in this case, but symmetry of $\boldc$ is still satisfied because $\boldc$ is counting disks with only two boundary punctures. 
\end{proof}

\section{Graded open-closed TQFT structure}\label{sec:TQFT}

In this section we show that the graded Frobenius algebra structure on V-shaped symplectic homology and its Lagrangian counterpart from \S\ref{sec:sec-bialg} fit together to a graded version of a two-dimensional open-closed TQFT.
Throughout this section we use coefficients in a field $\bk$.

\subsection{Two-dimensional open-closed TQFTs}

We begin by recalling the description of a two-dimensional open-closed TQFT in terms of generators and relations from Lauda and Pfeiffer~\cite{Lauda-Pfeiffer}. It associates to the circle a finite dimensional $\bk$-vector space $C$ and to the closed interval a finite dimensional $\bk$-vector space $A$. Its generators are the following operations (see Figure~\ref{fig:TQFT-generators}):
\begin{itemize}
\item the product $\boldmu_C:C\otimes C\to C$ and the unit $\boldeta_C:\bk\to C$;
\item the coproduct $\boldlambda_C:C\to C\otimes C$ and the counit $\boldeps_C:C\to \bk$;
\item the product $\boldmu_A:A\otimes A\to A$ and the unit $\boldeta_A:\bk\to A$;
\item the coproduct $\boldlambda_A:A\to A\otimes A$ and the counit $\boldeps_A:A\to \bk$;
\item the {\em zipper} (or closed-open map) $\boldzeta:C\to A$;
\item the {\em cozipper} (or open-closed map) $\boldzeta^*:A\to C$.
\end{itemize}
These satisfy the following relations (in our terminology):
\begin{enumerate}
\item $(C,\boldmu_C,\boldeta_C,\boldlambda_C,\boldeps_C)$ is a commutative and cocommutative Frobenius algebra. 
\item $(A,\boldmu_A,\boldeta_A,\boldlambda_A,\boldeps_A)$ is a Frobenius algebra. 
\item The zipper is an algebra homomorphism, 
$$
  \boldmu_A(\boldzeta\otimes\boldzeta) = \boldzeta\,\boldmu_C,\qquad \boldzeta\,\boldeta_C=\boldeta_A.   
$$
\item The zipper lands in the center of $\boldmu_A$, 
$$
  \boldmu_A(\boldzeta\otimes 1) = \boldmu_A\tau(\boldzeta\otimes 1).  
$$
\item The cozipper is dual to the zipper via the copairings $\boldc_C=\boldlambda_C\boldeta_C$ and $\boldc_A=\boldlambda_A\boldeta_A$, 
$$
  (1\otimes\boldzeta)\boldc_C = (\boldzeta^*\otimes 1)\boldc_A.
$$
\item The {\em Cardy condition}
$$
  \boldzeta\,\boldzeta^* = \boldmu_A\tau\boldlambda_A. 
$$
\end{enumerate}
In~\cite{Lauda-Pfeiffer} condition (5) is stated as the equivalent condition in Lemma~\ref{lem:cozipper-dual-zipper} below (without signs), and it is proved that a 2D open-closed TQFT is equivalent to the algebraic structure defined by these generators and relations. 

\begin{figure}
\begin{center}
\includegraphics[width=.9\textwidth]{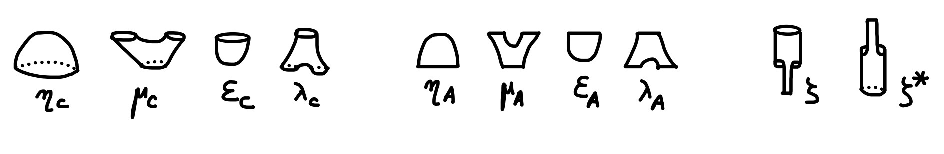}
\caption{The generators of an open-closed TQFT.}
\label{fig:TQFT-generators} 
\end{center}
\end{figure}

\subsection{Graded open-closed TQFTs}\label{ss:graded-open-closed-TQFT}

In this subsection we adapt the description of a two-dimensional open-closed TQFT to the $\Z$-graded setting, also dropping the finite dimensionality conditions on $C$ and $A$. Thus $C$ and $A$ are now graded Tate vector spaces.

After shifting degrees, we can and will assume that the products $\boldmu_C$ and $\boldmu_A$ have degree $0$. Then relations (1)--(5) determine all the other degrees in terms of the degrees of $\boldlambda_C$ and $\boldlambda_A$: 
\begin{gather*}
  |\boldmu_C| = |\boldeta_C| = |\boldmu_A| = |\boldeta_A| = |\boldzeta| = 0, \cr
  |\boldlambda_C|=|\boldc_C|=-|\boldeps_C|, \qquad 
  |\boldlambda_A|=|\boldc_A|=-|\boldeps_A|, \cr
  |\boldzeta^*| = |\boldlambda_C| - |\boldlambda_A|. 
\end{gather*}
Relation (6) yields the additional condition $|\boldzeta^*|=|\boldlambda_A|$, which combined with the previous one implies that $|\boldlambda_C|=2|\boldlambda_A|$ is even. This gives rise to an algebraic structure which recovers the one of the previous subsection in the case that $|\boldlambda_A|$ is also even. 

On degree shifted symplectic homology $S\H_*(\p V)$, however, the coproduct $\boldlambda_C$ has {\em odd} degree $1-2n$, so the Cardy condition makes no sense in its original form. This leads us to the following definition.

\begin{definition}[{cf.~\cite[Definition~6.1]{CHO-algebra}}]\label{defi:gradedTQFT}
A {\em graded (two-dimensional) open-closed TQFT} consists of two graded Tate vector spaces $C,A$ 
and the following morphisms:
\begin{itemize}
\item the degree $0$ product $ \boldmu_C:C\otimes^* C\to C$ and the unit $\boldeta_C:\bk\to C$;
\item the coproduct $\boldlambda_C:C\to C\otimes^! C$ and the counit $\boldeps_C:C\to \bk$;
\item the degree $0$ product $ \boldmu_A:A\otimes^* A\to A$ and the unit $\boldeta_A:\bk\to A$;
\item the coproduct $\boldlambda_A:A\to A\otimes^! A$ and the counit $\boldeps_A:A\to \bk$;
\item the {\em zipper} (or closed-open map) $\boldzeta:C\to A$;
\item the {\em cozipper} (or open-closed map) $\boldzeta^*:A\to C$.
\end{itemize} 
These are required to satisfy the following signed versions of relations (1)--(6) above: 
\begin{enumerate}
\item $(C, \boldmu_C,\boldeta_C,\boldlambda_C,\boldeps_C)$ is a commutative and cocommutative graded\break  Frobenius algebra, called \emph{the closed sector}. 
\item $(A, \boldmu_A,\boldeta_A,\boldlambda_A,\boldeps_A)$ is a graded Frobenius algebra, called \emph{the open sector}. 
\item The zipper is an algebra homomorphism, 
$$
   \boldmu_A(\boldzeta\otimes\boldzeta) = \boldzeta\, \boldmu_C,\qquad \boldzeta\,\boldeta_C=\boldeta_A.   
$$
\item The zipper lands in the center of $ \boldmu_A$, 
$$
  \boldmu_A(\boldzeta\otimes 1) = \boldmu_A\tau(\boldzeta\otimes 1).  
$$
\item The cozipper is dual to the zipper via the copairings $\boldc_C=\boldlambda_C\boldeta_C$ and $\boldc_A=\boldlambda_A\boldeta_A$, 
$$
  (1\otimes\boldzeta)\boldc_C = (\boldzeta^*\otimes 1)\boldc_A.
$$
\item The {\em graded Cardy condition}
$$
\left\{\begin{array}{ll}
\boldzeta\,\boldzeta^*=(-1)^{|\lambda_A|} \boldmu_A\tau\boldlambda_A & 
\mbox{ if } |\boldlambda_C|=2|\boldlambda_A| \mbox{ and } \dim A<\infty, 
\\
\boldzeta\,\boldzeta^*= \boldmu_A\tau\boldlambda_A=0
& \mbox{ if } |\boldlambda_C|\neq 2|\boldlambda_A| \mbox{ and } \dim A<\infty, 
\\
\boldzeta\,\boldzeta^*=0
& \mbox{ if } |\boldlambda_C|\neq 2|\boldlambda_A|  \mbox{ and } \dim A=\infty. 
\end{array}\right.
$$ 
\end{enumerate}
\end{definition}

\begin{lemma}[{\cite[Lemma~6.6]{CHO-algebra}}] \label{lem:cozipper-dual-zipper}
Assuming relations (1) and (2), relation (5) is equivalent to the following relation in terms of the pairings $\boldp_C= (-1)^{|\boldlambda_C|}\boldeps_C\boldmu_C$ and $\boldp_A=(-1)^{|\boldlambda_A|}\boldeps_A\boldmu_A$:
\begin{equation*}\label{eq:cozipper-dual-zipper}
  \boldp_C(1\otimes\boldzeta^*) = (-1)^{|\boldlambda_A|+|\boldlambda_C|} \boldp_A(\boldzeta\otimes 1). 
\end{equation*}
\end{lemma}

Consider now a Liouville domain $V$ of dimension $2n$ and an exact Lagrangian submanifold $L\subset V$ with Legendrian boundary $\p L\subset\p V$. We define the zipper (or closed-open map) and cozipper (or open-closed map)
$$
\boldzeta : S\H_*(\p V)\to S\H_*(\p L), \qquad \boldzeta^*:S\H_*(\p L)\to S\H_{*-n}(\p V)
$$
by counts of discs with one positive interior puncture and one negative boundary puncture, respectively one positive boundary puncture and one negative interior puncture, as shown in Figure~\ref{fig:zip-cozip-def}. 
\begin{figure}
\begin{center}
\includegraphics[width=.4\textwidth]{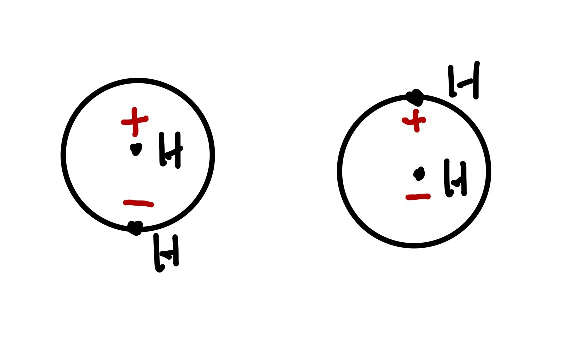}
\caption{The zipper and cozipper in symplectic homology.}
\label{fig:zip-cozip-def} 
\end{center}
\end{figure}

\begin{theorem}[Graded open-closed TQFT structure on symplectic homology]\label{thm:open-closed-RFH}
Let $V$ be a Liouville domain of dimension $2n$ and $L\subset V$ an exact Lagrangian submanifold with Legendrian boundary $\p L\subset\p V$. Then the graded Frobenius algebra structures on $S\H_*(\p V)$ and $S\H_*(\p L)$, together with the maps $\boldzeta$ and $\boldzeta^*$, fit together into a graded open-closed TQFT structure, with coproducts of degrees $|\boldlambda_C|=1-2n$ and $|\boldlambda_A|=1-n$. The Poincar\'e duality isomorphisms intertwine this structure with the corresponding structure on cohomology $S{\H}^{1-2n-*}(\p V)$ and $S{\H}^{1-n-*}(\p L)$. 
\end{theorem}

\begin{proof}
Relations (1) and (2) are Theorems~\ref{thm:PD-sec-bialg-SH} and~\ref{thm:PD-sec-bialg-SH-Lag}, respectively.
The proofs of relations (3) and (4) are shown in Figures~\ref{fig:zipper-algebra} and~\ref{fig:zipper-center}, respectively.\begin{figure}
\begin{center}
\includegraphics[width=.75\textwidth]{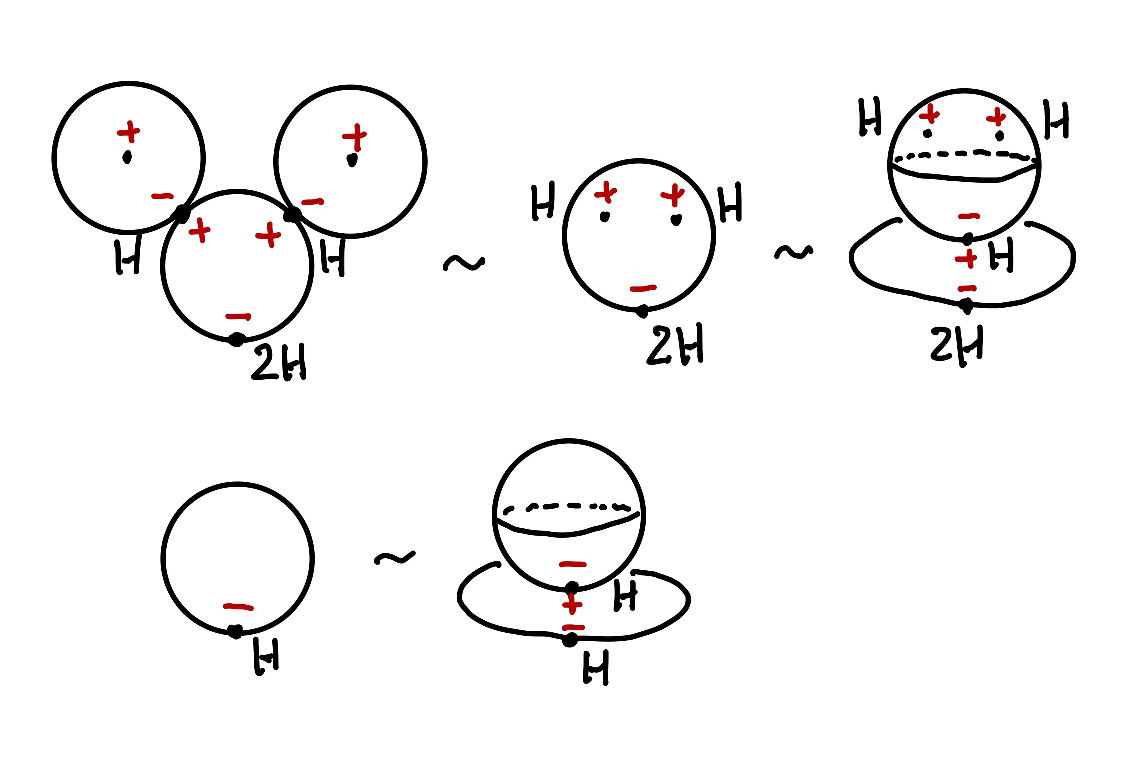}
\caption{The zipper is an algebra homomorphism.}
\label{fig:zipper-algebra} 
\end{center}
\end{figure}
\begin{figure}
\begin{center}
\includegraphics[width=.9\textwidth]{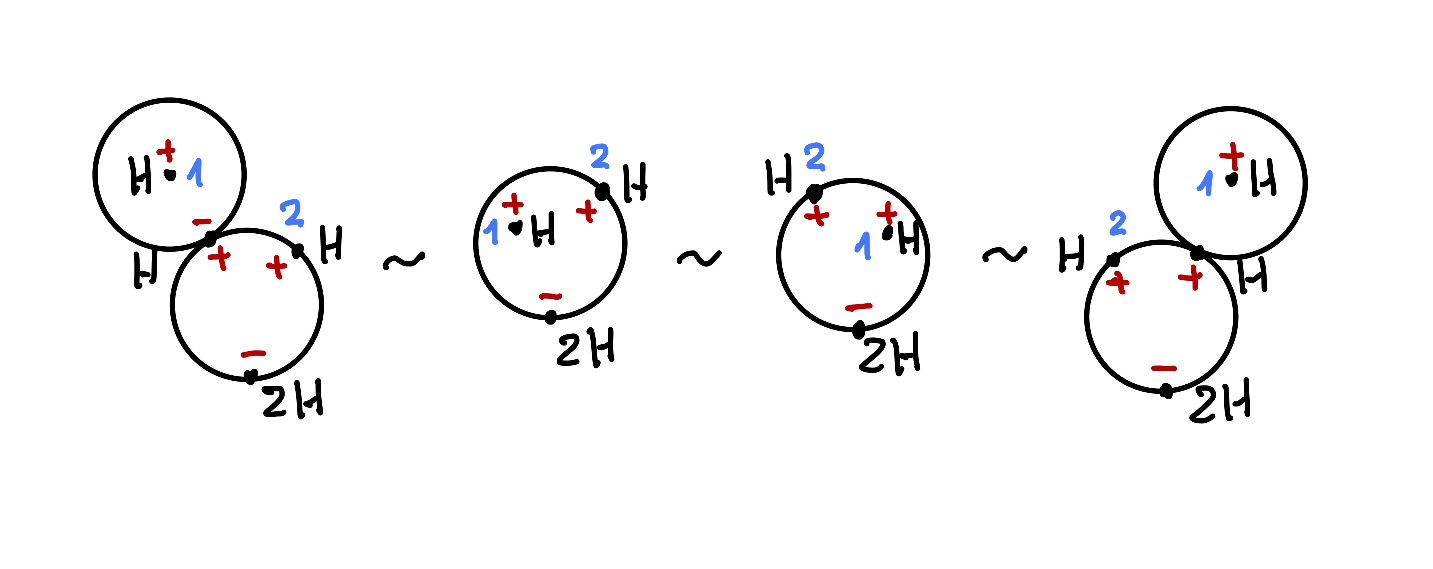}
\caption{The zipper lands in the center.}
\label{fig:zipper-center} 
\end{center}
\end{figure}
The proof of relation (5) is shown in Figure~\ref{fig:cozipper-dual-zipper}. There we consider a moduli space parametrized by a hexagon, where the left side corresponds to $(\boldzeta^*\otimes 1)\boldc_A$ and the right side to $-(1\otimes\boldzeta)\boldc_C$. The operations on the other four sides factor through the Hamiltonian $L$ and therefore vanish because $L$ has no periodic orbits in a given finite action window.
\begin{figure}
\begin{center}
\includegraphics[width=.9\textwidth]{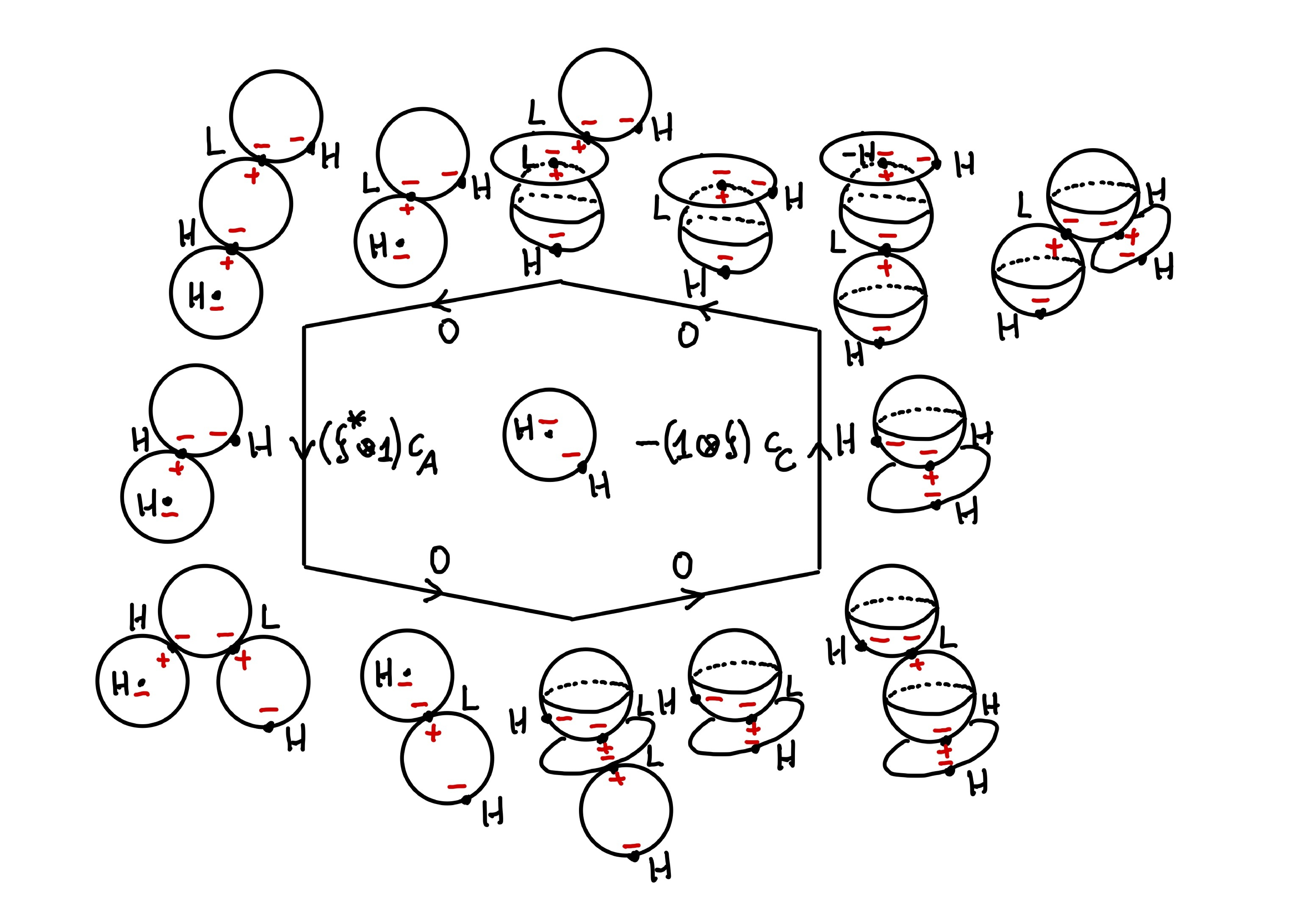}
\caption{The cozipper is dual to the zipper.}
\label{fig:cozipper-dual-zipper} 
\end{center}
\end{figure}
The proof of relation (6) is shown in Figure~\ref{fig:cardy}. Since $\boldlambda$ has odd degree it reads $\boldzeta\,\boldzeta^*=0$, and in addition $\boldmu_A\tau\boldlambda_A=0$ if $\dim S\H_*(\p L)<\infty$. Now $\boldzeta\,\boldzeta^*=0$ follows from the figure because the Hamiltonian $L$ has no periodic orbits in a given finite action window. The composition $\boldmu_A\tau\boldlambda_A$ is given by the right hand picture in the figure with $L$ replaced by $H$, which vanishes because we can deform $H$ back to $L$.  
\begin{figure}
\begin{center}
\includegraphics[width=.75\textwidth]{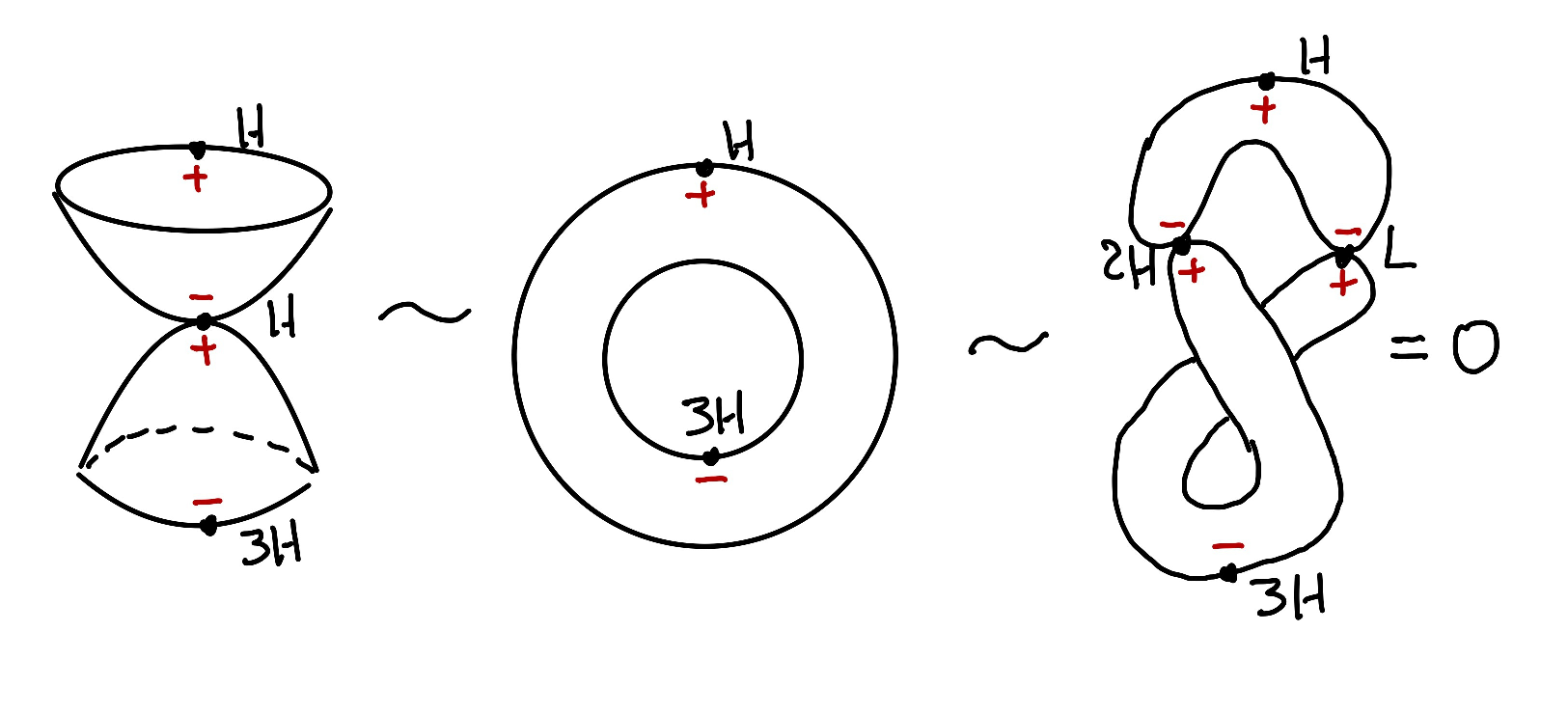}
\caption{The graded Cardy condition.}
\label{fig:cardy} 
\end{center}
\end{figure}
\end{proof}

\subsection{Topological open-closed and closed-open maps}  \label{sec:topological-oc-co}

Now we specialize to the pair $(V,L)=(D^*M,D^*_qM)$, so that we work with Rabinowitz loop homology $\wh{\H}_*\Lambda=S\H_*(S^*M)$ and Rabinowitz based loop homology $\wh H_*\Omega=S\H_*(S^*_qM)$. In this subsection we identify the closed-open and open-closed maps $\boldzeta,\boldzeta^*$ described above with their topological counterparts, the topological shriek and pullback/pushforward maps induced by the inclusion $i:\Om\into\Lambda$, denoted $i_!$, $i^*$ etc. Note that the latter are also well-defined at the level of reduced (co)homology groups.

\begin{proposition}\label{prop:topological_shriek} 
With respect to the splittings for $\wh{H}_*\Lambda$ and $\wh{H}_*\Omega$ in Theorem~\ref{thm:splitting-intro}, the following hold:
  
(a) The zipper $\boldzeta:\wh{H}_*\Lambda\to \wh{H}_{*-n}\Omega$  
restricts to maps $\ol{H}_*\Lambda\to H_{*-n}\Omega$ and $\ol{H}^{-*+1}(\Lambda)\to H^{-*+1}\Omega$ which coincide respectively with the topological shriek map $i_!$ and the cohomological pullback map $i^*$. 

(b) The cozipper $\boldzeta^*:\wh{H}_*\Omega\to \wh{H}_*\Lambda$  
 restricts to maps $H_*\Omega\to \ol{H}_*\Lambda$ and $H^{-*+1-n}\Omega\to \ol{H}^{-*+1}\Lambda$ which coincide respectively with the topological pushforward map $i_*$ and the cohomological shriek map $i^!$. 
\end{proposition}

\begin{proof}[Proof of Proposition~\ref{prop:topological_shriek}] 
We prove only assertion (a), the proof of (b) being similar. 
The topological shriek map $i_!:H_*(\Lambda)\to H_{*-n}(\Omega)$ descends to $i_!:\ol{H}_*(\Lambda)\to H_{*-n}(\Omega)$ because the point class lies in its kernel for degree reasons. It admits the following description in Morse theory (see~\cite{AS-product-structures,CHO-MorseFloerGH}). 
Consider a smooth Lagrange function $L:S^1\times TM\to\R$ which
outside a compact set has the form $L(t,q,v)=\frac12|v|^2-V_\infty(t,q)$ for
a smooth potential $V_\infty:S^1\times M\to\R$. It induces a smooth action
$$
   S_L:\Lambda\to\R,\qquad q\mapsto\int_0^1L(t,q,\dot q)dt, 
$$
which we can assume to be a Morse function whose negative flow with
respect to the $W^{1,2}$-gradient $\nabla S_L$ is Morse--Smale. 

View $\Omega\subset \Lambda$ as a codimension $n$ Hilbert submanifold and choose the base point $q$ generic so that all critical points of $S_L$ lie outside $\Omega$. We can further perturb $V_\infty$ so that $S_L|_\Omega$ is a Morse function and the following moduli spaces are cut out transversally. Given $a\in\mathrm{Crit}(S_L)$ and $b\in \mathrm{Crit}(S_L|_\Omega)$ set
$$
   \cM(a;b):=W^-(a)\cap W^+(b),
$$
where $W^-(a)$ is the unstable manifold of $a$ with respect to the flow of $-\nabla S_L$ on $\Lambda$, and $W^+(b)$ is the stable manifold of $b$ with respect to the flow of $-\nabla(S_L|_\Omega)$ on $\Om$.
Generically, this is a smooth manifold of dimension $\dim\cM(a,b)=\ind(a)-\ind(b)-n$, where $\ind(a)$ and $\ind(b)$ denote the Morse indices of $a$ and $b$ with respect to $S_L$ and $S_L|_\Omega$, respectively. This manifold is compact up to breaking of Morse trajectories, and it is already compact if its dimension is $0$. The oriented count of the elements of the 0-dimensional moduli spaces $\cM(a,b)$, $a\in\mathrm{Crit}(S_L)$, $b\in \mathrm{Crit}(S_L|_\Omega)$ defines the map $i_!$. 

The zipper $\boldzeta:\ol{H}_*\Lambda\to H_{*-n}\Omega$ can alternatively be described using the usual Hamiltonian profile for the symplectic homology $SH_*(T^*M)$, i.e. a Hamiltonian that is constant equal to zero on $D^*M$ and linear of positive slope with respect to the radial coordinate outside $D^*M$. 

We use the isomorphism $\Psi_\Lambda:SH_*(D^*M)\stackrel{\simeq}\longrightarrow H_*\Lambda$ from~\cite{AS-product-structures,Abouzaid-cotangent}. It is given by a count of mixed configurations consisting of a disc with boundary on the zero section, one marked point on the boundary and one positive interior puncture, solving a perturbed Cauchy-Riemann equation, together with a semi-infinite descending gradient line in $\Lambda$ starting at the loop determined by the restriction of the disc to its boundary.

We also use the isomorphism $\Psi_\Omega:SH_*(D^*_qM)\stackrel{\simeq}\longrightarrow H_{*-n}\Omega$ with $\Omega=\Omega_qM$, i.e., the Lagrangian counterpart of $\Psi_\Lambda$ defined in~\cite{AS-product-structures,Abouzaid-cotangent}.
It is given by a count of mixed configurations consisting of a disc with three boundary punctures, with one boundary component constrained to $M$, the other two boundary components constrained to $T^*_qM$, a positive puncture at the end bordered by the two $T^*_qM$-components, solving a perturbed Cauchy-Riemann equation, together with a semi-infinite descending gradient line in $\Omega$ starting at the loop determined by the restriction of the strip to the boundary component which is constrained to $M$. Note that the strip must be asymptotic to $q$ at each of the punctures bordered by $M$ and $T^*_qM$, since $q$ is the unique intersection point of $M$ and $T^*_qM$.  
Recall that in $H_*\Lambda$ and $H_{*-n}\Omega$ we use twisted coefficients as described in
Appendix~\ref{ss:local-systems}.

We need to show that the following diagram commutes:
$$
\xymatrix{
   SH_*(D^*M)  \ar[r]^{\boldzeta}\ar[d]_{\Psi_\Lambda}^\simeq & SH_*(D^*_qM) \ar[d]^{\Psi_\Omega}_\simeq\\
   H_*\Lambda \ar[r]^{i_!}& H_{*-n}\Omega.
}
$$
We prove that both compositions are equal to the map $\Gamma:SH_*(D^*M)\break \to H_{*-n}\Omega$ induced by the count of moduli spaces consisting of a disc with boundary on the zero section, one marked point on the boundary and one positive interior puncture, solving a perturbed Cauchy-Riemann equation, such that the 
marked point on the boundary is constrained to be equal to $q$, together with a semi-infinite descending gradient line in $\Omega$ starting at the loop determined by the restriction of the disc to its boundary. 

We first discuss the composition $i_!\circ \Psi_\Lambda$. After gluing, the composition is described by an obvious moduli space involving (starting at the component which contains the positive puncture) a disc with boundary constrained to $M$ and with one marked point on the boundary, a finite length descending gradient line in $\Lambda$ whose lowest energy point is a loop based at $q$, and a semi-infinite descending gradient line in $\Omega$. We bring the length of the intermediate cylinder to zero in a $1$-parameter family. This produces a homotopic operation and we observe that at the zero-length end of the homotopy we recover the map $\Gamma$. 

We now discuss the composition $\Psi_\Omega\circ \boldzeta$. After gluing, it is described by the count of mixed configurations consisting of a disc with one interior positive puncture, two boundary punctures, one boundary arc constrained to $M$, the other boundary arc constrained to $T^*_qM$, solving a perturbed Cauchy-Riemann equation, together with a semi-infinite descending gradient line in $\Omega$ starting at the loop determined by the restriction of the disc to the boundary arc which is constrained to $M$. Note that the disc must be asymptotic to $q$ at each boundary puncture, since $q$ is the unique intersection point of $M$ and $T^*_qM$. We vary the conformal type of the disc in a $1$-parameter family within its moduli space by bringing the two boundary punctures together and shrinking the boundary arc labeled $T^*_qM$, see Figure~\ref{fig:annuli-one-parameter}. This produces a homotopic operation. We find at the other end of the moduli space a nodal disc with two irreducible components $D_1$ and $D_2$, together with a semi-infinite descending gradient line in $\Omega$. The irreducible component $D_1$ of the nodal disc contains the interior puncture, it has the node on its boundary, and the boundary is labeled $M$. The irreducible component $D_2$ contains the two boundary punctures and the node, the boundary arcs adjacent to the node are labeled $M$ and the third boundary arc is labeled $T^*_qM$. On both components $D_1$ and $D_2$ the resulting curves $u_1$ and $u_2$ solve a Cauchy-Riemann equation perturbed by a non-negative $1$-form and a Hamiltonian $H$ as in the definition of symplectic homology which vanishes near $M$. The maximum principle forces the curve $u_2$ to be contained in $D^*M$, where the Hamiltonian vanishes, so that the curve solves a genuine, unperturbed Cauchy-Riemann equation. Since both $M$ and $T^*_qM$ are exact it follows that $u_2$ is constant, necessarily equal to $q$.  (This fact is akin to~\cite[Ch.~13, Exercise~5.3]{Abouzaid-cotangent}.)   
As a consequence, the node on the boundary of $u_1$ is sent to the point $q$ and we find that the count of such moduli spaces defines the map $\Gamma$.
\end{proof}

\begin{figure} [ht]
\centering
\input{annuli-one-parameter.pstex_t}
\caption{The $1$-dimensional moduli space of discs with 1 interior puncture and 2 boundary punctures.}
\label{fig:annuli-one-parameter}
\end{figure}

\section{Proof of the main theorems on loop spaces}\label{sec:loop-proofs}

In the preceding sections we have proved our main results in the setting of general Liouville domains. Their application to string topology is based on variations and refinements of Viterbo's isomorphism~\cite{Viterbo-cotangent}, which we now describe. 

In this section, $M$ denotes a closed connected manifold with free loop space $\Lambda$ and based loop space $\Om$. We denote by $T^*M$ the unit cotangent bundle with its canonical Liouville form, and by $D^*M\subset T^*M$ the unit disc bundle for some choice of Riemannian metric on $M$. This is a Liouville domain with boundary the unit sphere bundle $S^*M$.
We use coefficients in a unital commutative ring $R$.


The connection between Floer theory and the topology of free loops is provided by the following theorem, which is the result of a joint effort of many people. Let $\sigma$ be the local system on loops on $T^*M$ determined by the second Stiefel-Whitney class on $M$, let $\tilde\co=\ev_0^*\co$ the pull-back to loops on $M$ of the orientation local system $\co$ on $M$, and recall that we denote $H_*\Lambda=H_*(\Lambda;\tilde\co)$. In the sequel we compute Hamiltonian Floer homology groups with coefficients in $\sigma$, and Lagrangian Floer homology groups with coefficients in $\Sigma|_\Omega$. See also the discussion in the Introduction.

\begin{theorem}[]\label{thm:Viterbo}
(a)~\cite{Viterbo-cotangent,AS,AS-corrigendum,AS2,SW,Ritter,Cieliebak-Latschev,Kragh,Abouzaid-cotangent} There are isomorphisms of $R$-modules
\begin{equation*}
   SH_*(D^*M)\simeq H_*\Lambda,\qquad SH^*(D^*M)\simeq H^*\Lambda. 
\end{equation*}
The first isomorphism intertwines the pair-of-pants product on symplectic homology with the loop product on loop space homology.

(b)~\cite{CHO-MorseFloerGH} The second isomorphism induces  
$
SH^*_{>0}(D^*M)\simeq H^*(\Lambda,\Lambda_0),
$
an isomorphism which intertwines the secondary pair-of-pants product on positive symplectic cohomology with the loop product on loop space cohomology rel $\Lambda_0$.  
\hfill\qed
\end{theorem}

The connection between Floer theory and the topology of based loops is provided by the following theorem, which is the Lagrangian analogue of Theorem~\ref{thm:Viterbo}.

\begin{theorem}[]\label{thm:Viterbo-Lag}
(a)~\cite{Abbondandolo-Portaluri-Schwarz,AS2,AS-corrigendum,Abouzaid2012a} There are isomorphisms of $R$-modules
\begin{equation*}
   SH_{*+n}(D^*_qM)\simeq H_*\Omega,\qquad SH^{*+n}(D^*_qM)\simeq H^*\Omega. 
\end{equation*}
The first isomorphism intertwines the pair-of-pants product on wrapped Floer homology with the Pontrjagin product on based loop homology.

(b)~\cite{CHO-MorseFloerGH} The group $SH^{*+n}(D^*_qM)$ carries a secondary product which extends the one on $SH^{*+n}_{>0}(D^*_qM)$ defined in~\cite{CO}, the group $H^*\Omega$ carries a secondary product which extends the one on $H^*(\Omega,\{q\})$ defined in~\cite{Goresky-Hingston}, and the second isomorphism intertwines these two products.
\hfill\qed
\end{theorem}

For both these results we refer to the Appendix~\ref{ss:local-systems} for a discussion of the relevant local coefficient systems. 

\begin{proof}[Proof of Theorem~\ref{thm:TQFT-loop-intro}]
The existence of a graded open-closed TQFT structure on the pair $(\wh \H_*\Lambda,\wh H_*\Omega)$, and the fact that Poincar\'e duality intertwines it with the corresponding structure on cohomology, follows directly from Theorem~\ref{thm:open-closed-RFH} applied to $L=D^*_qM\subset V=D^*M$. 

The statement about the closed-open and open-closed maps is a consequence of Proposition~\ref{prop:topological_shriek}, which identifies the zipper and cozipper as extensions of the topological shriek and pushforward maps. That Proposition in turn uses Theorems~\ref{thm:Viterbo} and~\ref{thm:Viterbo-Lag}. 
\end{proof}

\begin{proof}[Proof of Theorem~\ref{thm:filtered-PD}]
The existence of the double filtrations on $\wh{H}_*\Lambda$ and $\wh{H}^*\Lambda$ and their compatibility with the products are immediate consequences of the existence of the action filtrations $SH_*^{(a,b)}(\p V)$, $SH^*_{(a,b)}(\p V)$ on symplectic (co)homology and their compatibility with the pair-of-pants products, applied to the unit disc cotangent bundle $V=D^*M$. Compatibility with Poincar\'e duality follows directly from the proof of Theorem~\ref{thm:coh-product}.
Compatibility with the splitting provided by Theorem~\ref{thm:splitting-intro}
follows from compatibility of the isomorphisms $SH_*(D^*M)\cong H_*\Lambda$ and $SH^*(D^*M)\cong H^*\Lambda$ with the action respectively length filtrations
proved in~\cite{CHO-MorseFloerGH}.
\end{proof}

\begin{proof}[Proof of Theorem~\ref{thm:splitting-intro}]
The Theorem is a consequence of combined results from~\cite{CO-cones,CHO-MorseFloerGH,CHO-reducedSH}. The existence of splittings is a particular case of results from~\cite[\S7]{CHO-reducedSH}, since cotangent bundles are Weinstein domains with essential skeleton. That the product on (based) Rabinowitz loop homology restricts to the the homological, resp. cohomological (based) loop product is proved in~\cite{CHO-MorseFloerGH}. Item (c) is a consequence of the cone description of (based) Rabinowitz loop homology, see~\cite[\S5]{CO-cones}.
\end{proof}

\section{Computations for odd-dimensional spheres}\label{sec:spheres}

In this section we illustrate the algebraic structures of this paper for loop spaces of odd-dimensional spheres $S^n$. For further details see~\cite{CHO-algebra,CHO-MorseFloerGH}.  
We use coefficients in a unital commutative ring $R$; this is possible even when it comes to coproducts because here the K\"unneth formula holds with arbitrary coefficients.

The computations that follow illustrate the fact that the structure of a graded Frobenius algebra is extremely rigid: for odd-dimensional spheres, the product determines uniquely the counital coproduct up to sign (and the sign is further determined by an additional extension property).
We need to distinguish the cases $n\geq 3$ and $n=1$.

\subsection{The case of odd $n\geq 3$} 

As a ring with respect to the loop product $\mu$, the degree shifted loop homology of $S^n$ (which equals reduced loop homology because $\chi(S^n)=0$ for $n$ odd) is given by
$$
   \H_*\Lambda S^n = \Lambda[a,u],\qquad |u|=n-1,\ |a|=-n.
$$
The loop coproduct was computed in~\cite{CHO-MorseFloerGH} to be
\begin{align*}
   \lambda(au^k) &= \sum_{i+j=k-1\atop i,j\geq 0} au^i\otimes au^j,\cr
   \lambda(u^k) &= \sum_{i+j=k-1\atop i,j\geq 0} (au^i\otimes u^j-u^i\otimes au^j).
\end{align*}
The coproduct $\lambda$ has odd degree $1-2n$, it is
coassociative and cocommutative, but it has no counit. 
Moreover, a direct computation shows that Sullivan's relation~\eqref{eq:Sullivan} holds in its original form. 

As a ring with respect to $\boldmu$, the degree shifted Rabinowitz loop homology of $S^n$ has been computed in Example~\ref{ex:spheres} to be
$$
   \wh{\H}_*\Lambda S^n = \Lambda[a,u,u^{-1}],\qquad |u|=n-1,\ |a|=-n.
$$
Here, as in Example~\ref{ex:spheres}, $\Lambda$ is the free graded commutative algebra and $u,a$ are the generators of the homology. 
Thus $\boldmu$ has degree $0$, it is associative, commutative, and unital with unit $\boldeta=\bold1$.  
(We write $\bold1$ in boldface to distinguish it from the identity map $1$.)

We claim that the coproduct $\boldlambda$ is given by 
\begin{align*}
   \boldlambda(au^k) &= \sum_{i+j=k-1} au^i\otimes au^j,\cr
   \boldlambda(u^k) &= \sum_{i+j=k-1} (au^i\otimes u^j-u^i\otimes au^j).
\end{align*}
Note that this coproduct has odd degree $1-2n$, it is coassociative, cocommuta\-tive, 
and counital with counit 
$$
  \boldeps(u^k)=0,\qquad \boldeps(au^k)=\begin{cases}
  1, & k=-1,\cr 0, & \text{else}. \end{cases}
$$
Moreover, it extends the coproduct $\lambda$ on $\H_*\Lambda$, meaning that upon restriction to $\H_*\Lambda$ and then truncation to $\H_*\Lambda\otimes \H_*\Lambda$ it is equal to $\lambda$. 

Assuming the claim, let us compute the pairing $\boldp$ and the copairing $\boldc$. The pairing $\boldp=-\boldeps\boldmu$ is given by
\begin{align*}
  \boldp(u^i\otimes u^j) = 0, \qquad \boldp (au^i\otimes&  au^j) = 0,\cr
  \boldp(au^i\otimes u^j) = \boldp(u^i\otimes au^j) &= \begin{cases}
  -1, & i+j=-1,\cr 0, & \text{else}. \end{cases}
\end{align*}
The copairing is determined by the relation $(1\otimes\boldp)(\boldc\otimes 1)=1$, giving 
$$
   \boldc = \sum_{i+j=-1} (au^i\otimes u^j-u^i\otimes au^j).
$$
Note that $\boldc= \boldlambda(\bold1)$ as expected. 

To prove the claim we use the fact that the coproduct on $\wh \H_*\Lambda$ is counital and extends the coproduct $\lambda$ on $\H_*\Lambda$ in the above sense. Counitality implies that the counit is nonzero, i.e. it acts nontrivially on $\wh{\H}_{1-2n}\Lambda S^n$. This last group is $1$-dimensional generated by $au^{-1}$, and since the coproduct is defined at chain level over $\Z$, the counit must be equal to $\pm 1$ on $a u^{-1}$ and zero elsewhere. In other words, the counit must be equal to $\pm\boldeps$, with $\boldeps$ defined above. The counit $\pm\boldeps$ determines a pairing equal to $\pm\boldp$, a copairing equal to $\pm\boldc$, and further, via the unital coFrobenius relation from Proposition~\ref{prop:coFrob-criterion}, a coproduct equal to $\pm\boldlambda$, with $\boldlambda$ as above. The restriction and truncation to $\H_*\Lambda$ of the coproduct equals $\pm\lambda$, so the correct coproduct is $\boldlambda$. This proves the claim. 

One can verify by direct computation that $(\wh{\H}_*\Lambda S^n,\boldmu,\boldlambda,\boldeta,\boldeps)$ is a graded Frobenius algebra. 

Denoting by $\{(u^k)^\vee, (au^k)^\vee\, : \, k\in \Z\}$ the basis dual to the basis $\{u^k,au^k \, : \, k\in\Z\}$, the Poincar\'e duality isomorphism is given by
$$
  -\vec\boldp(u^i) = (au^{-i-1})^\vee,\qquad
  -\vec\boldp(au^i) = (u^{-i-1})^\vee.
$$
By Example~\ref{ex:spheres}, 
the splitting from Theorem~\ref{thm:splitting-intro} can be described explicitly as follows (see also~\cite[\S7]{CHO-reducedSH})
\begin{equation*}
\xymatrix
@C=30pt
{  & & & u^{-2}\Lambda[a,u^{-1}] \ar[dl]_-i \ar[d]^-{\bar j}& \\
  0 \ar[r] & \Lambda[a,u] \ar[r]_\iota
  & \ar@/_/[l]_{\bar p} \Lambda[a,u,u^{-1}] \ar[r]_\pi
  & \ar@/_/[l]_{\bar i} u^{-1}\Lambda[a,u^{-1}] \ar[r] & 0 \\
}
\end{equation*}
where all maps are the obvious inclusions or projections and $\iota,i,\bar i,\bar j$ are algebra maps (nonunital except for $\iota$).
Note that the algebra structure on $\ol{\H}^{1-2n-*}\Lambda S^n$ in the lower right corner and the map $\bar i$ are uniquely determined by the rest of the diagram, so the splitting is canonical in this example. This confirms general results from~\cite[\S4]{CHO-MorseFloerGH}.

The preceding discussion carries over to the based loop space $\Om S^n$ by simply dropping the variable $a$, and one easily works out the maps $\boldi_!$ and $\boldi_*$ of the graded open-closed TQFT structure from Theorem~\ref{thm:TQFT-loop-intro} in this example, see~\cite{CHO-algebra}. 

The previous technique for computing $\boldlambda$ allows to recover the coproduct on reduced loop homology up to a global sign, without any preliminary knowledge of its value on any particular class. This is to be contrasted to~\cite{CHO-MorseFloerGH}, where the computation uses the value of the coproduct on a set of generators of the reduced loop homology algebra.

\subsection{The case $n=1$}  \label{sec:n=1}

As a ring with respect to the loop product $\mu$, the ordinary loop homology of $S^1$ (which equals the reduced homology) is given by
$$
   \H_*\Lambda S^1 = \Lambda[a,u,u^{-1}],\qquad |u|=0,\ |a|=-1.
$$
According to the description in~\cite{CHO-MorseFloerGH}, the loop coproduct in this case depends on the choice of a nowhere vanishing vector field on $S^1$, and we shall discuss this later in the section in relation to splittings. Our first goal is to compute the Frobenius algebra structure on $\wh \H_*\Lambda S^1$. To that effect, we record the fact that the canonically defined coproduct on $\H_*(\Lambda,\Lambda_0)$ is given by~\cite{CHO-MorseFloerGH}   
\begin{align*}
  \lambda(au^k) &= \begin{cases}
     \sum_{i=1}^{k-1}au^i\otimes au^{k-i} & k\geq 0, \\
     -\sum_{i=k+1}^{-1}au^i\otimes au^{k-i} & k<0,
  \end{cases} \cr 
  \lambda(u^k) &= \begin{cases}
     \sum_{i=1}^{k-1}(au^i\otimes u^{k-i} - u^i\otimes au^{k-i}) & k\geq 0, \\
     -\sum_{i=k+1}^{-1}(au^i\otimes u^{k-i} - u^i\otimes au^{k-i}) & k<0.
  \end{cases}  
\end{align*}
Rabinowitz loop homology $\wh\H_*\Lambda S^1$ is a direct sum as a graded Frobenius algebra (i.e., the product and coproduct both split)
$$
   \wh{\H}_*\Lambda S^1 = \Lambda[a_+,u_+,u_+^{-1}]]\oplus \Lambda[a_-,u_-^{-1},u_-]],\quad |u_\pm|=0,\ |a_\pm|=-1.
$$
That the product and the coproduct both split is a consequence of the confinement lemmas in~\cite{CO}. Here the two summands correspond to the two connected components of the unit sphere cotangent bundle $S^*S^1=\{+1,-1\}\times S^1$, and our convention is chosen such that a term $u_\pm^k$ has winding number $k\in\Z$ around the circle. 
Since both summands are identical under the transformation $u_+\to u_-^{-1}$,
let us describe one of them, dropping the subscript $\pm$. The graded $R$-module
$$
   \Lambda[a,u,u^{-1}]],\qquad |u|=0,\ |a|=-1
$$
consists of Laurent series
$$
   \sum_{i=-\infty}^N(\alpha_iu^i+\beta_iau^i),\qquad \alpha_i,\beta_i\in R,\ N\in\N
$$
with product given by the usual multiplication. Thus $\boldmu$ has degree $0$, it is associative, commutative, and unital with unit $\boldeta=\bold1$.  

We claim that the coproduct is 
\begin{align*}
   \boldlambda(au^k) &= \sum_{i+j=k} au^i\otimes au^j,\cr
   \boldlambda(u^k) &= \sum_{i+j=k} (au^i\otimes u^j-u^i\otimes au^j).
\end{align*}
Note the similarity to the coproduct on $\wh{\H}_*\Lambda S^n$ for $n\geq 3$ odd, with the difference that the condition $i+j=k-1$ in the sums now becomes $i+j=k$.\footnote{This difference can be understood geometrically as follows. In the case of the circle, the exponent of $u^k$ has a topological interpretation as a winding number around $S^1$, and the fact that the sum runs over $i+j=k$ expresses the fact that the winding number is preserved by the coproduct, which ``cuts the loops'' at self-intersections. In the case of spheres of odd dimension $n\ge 3$ endowed with the round metric, the exponent of $u^k$ is rather to be interpreted as a critical value, and the fact that the sum runs over $i+j=k-1$ expresses the fact that the critical value drops under the coproduct, since self-intersections typically happen along loops that are not great circles.} 
The coproduct $\boldlambda$ has odd degree $-1$, it is coassociative, cocommutative, and counital with counit 
$$
  \boldeps(u^k)=0,\qquad \boldeps(au^k)=\begin{cases}
  1, & k=0,\cr 0, & \text{else}. \end{cases}
$$
Also, it extends the coproduct $\lambda$ on $\H_*(\Lambda,\Lambda_0)$ in the previous sense. 

Towards proving the claim, we compute the copairing 
$$
   \boldc = \boldlambda(\bold1) = \sum_{i+j=0} (au^i\otimes u^j-u^i\otimes au^j),
$$
and the pairing $\boldp=-\boldeps\boldmu$ given by
\begin{align*}
  \boldp(u^i\otimes u^j) = 0,\qquad \boldp(au^i\otimes & au^j) = 0,\cr
  \boldp(au^i\otimes u^j) = \boldp(u^i\otimes au^j) &= \begin{cases}
  -1, & i+j=0,\cr 0, & \text{else}. \end{cases}
\end{align*}
The claim is a consequence of the following facts: (i) the coproduct on $\wh \H_*\Lambda$ is counital and extends the coproduct $\lambda$ on $\H_*(\Lambda,\Lambda_0)$, and (ii) the copairing is nondegenerate and the total winding number at its outputs is zero (this last fact is geometric and is specific to the circle). To prove the claim, denote $counit$, $coproduct$, $copairing$ the operations to be determined. We first show $counit=\pm\boldeps$, with $\boldeps$ defined as above: we apply the relation $(counit\otimes 1)coproduct=1$ to the unit, we note that (a) the unit is represented by loops with winding number zero, and (b) the copairing is given by $coproduct(unit)$ and is nondegenerate, and we deduce that $counit$ vanishes on all homology classes represented by loops with winding number $\neq 0$. Since $\wh\H_{-1}\Lambda S^1$ is one-dimensional in winding number zero with generator $a$, and $coproduct$ is defined over $\Z$, we find that the counit evaluates to $\pm 1$ on $a$ and zero elsewhere, i.e. $counit=\pm\boldeps$. 
 
From here we conclude as in the case of spheres of odd dimension $\ge 3$: the counit $\pm\boldeps$ determines a coproduct equal to $\pm\boldlambda$ with $\boldlambda$ as above, whose restriction and truncation to $\H_*(\Lambda,\Lambda_0)$ equals $\pm\lambda$. Therefore the correct coproduct is $\boldlambda$, which proves the claim.

One can verify by direct computation that $(\wh{\H}_*\Lambda S^1,\boldmu,\boldlambda,\boldeta,\boldeps)$ is a graded Frobenius algebra. 

Denoting by $\{(u^k)^\vee, (au^k)^\vee \, : \, k\in\Z\}$ the basis dual to the basis $\{u^k,au^k\, : \, k\in\Z\}$, 
the Poincar\'e duality isomorphism is given by
$$
  -\vec\boldp(u^i) = (au^{-i})^\vee,\qquad
  -\vec\boldp(au^i) = (u^{-i})^\vee.
$$

We now discuss the relationship with splittings and coproducts on reduced loop homology. 
According to~\cite{CHO-MorseFloerGH}, the loop coproduct on $\H_*\Lambda S^1$ depends on the choice of a nowhere vanishing vector field on $S^1$. 
Up to homotopy there are two such choices $v_\pm(x) = \pm 1$, giving rise to two coproducts 
\begin{align*}
  \lambda_+(au^k) &= \begin{cases}
     \sum_{i=0}^{k}au^i\otimes au^{k-i} & k\geq 0, \\
     -\sum_{i=k+1}^{-1}au^i\otimes au^{k-i} & k<0,
  \end{cases} \cr 
  \lambda_+(u^k) &= \begin{cases}
     \sum_{i=0}^{k}(au^i\otimes u^{k-i} - u^i\otimes au^{k-i}) & k\geq 0, \\
     -\sum_{i=k+1}^{-1}(au^i\otimes u^{k-i} - u^i\otimes au^{k-i}) & k<0,
  \end{cases}  
\end{align*}
\begin{align*}
  \lambda_-(au^k) &= \begin{cases}
     \sum_{i=1}^{k-1}au^i\otimes au^{k-i} & k > 0, \\
     -\sum_{i=k}^{0}au^i\otimes au^{k-i} & k\leq 0,
  \end{cases} \cr 
  \lambda_-(u^k) &= \begin{cases}
     \sum_{i=1}^{k-1}(au^i\otimes u^{k-i} - u^i\otimes au^{k-i}) & k> 0, \\
     -\sum_{i=k}^{0}(au^i\otimes u^{k-i} - u^i\otimes au^{k-i}) & k\leq 0.
  \end{cases}  
\end{align*}
The coproduct $\boldlambda_\pm$ has odd degree $-1$, it is coassociative and cocommutative, but it has no counit. 
One readily verifies that in this case Sullivan's relation holds in the generalized form~\eqref{eq:gen-Sullivan} with   
$$
   \lambda_\pm\eta = \lambda_\pm(\bold1) = \pm(a\otimes \bold1-\bold1\otimes a).
$$
The splitting from Theorem~\ref{thm:splitting-intro} reads
\begin{equation*}
\xymatrix
@C=30pt
{  & & {\begin{matrix} u_+^{-1}\Lambda[a_+,u_+^{-1}]] \\ \oplus \\ u_-\Lambda[a_-,u_-]] \end{matrix}} \ar[dl]_-i \ar[d]^-{\bar j}& \\
  \Lambda[a,u,u^{-1}] \ar[r]_\iota
  & \ar@/_/[l]_{\bar p}
  {\begin{matrix} \Lambda[a_+,u_+,u_+^{-1}]] \\ \oplus \\ \Lambda[a_-,u_-^{-1},u_-]] \end{matrix}}
  \ar[r]_\pi
  & \ar@/_/[l]_{\bar i} \H^{-1-*}\Lambda S^1
}
\end{equation*}
where $i$ is the obvious inclusion (which is a nonunital algebra map), $\pi$ is the obvious projection, and $\iota$ is the unital algebra map sending $u^k$ to $u_+^k+u_-^k$ and $au^k$ to $a_+u_+^k+a_-u_-^k$. 
This follows from the fact that the powers of $u$ and $u_\pm$ correspond to the winding number around the circle and are therefore preserved under all maps, and $i$ is the inclusion of the negative action part. The image of $\bar i$ must be a subalgebra complementary to the image of $\iota$, so it must be generated by the image of $i$ and two elements $\alpha_+a_++\alpha_-a_-$, $\beta_+\bold1_++\beta_-\bold1_-$ with $\alpha_+\neq\alpha_-$, $\beta_+\neq\beta_-$. A short computation shows that this is only possible if one of $\alpha_\pm$ is zero and one of $\beta_\pm$ is zero, so $\im\bar i$ must be the direct sum of $\im i$ and one of the four subspaces
$$
   {\rm span}\{\bold1_+,a_+\},\quad
   {\rm span}\{\bold1_+,a_-\},\quad
   {\rm span}\{\bold1_-,a_+\},\quad
   {\rm span}\{\bold1_-,a_-\}.
$$
So, algebraically, there are four possible splitting maps $\bar i$, each of which induces a product on $\H^{-1-*}\Lambda S^1$ extending the cohomology product on $\H^{-1-*}(\Lambda S^1,\Lambda_0S^1)$. Geometrically, only the first and last of the four possibilities are realized by choices of nowhere vanishing vector fields on $S^1$, corresponding to the loop coproducts $\lambda_\pm$ above. 

Again, the discussion carries over to the based loop space $\Om S^1$ by dropping the variable $a$ and one easily works out the graded open-closed TQFT structure, see~\cite{CHO-algebra}.

\appendix

\section{Grading and local systems} \label{sec:grading}

\subsection{Grading conventions} \label{sec:grading1}

  Rabinowitz loop homology $\wh{H}_*\Lambda$ is graded by the
  Conley--Zehnder indices of $1$-periodic orbits of Hamiltonians
  involved in the definition of symplectic homology. These indices are defined using the
  canonical trivializations of the tangent bundle of $T^*M$ along
  loops defined in~\cite{Abouzaid-cotangent}, with degrees shifted
  down by $1$ along loops which are orientation reversing (i.e., along
  which the pullback bundle $TM$ is nonorientable). With this grading
  the maps $\eps$ and $\iota$ in~\eqref{eq:les-loop-intro} are degree preserving.
  Similarly, based Rabinowitz loop homology $\wh{H}_*\Omega$ is graded by the Conley-Zehnder indices of Hamiltonian chords.  
We refer to~\cite[\S4.1]{Ekholm-Oancea} and references therein for more details.

\subsection{Local systems}\label{ss:local-systems}

All the results on loop space homology in this paper hold for any closed connected (not necessarily orientable) manifold $M$. For this, we use loop space homology with coefficients twisted by suitable local systems described in~\cite{Abouzaid-cotangent}, see also~\cite[Appendix A]{CHO-MorseFloerGH}. Denote $\ev:\Lambda\to M$ the evaluation map at the starting point of a loop and $\co$ the orientation local system on $M$. The loop space $\Lambda$ carries two canonical local systems: the {\em orientation local system} $\wt\co=\ev^*\co$ on components of orientation preserving 
loops (which is trivial iff $M$ is orientable), and the {\em spin local system} $\sigma$ (which is trivial iff the second Stiefel--Whitney class of $M$ vanishes). Then all results on closed loops in this paper hold in the following two situations (with corresponding twists in cohomology):
\begin{enum}
\item $H_*\Lambda$ is twisted by $\sigma\otimes\wt\co$ and $SH_*(D^*M)$ is untwisted;
\item $H_*\Lambda$ is twisted by $\wt\co$ and $SH_*(D^*M)$ is twisted by $\sigma^{-1}$.
\end{enum}
Similarly, the results on based loops hold in the following situations:
\begin{enum}
\item wrapped Floer homology is untwisted and based loop homology is twisted by $\sigma|_\Omega$, the restriction of the spin local system to $\Omega$;
\item wrapped Floer homology is twisted by $\sigma|_\Omega^{-1}$ and based loop homology is untwisted.
\end{enum}

\section{Proof of Proposition~\ref{prop:index_reverse}} \label{sec:index_reverse}

We keep the notation from \S\ref{sec:index}.
The proof uses the \emph{Long index} for paths $P:[0,1]\to \mathrm{Sp}(2N)$ which do not necessarily start at the identity, cf.~\cite[Definition~6.2.9]{Long-book}. We denote it 
$$
i_L(P)\in\Z. 
$$
The Long index is defined from the Bott-Long index by attaching to $P$ an arbitrary path $\xi$ starting at $\mathrm{Id}$ and ending at $P(0)$ and setting $i_L(P)=i(\xi\# P)-i(\xi)$. Among the properties of the Long index we will use additivity under concatenations, vanishing on paths for which the nullity is constant along the path, and homotopy invariance with fixed endpoints. Taking into account that $i(P\equiv \mathrm{Id})=-N$, we see that the Long index and the Bott-Long index for paths $P$ which start at the identity are related by the equation 
$$
i_L(P)=i(P)+N.
$$ 
In particular, equation~\eqref{eq:index_reverse} for paths starting at the identity is equivalent to
\begin{equation}\label{eq:index_reverse_Long}
i_L(\ol{P})=-i_L(P)-\nu(P)+2N.
\end{equation}
To prove~\eqref{eq:index_reverse_Long} we use a method from~\cite[Lemma~2.3]{Cieliebak-Frauenfelder-Oancea}. Denote $P_-(t)=P(1-t)$, so that $\ol{P}=P_-P(1)^{-1}$. We start from the relation $P_-P_-^{-1}\equiv \mathrm{Id}$. Given two paths $Q,R$ starting at the identity, their product $QR:t\mapsto Q(t)R(t)$ is homotopic to the concatenation $Q\#Q(1)R$. As a particular case, the path $P_-P_-^{-1}=\ol{P}P(1)P_-^{-1}$ is homotopic to the concatenation of $P_-P(1)^{-1}=\ol{P}$ and $\ol{P}(1)P(1)P_-^{-1}=P_-^{-1}$. From $i_L(\mathrm{Id})=0$ we therefore obtain 
$$
   i_L(\ol{P})=-i_L(P_-^{-1}). 
$$
Since the concatenation $P\#P_-$ is homotopic with fixed endpoints to $\mathrm{Id}$, we get $i_L(P_-)=-i_L(P)$. On the other hand, we prove below that  
\begin{equation} \label{eq:index_Long_inverse}
i_L(R^{-1})=-i_L(R)-\nu(R(1))+\nu(R(0))  
\end{equation}
for any path $R:[0,1]\to\mathrm{Sp}(2N)$. Then~\eqref{eq:index_reverse_Long} follows: 
\begin{eqnarray*}
i_L(\ol{P}) & = & -i_L(P_-^{-1})\\
& = & -(-i_L(P_-)-\nu(P_-(1))+\nu(P_-(0))) \\
& = & -(i_L(P)-\nu(P(0))+\nu(P(1)))\\
& = & -i_L(P)-\nu(P)+2N. 
\end{eqnarray*}

Equation~\eqref{eq:index_Long_inverse} follows directly from the definition of the Long index $i_L$ and the equation  
\begin{equation} \label{eq:iR-1}
i(R^{-1})=-i(R)-\nu(R)
\end{equation}
for paths $R:[0,1]\to \mathrm{Sp}(2N)$ with $R(0)=\mathrm{Id}$. 

We are thus left to prove~\eqref{eq:iR-1}. 
Denote $\cP^*(2N)=\{R:[0,1]\to \mathrm{Sp}(2N)\, : \, R(0)=\mathrm{Id}, \quad \det (R(1)-\mathrm{Id})\neq 0\}$. We first observe that~\eqref{eq:iR-1} holds for paths $R\in\cP^*(2N)$: the Bott-Long index $i(R)$ is equal to the Conley-Zehnder index on $\cP^*(2N)$~\cite[Definition~5.2.7]{Long-book}, and the definition of the Conley-Zehnder index in terms of the canonical extension $\rho:\mathrm{Sp}(2N)\to U(1)$ of the determinant function $\det:U(N)\to U(1)$ from~\cite[Theorem~3.1]{SZ92} implies the equality $i(R^{-1})=-i(R)$ for all paths $R\in\cP^*(2N)$ since $\rho(A^{-1})=\ol{\rho(A)}$ for every $A\in\mathrm{Sp}(2N)$. Since $\nu(R)=0$ for any path $R\in\cP^*(2N)$, this proves~\eqref{eq:iR-1}.

To prove~\eqref{eq:iR-1} for arbitrary paths we use the following minimizing characterization of the Bott-Long index, which combines Theorem~5.4.1, Definition~5.4.2, Corollary~6.1.9 and Definition~6.1.10 from~\cite{Long-book}: 
\begin{align*}
i(R) & =\sup_{U\in\cN(R)}\inf\{i(B)\, : \, B\in U\cap \cP^*(2N)\} \\
& = \inf_{U\in\cN(R)}\sup\{i(B)\, : \, B\in U\cap \cP^*(2N)\} - \nu(R).
\end{align*}
Here $\cN(R)$ is the set of neighbourhoods of $R$ in the space of paths $[0,1]\to \mathrm{Sp}(2N)$ starting at $\mathrm{Id}$. We obtain
\begin{align*}
i(R^{-1}) & = \sup_{U\in\cN(R^{-1})}\inf\{i(B)\, : \, B\in U\cap \cP^*(2N)\} \\
& = \sup_{U\in\cN(R)}\inf\{i(B^{-1})\, : \, B\in U\cap \cP^*(2N)\} \\
& = \sup_{U\in\cN(R)}\inf\{-i(B)\, : \, B\in U\cap \cP^*(2N)\} \\
& = - \inf_{U\in\cN(R)}\sup\{i(B)\, : \, B\in U\cap \cP^*(2N)\} \\
& = -i(R)-\nu(R).
\end{align*}
The third equality makes use of~\eqref{eq:iR-1} for paths in $\cP^*(2N)$.

\section{Proof of Theorem~\ref{thm:level-potency-Hingston}}\label{app:level-potency}

The proof of Theorem~\ref{thm:level-potency-Hingston} uses the following lemma. Given a closed geodesic $c$ of length $L$, index $\lambda$, and nullity $\nu$, we denote 
$$
H_*(c) = \lim_{U\supset c} H_*(\Lambda \cap U,\Lambda ^{<L}\cap U),\
H^*(c) = \lim_{U\supset c} H^{\ast }(\Lambda \cap U,\Lambda ^{<L}\cap U),
$$
where $U$ is an open set in $\Lambda$ containing the point $c\in\Lambda$.

\begin{lemma}\label{lem:level-potency}
The condition $H^\lambda(Sc)\neq 0$ in Theorem~\ref{thm:level-potency-Hingston}(a3) is equivalent to $H_\lambda(c)\neq 0$. 
The condition $H_{\lambda+\nu+1}(Sc)\neq 0$ in Theorem~\ref{thm:level-potency-Hingston}(b3) is equivalent to $H_{\lambda+\nu}(c)\neq 0$. 
\end{lemma}

\begin{proof}
The groups $H_*(c)$ and $H^*(c)$ have the following properties: (i) They are supported in degrees $\{\lambda,\dots,\lambda+\nu\}$~\cite{Gromoll-Meyer-Topology}. (ii) They are isomorphic with rational coefficients to the $S^1$-equivariant (co)homology groups
$H_*(c)\simeq H_*^{S^1}(Sc)$, $H^*(c)\simeq H^*_{S^1}(Sc)$. 
This holds because $c$ has finite isotropy and finite groups are $\Q$-acyclic~\cite[Prop.~6.1.10]{Weibel}.

To prove the first equivalence, consider the fragment of the Gysin sequence in cohomology $H^{\lambda-2}_{S^1}(Sc)\to H^\lambda_{S^1}(Sc)\to H^\lambda(Sc)\to H^{\lambda-1}_{S^1}(Sc)$. 
By (i) and (ii) the first and fourth term vanish, so that $H^\lambda(c)\simeq H^\lambda_{S^1}(Sc)\simeq H^\lambda(Sc)$. On the other hand, $H^\lambda(c)\neq 0$ if and only if $H_\lambda(c)\neq 0$. 

To prove the second equivalence, let $k=\lambda+\nu$ and consider the fragment $H_{k+2}^{S^1}(Sc)\to H_k^{S^1}(Sc)\to H_{k+1}(Sc)\to H_{k+1}^{S^1}(Sc)$
of the Gysin sequence in homology.  
By (i) and (ii) the first and fourth term vanish, so that $H_k(c)\simeq H_k^{S^1}(Sc)\simeq H_{k+1}(Sc)$. 
\end{proof}

\begin{proof}[Proof of Theorem~\ref{thm:level-potency-Hingston}]
We assume without loss of generality that all closed geodesics are isolated (otherwise there are infinitely many of them), and prove 
$$
(\mathrm{a1})\Rightarrow (\mathrm{a2}) \Rightarrow (\mathrm{a3})
\qquad 
\mbox{and} 
\qquad    
(\mathrm{b1})\Rightarrow (\mathrm{b2}) \Rightarrow (\mathrm{b3}).
$$
That (a3) implies the existence of infinitely many closed geodesics is the main theorem in~\cite{Hingston97}, in view of the first part of Lemma~\ref{lem:level-potency}. That (b3) implies the existence of infinitely many closed geodesics is~\cite[Proposition~1]{Hingston93} in view of the second part of Lemma~\ref{lem:level-potency}. 

(a1) $\Rightarrow$ (a2) is~\cite[Lemma~11.2]{Goresky-Hingston}, while (b1) $\Rightarrow$ (b2) is~\cite[Lemma~7.2]{Goresky-Hingston}.

We prove (a2) $\Rightarrow$ (a3). Recall that $c$ is an isolated closed geodesic of index $\lambda$, $Sc$ denotes its saturation with respect to the $S^1$-action, and $x\in H^*(Sc)$ is a class such that $x^m\in H^*(Sc^m)$ is nonzero for all $m\ge 1$. Here $x^m=x \oast \dots \oast x$ ($m$ times) is defined in terms of the loop product in local cohomology, which has degree $n-1$. Let $\deg(x)=k$, so that $\deg(x^m)=mk + (m-1)(n-1)$. 
We first observe that $\lambda\le k$ because $H^*(Sc)$ is supported in degrees $\{\lambda,\dots,\lambda+\nu+1\}$. Next we observe that $\ol{\ind}(c)=\lim_m \ind(c^m)/m$ (the \emph{mean index} of $c$) equals 
$$
\ol{\ind}(c)=k + n-1. 
$$
This follows from $\lim_m \ind(c^m)/m=\lim_m \deg(x^m)/m$, which equals $k+n-1$. To prove the equality between the limits we note that $|\deg(x^m)-\ind(c^m)|\le 2n-1$, by the support property of $H^*(Sc^m)$. 

The Bott index iteration formula~\cite[Theorem~A]{Bott56} gives that $\ind(c^m)=\sum_{z^m=1}\Lambda(z)$, where $\Lambda:S^1\to \N$ is the \emph{Bott index function}. The latter is lower semi-continuous, it can have discontinuities only at eigenvalues of the Poincar\'e return map lying on the circle~\cite[Theorem~C]{Bott56}, and the total jump at an eigenvalue, i.e. the sum of the left and right jumps, is bounded from above by the multiplicity of the eigenvalue as a root of the characteristic polynomial (this is a consequence of the characterization of the jumps in terms of the Krein type of the eigenvalue~\cite[Corollary~I.5.15]{Ekeland-book}, \cite[Theorem~9.1.7]{Long-book}, see also \cite[\S6.3]{Goresky-Hingston}). These properties imply that: (i) 
$\Lambda\le \lambda + n-1$ (because $\Lambda(1)=\lambda$ and the total multiplicity of the eigenvalues is $2n-2$), hence also $\Lambda\le k+n-1$; (ii) $\ol{\ind}(c)=\frac 1{2\pi}\int_0^{2\pi}\Lambda(e^{i\theta})d\theta$ (this is~\cite[Corollary~1]{Bott56}). Since $\ol{\ind}(c)=k+n-1$, we infer that the Bott function is constant equal to $k+n-1$ except possibly at the eigenvalues of the Poincar\'e return map. This implies $\lambda=k$, the eigenvalue 1 has the maximal multiplicity $2n-2$, and the Bott function is equal to $\lambda$ at $1$ and is constant equal to $\lambda+n-1$ elsewhere on the circle.
This gives $\ind(c^m)=m\lambda + (m-1)(n-1)$. 

We prove (b2) $\Rightarrow$ (b3). Let $\deg(X)=k$, so that $\deg(X^m)=mk - (m-1)n$. We first observe that $\lambda+\nu+1\ge k$ because the homology $H_*(Sc)$ is supported in degrees $\{\lambda,\dots,\lambda+\nu+1\}$. Next we observe that the mean index of $c$ is 
$$
\ol{\ind}(c)=k-n.
$$ 
This follows from $\lim_m \ind(c^m)/m=\lim_m \deg(X^m)/m$, which holds because $\deg(X^m)$ stays at uniformly bounded distance from $\ind(c^m)$. We consider the Bott iteration formula $\ind(c^m)+\nu(c^m)=\sum_{z^m=1}(\Lambda(z)+N(z))$, with $\Lambda$ as above and $N:S^1\to\N$ the \emph{Bott nullity function}. The latter is zero away from the eigenvalues of the Poincar\'e return map, and is equal to the nullity at each eigenvalue. The properties of $\Lambda$ listed previously, together with the fact that the left and right jumps at an eigenvalue are bounded from above by the nullity of the eigenvalue~\cite[Theorem~C]{Bott56}, imply that $\Lambda+N$ is upper semi-continuous, it can have discontinuities only at eigenvalues of the Poincar\'e return map lying on the circle, and the total jump at an eigenvalue is bounded from above by the multiplicity of the eigenvalue. Arguing as in the proof of (a2) $\Rightarrow$ (a3), from $\Lambda(1)+N(1)=\lambda+\nu\ge k-1$ we infer that $\lim_m(\ind(c^m)+\nu(c^m))/m=k-n$ is possible only if $\lambda+\nu=k-1$, the eigenvalue $1$ has the maximal multiplicity $2n-2$, and the Bott function $\Lambda+N$ is equal to $\lambda+\nu$ at $1$ and is constant equal to $\lambda+\nu - (n-1)$ elsewhere on the circle.
This gives $\ind(c^m)+\nu(c^m)=m(\lambda+\nu) - (m-1)(n-1)$. 
\end{proof}

Observe that the proof of (b2) $\Rightarrow$ (b3) is analogous to that of (a2) $\Rightarrow$ (a3), using $-(\Lambda+N)$ instead of $\Lambda$ and noting that $-(\Lambda+N)$ has exactly the same properties as $\Lambda$.

\bibliographystyle{abbrv}
\bibliography{000_SHpair}

\end{document}